\documentclass{amsart}
\usepackage{bbm}
\usepackage{bbding}
\usepackage[mathcal]{euscript}
\usepackage{graphicx}
\usepackage{amsthm}
\usepackage{calc}
\usepackage{mathrsfs,dsfont}
\usepackage{CJK,fancyhdr,amscd}
\usepackage{anysize} %control the page
\usepackage{amsmath,amssymb,amsfonts}
\usepackage{epsfig,enumerate}
\usepackage{latexsym}
\usepackage{indentfirst, latexsym}
\usepackage{graphics}
\usepackage[all,poly,knot]{xy}
\usepackage[colorlinks,linkcolor=blue,anchorcolor=blue,citecolor=blue,pagebackref]{hyperref}
\usepackage{geometry}
\numberwithin{equation}{section}
\newtheorem{theorem} {Theorem} [section]
\newtheorem{proposition}[theorem]{Proposition}
\newtheorem{corollary}  [theorem]     {Corollary}
\newtheorem{lemma}  [theorem]     {Lemma}
\newtheorem{question}  [theorem]     {Question}

\newtheorem{remark}  [theorem]     {Remark}
\newtheorem{claim}  [theorem]     {Claim}

\theoremstyle{definition}
\newtheorem{definition}  [theorem]     {Definition}

\newtheorem{example}  [theorem]     {Example}

\renewcommand*\backref[1]{}
\renewcommand*\backrefalt[4]{ \ifcase #1 \or (cited on page #2) \else (cited on pages #2) \fi}%(no citations)

\begin{document}

\title{On balanced Hermitian threefolds with parallel Bismut torsion}

\begin{abstract}
 We continue our study on Hermitian manifolds that are {\em Bismut torsion parallel,} or {\em BTP} for brevity, which means that the Bismut connection has parallel torsion tensor. For $n\geq 3$, BTP metrics can be balanced (and non-K\"ahler). In this paper, we give a detailed description to characterize all compact, balanced BTP threefolds.
\end{abstract}

\author{Quanting Zhao}
\address{Quanting Zhao. School of Mathematics and Statistics, and Hubei Key Laboratory of Mathematical Sciences, Central China Normal University, P.O. Box 71010, Wuhan 430079, P. R. China.} \email{zhaoquanting@126.com;zhaoquanting@mail.ccnu.edu.cn}
\thanks{Zhao is partially supported by NSFC with the grant No. 12171180 and 12371079. The corresponding author Zheng is partially supported by NSFC grants 12471039 and 12141101,  by Chongqing Normal University grant 24XLB026, and by the 111 Project D21024.}

\author{Fangyang Zheng}
\address{Fangyang Zheng. School of Mathematical Sciences, Chongqing Normal University, Chongqing 401331, China}
\email{20190045@cqnu.edu.cn} \thanks{}

%\date{\today}


\maketitle

\tableofcontents

\section{Introduction and statement of results}

A Hermitian manifold $(M^n,g)$ is said to be {\em Bismut torsion parallel} (or {\em BTP} for brevity), if its Bismut connection has parallel torsion tensor. Unless mentioned otherwise, we will always assume that $M^n$ is compact and $g$ is non-K\"ahler. Such manifolds form a relatively rich and highly interesting class of special Hermitian manifolds. It contains all Bismut K\"ahler-like (or BKL) manifolds and all Vaisman manifolds as special examples. In complex dimension $2$, BTP $=$ BKL $=$ Vaisman by \cite{ZhaoZ19Str}, and compact non-K\"ahler Vaisman surfaces are classified by Belgun in \cite{Belgun}. In complex dimensions $3$ or higher, BKL and Vaisman manifolds form proper, disjoint subsets of the set of non-balanced BTP manifolds as shown in \cite{YZZ}. Also, starting in complex dimension $3$, there are examples of balanced (and non-K\"ahler) BTP manifolds, which tend to form a much smaller set compared with the non-balanced ones, as illustrated by complex nilmanifolds with BTP metrics.

In \cite{ZhaoZ24}, we have studied the general properties of BTP manifolds and analyzed the structure of non-balanced BTP manifolds. In particular, the BTP condition, which means the parallelness of the Bismut torsion, is equivalent to some conditions involving the Bismut curvature tensor alone, by \cite[Theorem 1.1]{ZhaoZ24}. This is a rather distinctive property about the Bismut connection $\nabla^b$, and one certainly does not expect the same to hold for any other metric connection on $M^n$,   as torsion and curvature are independent geometric invariants in general. In \cite[Proposition 1.7]{ZhaoZ24}, we also established the existence of a special kind of local unitary frames (called {\em admissible frames}) on any non-balanced BTP manifold, under which the Chern torsion components are particularly simple. This frame makes the analysis on non-balanced BTP manifolds more accessible, so we get a structure theorem of such manifolds of complex dimension $3$ in \cite[Theorem 1.16]{ZhaoZ24}.

As mentioned before, in complex dimension $3$ or higher there are examples of compact balanced (but non-K\"ahler) BTP manifolds, although such manifolds tend to be fewer than non-balanced BTP ones. To justify this statement,  note that Proposition 1.10 of \cite{ZhaoZ24} classified BTP metrics among all complex nilmanifolds with nilpotent $J$ in the sense of Cordero-Fern\'andez-Gray-Ugarte \cite{CFGU}. It turns out that the majority of these BTP nilmanifolds are non-balanced, while only a small portion of them are balanced. For instance, when the complex dimension is $3$, only one of them (up to scaling of the metric) is balanced, while the non-balanced ones form a two-parameter family.

For balanced BTP manifolds, we do not have the admissible frames to help us, which is why we are only able to deal with the $3$-dimensional situation in this paper. For $n=3$, Zhou and the second named author observed in \cite{ZhouZ} that any balanced Hermitian threefold $(M^3,g)$ always admits a special type of local unitary frames (called {\em special frames}) where the Chern torsion take a relatively simple form. This technical tool enabled us to get the analysis started and  to obtain characterization on all compact balanced BTP threefolds.

Before stating the main result, let us recall that the {\em $B$ tensor} of a Hermitian manifold $(M^n,g)$ is the global non-negative $(1,1)$-tensor on $M^n$ defined by $B_{i\bar{j}} = \sum_{r,s=1}^n T^j_{rs} \overline{T^i_{rs} }$ under any local unitary frame, where $T^j_{ik}$ are the torsion components of the Chern connection under the frame. When $g$ is BTP, we have $\nabla^bT=0$ hence $\nabla^bB=0$, so the eigenvalues of $B$ are all global nonnegative constants. In particular, the rank of $B$ is constant on $M^n$. The following is the main result of this paper:

\begin{theorem} \label{thm1.1}
Let $(M^3,g)$ be a compact, balanced BTP threefold, with $g$ non-K\"ahler. Denote by $r$ the rank of the $B$ tensor.
\begin{enumerate}
\item If $r=3$, then $g$ is Chern flat and the universal cover of $(M^3,g)$ is the complex simple Lie group $SL(2,{\mathbb C})$ equipped with (a constant multiple of) the standard metric (see Example \ref{exampleSL2C}). Furthermore, the canonical line bundle of $M^3$ is always holomorphically trivial.
\item If $r=1$, then $(M^3,g)$ is the Wallach threefold (see \S \ref{FANO}), namely, as a complex manifold $M^3$ is the flag threefold ${\mathbb P}(T_{{\mathbb P}^2})$, the projectivization of the holomorphic tangent bundle of ${\mathbb P}^2$, and $g$ is the invariant Hermitian metric given by the Killing form.
\item If $r=2$, then $(M^3,g)$ is said to be a balanced BTP threefold {\bf of middle type}, in this case, the kernel of $B$ is a holomorphic line bundle $L$ on $M^3$ satisfying $L^{\otimes 2}\cong {\mathcal O}_M$, and $(M^3,g)$ will be called {\bf primary} or {\bf secondary} depending on whether $L\cong {\mathcal O}_M$ (holomorphically trivial) or not.
    \begin{enumerate}
    \item The secondary case. $L^{\otimes 2}\cong {\mathcal O}_M$ implies that $M$ has a double cover which is primary.
    \item The primary case. $L\cong {\mathcal O}_M$ holds if and only if the (global) holonomy group $\mbox{Hol}^{\,\,b}(M)$ of the Bismut connection is abelian. In fact, $\mbox{Hol}^{\,\,b}(M) = U(1)\!\times\!U(1)\!\times\!1$ when $M$ is primary and  $\mbox{Hol}^{\,\,b}(M) =G$ when $M$ is secondary, where $G$ is the ${\mathbb Z}_2$-extension of $U(1)\!\times\!U(1)\!\times\!1$ given in Proposition \ref{propsecondary} and $G$ is not abelian.
    \item Let $(M^3,g,J)$ be primary. Then there exists another complex structure $I$ on $M^3$ compatible with $g$, so that the Hermitian threefold $(M^3,g,I)$ is Vasiman, and the Bismut connection of the two Hermitian threefolds coincide. $(M^3,g,I)$ is called a {\bf Vaisman companion} of $(M^3,g,J)$.
    \end{enumerate}
\end{enumerate}
\end{theorem}

In other words, a compact balanced BTP threefold is either Chern flat which is a compact quotient of the complex simple Lie group $SL(2,{\mathbb C})$ equipped with the standard metric, or the Wallach threefold, or of middle type. Up to a double cover if necessary, any middle type one is always primary and admits Vaisman companions, which are Vasiman threefolds sharing the same Bismut connection with the original primary threefold, thus having abelian Bismut holonomy group. See \cite{N} for some very recent results on holonomy groups of Hermitian manifolds.

It follows directly from the proof of the main theorem that the first two cases above are actually BAS, which stands for \emph{Bismut Ambrose Singer}. This notion was introduced in \cite{NZ1, NZ} and it means that the Bismut connection $\nabla^b$ has parallel torsion and curvature, namely, $\nabla^b T^b =0$ and $\nabla^b R^b=0$. A classic theorem of Ambrose and Singer \cite{AS} (and its Hermitian version by Sekigawa \cite{Sekigawa})
says if a complete Riemannian (Hermitian) manifold admits a metric connection (Hermitian connection) which has parallel torsion and curvature, then the manifold is locally homogeneous, that is, its universal cover is a homogeneous Riemannian (Hermitian) manifold. In particular, BAS manifolds are always locally homogeneous. While, for the third case, namely, the middle type, it turns out that the entire Bismut (or Chern or Riemannian) curvature tensor is determined by two real-valued functions, $s$ and $\sigma_2$, which are the first and second elementary symmetric functions of the eigenvalues of the Bismut Ricci form. Of course $s$ is just the Bismut (Chern) scalar curvature.
As observed in Proposition 1.7 of \cite{PodestaZ}, this case will be BAS if and only if both $s$ and $\sigma_2$ are constants.
See Lemma \ref{lemmaPodestaZ} for more details.

In Proposition \ref{LieHermitian} and the content after it, we classified $6$-dimensional unimodular Lie algebras which admits a  Hermitian structure
that is balanced BTP of the middle type. We discussed the isomorphic classes of these Lie algebras and the existence of uniform lattices in the corresponding simply-connected Lie groups. They form two families $A_{s,t}$ and $B_{z,t}$ for $s,t\in {\mathbb R}$ and $z\in {\mathbb C}$,
satisfying the followings
\begin{enumerate}
\item  $A_{0,0}=B_{0,0}=N^3$ determines the only balanced BTP nilmanifold in dimension $3$.
\item The overlapping part is $A_{0,t}=B_{0,t}$, $t\in {\mathbb R}$.
\item Both families are CYT.
\item $A_{s,t}$ and $B_{z,t}$ are $3$-step solvable whenever $(s,t)\neq (0,0)$ and $(z,t)\neq (0,0)$.
\item $A_{s,t}$ for $s+t\neq 0$ and $B_{z,t}$ for $(z,t)\neq (0,0)$ are not of Calabi-Yau type.
\end{enumerate}

Recall that a compact Hermitian manifold $(M^n,g)$ is said to be {\em CYT}, which stands for {\em Calabi-Yau with torsion}, if the first Bismut Ricci curvature vanishes. It is said to be {\em of Calabi-Yau type} if its canonical line bundle is holomorphically trivial. Another result of this paper is the following

\begin{theorem} \label{thm1.2}
Any compact balanced BTP threefold of middle type cannot admit a pluriclosed metric. Consequently, both Fino-Vezzoni Conjecture and Streets-Tian Conjecture hold for compact BTP threefolds.
\end{theorem}

Recall that Fino-Vezzoni Conjecture states that if a compact complex manifold admits a balanced metric and a pluriclosed metric, then it must admit a K\"ahler metric. Streets-Tian Conjecture states that any compact Hermitian-symplectic manifold must admit a K\"ahler metric. Both conjectures are known to be true for non-balanced BTP threefolds by  \cite[Corollary 1.19]{ZhaoZ24}. For balanced BTP threefolds, by Theorem \ref{thm1.1} we just need to consider the middle type case, as the Fano case is already K\"ahlerian while the Chern flat case is either known (for Streets-Tian conjecture in any dimension by \cite{DLV}) or relatively easy to deduce as in the proof of Corollary \ref{cor6.8}.

\begin{remark}
The Bismut connection \cite{Bismut}, also known as Strominger connection \cite{Strominger} in the literature, is the unique Hermitian connection with totally skew-symmetric torsion. So BTP manifolds are special examples of the much broader class of Riemannian manifolds admitting a metric connection with skew-symmetric and parallel torsion. There have been extensive studies on the latter in the setting of Riemannian geometry or almost Hermitian geometry %by Friedrich, Agricola, Ferreira, Cleyton, Moroianu, Semmelmann, Schoemann and others. See for example
with an incomplete list \cite{AF04}, \cite{AF14}, \cite{AFF}, \cite{AFS}, \cite{AFK}, \cite{CMS} \cite{Sch} and the references therein. Here we carry out the investigation in the slightly more delicate Hermitian situation, where the emphasis is on complex structures, with the hope of understanding the Bismut geometry for complex manifolds.
\end{remark}

The paper is organized as follows.  In \S \ref{BS3D}, we will discuss the types of the $B$ tensor of a non-K\"ahler balanced {\em BTP} threefold, prove Proposition \ref{Btype}, and also show that the $\mathrm{rank}\,B=3$ case of Theorem \ref{thm1.1} leads to compact quotients of $SL(2,{\mathbb C})$. In \S \ref{FANO}, we will prove that the $\mathrm{rank}\,B=1$ case of Theorem \ref{thm1.1} leads to the Wallach threefold. In \S \ref{WCH3D}, we will carry out the detailed computation which shows that the Wallach threefold is indeed balanced and  BTP. In the Appendix, we will verify that the Wallach threefold has nonnegative bisectional curvature and positive holomorphic sectional curvature for its Chern connection, while its Riemannian connection has non-negative sectional curvature and constant Ricci curvature. In \S \ref{mddtype3D}, we will prove the $\mbox{rank}\,B=2$ case of Theorem \ref{thm1.1}, prove Theorem \ref{thm1.2}, and classify Hermitian unimodular Lie algebras of complex dimension $3$
which are balanced BTP of the middle type. Finally, in the last section, we will consider some generalization in higher dimensions, and prove Theorem \ref{thm6.3}.

\vspace{0.4cm}

\section{Balanced BTP threefolds}\label{BS3D}

We will follow the notations in \cite{ZhaoZ24} throughout this paper. Given a Hermitian manifold $(M^n,g)$, denote by $\nabla^c$, $\nabla^b$ the Chern and Bismut connection, and by $T$, $T^b$ the Chern and Bismut torsion. Under a frame $e$ of type $(1,0)$ tangent vectors, let $T^j_{ik}$ be the components of the Chern torsion, namely,
$$ T(e_i,e_k) := \nabla^c_{e_i}e_k - \nabla^c_{e_k}e_i - [e_i, e_k] = \sum_{j=1}^n T^j_{ik} e_j. $$
The BTP assumption on the metric $g$ means that $\nabla^bT^b=0$, which is equivalent to $\nabla^bT=0$.
This has been shown in \cite[the proof of Proposition 1]{ZhaoZ19Str}

Now let $(M^3,g)$ be a compact non-K\"ahler balanced BTP manifold. We will start with a technical observation in \cite[Proposition 2]{ZhouZ} which says that any balanced threefold always admits a particular type of unitary frame under which the Chern torsion takes a simple form. We include the proof here for readers' convenience. Let us denote by $B$  the $2$-tensor on the manifold  defined by $B_{i\overline{j}}=\sum_{r,s} T^j_{rs}\overline{T^i_{rs}}$ under any local unitary frame. It is  globally-defined and Hermitian symmetric.

\begin{lemma}[\cite{ZhouZ}] \label{lemmaA}
Let $(M^3,g)$ be a Hermitian threefold that is balanced. Then for any given $p\in M^3$, there exists a local unitary frame $e$ in a neighborhood of $p$ such that the $T^i_{ij}=0$ for $1\leq i,j\leq 3$ under $e$. Furthermore, if the rank of the $B$ tensor is constant in the neighborhood, then we can choose a local unitary frame $e$ so that $T^i_{ij}=0$ for $1\leq i,j\leq 3$ and
$$T^1_{23} \geq 0,\ T^2_{31} \geq 0,\ T^3_{12} \geq 0. $$
\end{lemma}

\begin{proof}
Given any $p\in M$, let $e$ be a unitary frame near $p$. Consider the $3\times 3$ matrix $A_{i\alpha} =T^{\alpha }_{jk}$ where $1\leq i,\alpha \leq 3$ and $(ijk)$ is a cyclic permutation of $(123)$. The metric $g$ is assumed to be balanced, which means $d(\omega^{n-1})=0$ where $\omega$ is the K\"ahler form of $g$, or  equivalently $\eta=0$ where $\eta$ is Gauduchon's torsion $1$-form defined by $\partial (\omega^{n-1}) = -\eta \wedge \omega^{n-1}$. Let $\varphi$ be the coframe dual to $e$, and write $\eta = \sum_i \eta_i \varphi_i$, then  $\eta_i=\sum_s T^s_{si}=0$ for each $i$. So we have, for $i \neq j$,
$$ A_{ij}=T^j_{jk}= - T^i_{ik} = T^i_{ki} = A_{ji}.$$
That is, the  $3\times 3$ matrix $A$ is symmetric. Therefore, there exists a unitary matrix $U$ of local smooth functions such that $\overline{U} A \,^t\!\overline{U}$ is diagonal. Now let $\tilde{e}$ be the new unitary local frame given by $\tilde{e}_i = \sum_s U_{is}e_s$. Then we have
\begin{eqnarray*}
 \widetilde{A}_{i\alpha} & = &  \widetilde{T}^{\alpha}_{jk} \ = \ \sum_{r,s,t} U_{jr}U_{ks} T^t_{rs} \overline{U}_{\alpha t} \ = \ \sum_{c,t} (U_{ja}U_{kb}-U_{jb}U_{ka}) T^t_{ab} \overline{U}_{\alpha t} \\
 & = & \sum_{c,t} (U_{ja}U_{kb}-U_{jb}U_{ka}) A_{ct} \overline{U}_{\alpha t} \ = \  \sum_{c,t} \det (U) \,(U^{-1})_{ci} A_{ct} \overline{U}_{\alpha t}   \\
 & = & \sum_{c,t} \det (U)\, \overline{U}_{ic}A_{ct} \overline{U}_{\alpha t} \ = \ \det (U)\, (\overline{U} A \,^t\!\overline{U})_{i\alpha}.
 \end{eqnarray*}
Here we assumed that $(ijk)$ and $(abc)$ are both cyclic permutations of $(123)$. Hence $\widetilde{A}$ is diagonal, which means that under the new unitary frame $\tilde{e}$ we have $\widetilde{T}^s_{st}=0$ for any $1\leq s,t\leq 3$, and the only possibly non-zero torsion components are $\widetilde{T}^i_{jk}=-\widetilde{T}^i_{kj}$ where $(ijk)$ is any cyclic permutations of $(123)$. Let us write $a_i=\widetilde{T}^i_{jk}$. The $B$ tensor under $\tilde{e}$ now takes the form
$$ B = 2 \begin{bmatrix} |a_1|^2 & 0 & 0 \\ 0 & |a_2|^2 & 0 \\ 0 & 0 & |a_3|^2 \end{bmatrix}\!\!.$$
Denote by $r$ the constant rank of $B$ in the neighborhood. At $p$, exactly $r$ of those $a_i$ are non-zero. Without loss of generality, we may assume that $a_1\cdots a_r|_p\neq 0$. Let $V$ be a possibly smaller neighborhood of $p$ in which $a_1\cdots a_r \neq 0$. Since the rank of $B$ is constantly $r$, within $V$ we must have $a_{r+1}=\cdots =a_3=0$.  Define smooth functions in $V$
$$ \rho_i = \frac{a_i}{|a_i|}, \quad  \forall \ 1\leq i\leq r \quad \mbox{and} \quad \rho_i=1 , \quad \forall \ r\!+\!1\leq i\leq 3.  $$
Consider a new unitary frame $e'$ in $V$ given by $e'_i =  (\overline{\rho}_j \overline{\rho}_k)^{\frac{1}{2}} \,\tilde{e}_i$ where $(ijk)$ is a cyclic permutation of $(123)$. Then we still have $T'^s_{st}=0$ for any $s$, $t$, while
$$ T'^i_{jk} = \widetilde{T}^i_{jk} \, (\rho_j\rho_k)^{\frac{1}{2}} (\overline{\rho}_k \overline{\rho}_i )^{\frac{1}{2}}  (\overline{\rho}_i \overline{\rho}_j )^{\frac{1}{2}} = a_i \, \overline{\rho}_i = |a_i|. $$
%Permute the elements of the frame $e'$ if necessary, we may assume that $a_1\geq a_2\geq a_3$,
Hence the frame $e'$ satisfies the requirement of the lemma.
\end{proof}

\begin{definition}
A local unitary frame $e$ on a given balanced Hermitian threefold $(M^3,g)$ is called a {\bf special frame,} if the Chern torsion components under $e$ satisfy
\begin{equation}\label{specialframe}
T^1_{1k}=T^2_{2k}=T^3_{3k}=0, \quad \forall \ 1\leq k\leq 3 \quad \mbox{and} \quad  T^1_{23} \geq 0,\ T^2_{31} \geq 0,\ T^3_{12} \geq 0.
\end{equation}
\end{definition}

Lemma \ref{lemmaA} simply says that special frames exist in a neighborhood  of any given point in a balanced threefold wherever the $B$ tensor has constant rank. Now let $(M^3,g)$ be a compact Hermitian threefold that is balanced and BTP. Then the condition $\nabla^bB=0$ guarantees that $B$ has constant rank over the entire $M$ as the eigenfunctions of $B$ are constants, so there always exist special frames everywhere. Let $e$ be a special frame and write
\begin{equation}
 a_1=T^1_{23}, \ \ \ a_2=T^2_{31}, \ \ \ a_3=T^3_{12}. \label{eq:a}
\end{equation}

\begin{lemma} \label{lemmaB}
Let $(M^3,g)$ be a Hermitian threefold that is balanced and BTP. Suppose that $e$ is a special frame. Then the local non-negative functions $a_1, a_2, a_3$ given by (\ref{eq:a}) are global constants and we may assume $a_1 \geq a_2 \geq a_3$ after a signed permutation of $e$.
Furthermore, it holds that
\begin{equation} \label{eq:aithetab}
\left\{ \begin{split}  a_1(\theta^b_{22}+\theta^b_{33} - \theta^b_{11} ) \ = \ 0,  \hspace{4.7cm} \\
 a_2(\theta^b_{11}+\theta^b_{33} - \theta^b_{22} ) \ = \ 0 , \hspace{4.7cm} \\
 a_3( \theta^b_{11}+\theta^b_{22} - \theta^b_{33} ) \ =  \ 0 , \hspace{4.7cm} \\
 a_1\theta^b_{12}+a_2\theta^b_{21} \ =\  a_1\theta^b_{13}+a_3\theta^b_{31} \ = \ a_2\theta^b_{23}+a_3\theta^b_{32} \ = \ 0,
 \end{split} \right.
\end{equation}
where $\theta^b$ denotes the matrix of the Bismut connection $\nabla^b$ under the frame $e$.
\end{lemma}

\begin{proof}
The $B$ tensor under a special frame $e$ takes the form
$$ B = 2 \begin{bmatrix} a_1^2 & 0 & 0 \\ 0 & a_2^2 & 0 \\ 0 & 0 & a_3^2 \end{bmatrix}\!\!.$$
By the {\em BTP} assumption, we have $\nabla^bB=0$, so the eigenfunctions of $B$ are all global constants on $M$,
which turn out to be $2a_1^2$, $2a_2^2$ and $2a_3^2$. Therefore each $a_i$ is a global constant. If these
 non-negative constants $a_1,a_2,a_3$ are not in non-increasing order, say for instance $a_3 \geq a_1 \geq a_2$, then by using the special frame $(e_3,e_1,e_2)$ instead, we get them into the non-increasing order. For another example, if $a_3 \geq a_2 \geq a_1$, then we may use the special frame $(-e_3,e_2,e_1)$ to achieve the requirement.

Denote by $\theta^b$ the matrix of the Bismut connection $\nabla^b$ under the frame $e$. Since $\nabla^bT=0$ and all $T^j_{ik}$ are constants, we have
\begin{equation}
0 = dT^j_{ik} =  \sum_r \big( T^j_{rk} \theta^b_{ir} + T^j_{ir} \theta^b_{kr} - T^r_{ik} \theta^b_{rj} \big).   \label{eq:6.4}
\end{equation}
Since the only possibly non-zero components of $T$ are $a_1$, $a_2$ and $a_3$, if we take $i$, $j$, $k$ all distinct in (\ref{eq:6.4}), we get the first three lines of (\ref{eq:aithetab}). Similarly, by taking $j=i \neq k$ in (\ref{eq:6.4}), we get the last line of (\ref{eq:aithetab}).
\end{proof}

Denote by $\varphi$ the unitary coframe dual to the special frame $e$, and by $\theta$, $\tau$ the matrix of connection and column vector of torsion under $e$ for the Chern connection $\nabla^c$. As the Chern torsion components $T^j_{ik}$ are defined by $T(e_i, e_k) = \sum_j T^j_{ik}e_j$, so we have $\tau_j =\frac{1}{2} \sum_{i,k} T^j_{ik}\,\varphi_i \wedge \varphi_k$ by the structure equations and Bianchi identities
set up in \cite[Section 2]{ZhaoZ24}.

Let $\gamma =\nabla^b -\nabla^c$ be the $(2,1)$-tensor introduced in \cite{YZ18Cur}, and for convenience we will also use the same letter to denote its matrix representation under $e$, that is, $\gamma = \theta^b - \theta^c$. On the Hermitian manifold $(M^n,g)$ and under any unitary frame, it holds that $\gamma_{ij}=\sum_k (T^j_{ik}\varphi_k - \overline{T^i_{jk}} \overline{\varphi}_k)$ as shown in \cite{YZ18Cur}.
These have also been given in \cite[Section 2]{ZhaoZ24}. Therefore in our case we have
\begin{equation} \label{eq:6.7}
\gamma = \begin{bmatrix} 0 & -\overline{\psi}_3  & \psi_2 \\ \psi_3 & 0 & -\overline{\psi}_1 \\ - \overline{\psi}_2 & \psi_1 &  0 \end{bmatrix}\!\!, \ \ \ \ \ \tau = \begin{bmatrix} a_1 \varphi_2\varphi_3 \\ a_2 \varphi_3\varphi_1 \\ a_3 \varphi_1\varphi_2  \end{bmatrix}\!\!,     \hspace{1cm} \mbox{where}
\end{equation}
\begin{equation} \label{eq:6.8}
\psi_1 = a_2\varphi_1+a_3\overline{\varphi}_1, \ \ \ \ \ \ \psi_2 = a_3\varphi_2+a_1\overline{\varphi}_2, \ \ \  \ \ \ \psi_3 = a_1\varphi_3+a_2\overline{\varphi}_3.
\end{equation}

\begin{proposition}\label{Btype}
Let $(M^3,g)$ be a non-K\"ahler balanced BTP threefold. Then under any special frame $e$, the $B$ tensor takes the form
\[B = \begin{bmatrix}  c & 0  & 0 \\ 0 & 0 &  0  \\ 0 & 0 & 0 \end{bmatrix}\!\!,\
\begin{bmatrix}  c & 0  & 0 \\ 0 & c &  0  \\ 0 & 0 & 0 \end{bmatrix}\  \text{or} \
\begin{bmatrix}  c & 0  & 0 \\ 0 & c &  0  \\ 0 & 0 & c \end{bmatrix}\!\!,\]
where $c>0$ is a constant.
\end{proposition}

\begin{proof}  Note that  since $M^3$ is not K\"ahler, we have $a_1>0$.
We shall divide the discussion into the following three cases: (1) $a_1>a_2>a_3$; (2) $a_1=a_2=a_3$; and (3) $a_1=a_2>a_3$ or $a_1>a_2=a_3$.

\vspace{0.15cm}

\noindent {\bf Case 1:} $a_1>a_2>a_3$.

\vspace{0.1cm}

In this case, $B$ has distinct eigenvalues. Since its eigenspaces are all $\nabla^b$-parallel, we know that the matrix $\theta^b$ is diagonal.
If $a_3>0$, by the first three lines of (\ref{eq:aithetab}), we get $\theta^b_{ii}=0$ for each $i$, hence $\theta^b=0$. This means that $M^3$ is Bismut flat. Such a manifold cannot be balanced unless it is K\"ahler, contradicting to our assumption that $M^3$ is balanced and non-K\"ahler, so we must have $a_3=0$, which implies $\psi_1=a_2\varphi_1$ and $\psi_2=a_1\overline{\varphi}_2$. Since $a_1, a_2>0$, by the first two lines of (\ref{eq:aithetab}), we get
\begin{equation*}
 \theta^b_{33}=0, \ \ \  \theta^b_{11}=\theta^b_{22} = \alpha ,
\end{equation*}
and
$$ \theta = \theta^b-\gamma =  \begin{bmatrix} \alpha  & \overline{\psi}_3 & -a_1\overline{\varphi}_2 \\ -\psi_3 & \alpha & a_2 \overline{\varphi}_1 \\ a_1\varphi_2 & -a_2\varphi_1 & 0 \end{bmatrix}\!\!. $$
Then the structure equation for the Chern connection yields
\begin{equation}
 d\varphi  =  -\,^t\!\theta \wedge \varphi + \tau = \begin{bmatrix} -\alpha \varphi_1 + \psi_3 \varphi_2 \\ - \alpha \varphi_2 - \overline{\psi}_3 \varphi_1  \\ a_2 \varphi_2 \overline{\varphi}_1 - a_1 \varphi_1 \overline{\varphi}_2   \end{bmatrix} \!\!. \label{eq:6.9}
 \end{equation}
By the first two equations of (\ref{eq:6.9}), we get
\begin{eqnarray*}
 d(\varphi_2 \overline{\varphi}_1 ) &  = & d\varphi_2 \wedge \overline{\varphi}_1 - \varphi_2 \wedge d\overline{\varphi}_1  \\
 & = & (-\alpha \varphi_2 - \overline{\psi}_3 \varphi_1 )\,\overline{\varphi}_1 - \varphi_2 (-\overline{\alpha} \,\overline{\varphi}_1 +  \overline{\psi}_3 \overline{\varphi}_2 ) \\
 & = &  \overline{\psi}_3 \,( \varphi_2 \,\overline{\varphi}_2 - \varphi_1\,\overline{\varphi}_1),
\end{eqnarray*}
where $\alpha + \overline{\alpha} =0$ is used. Taking the complex conjugation, we get
$$  d(\varphi_1 \overline{\varphi}_2 ) = -  \overline{ d(\varphi_2 \overline{\varphi}_1 )} =  \psi_3 \, ( \varphi_2 \,\overline{\varphi}_2 - \varphi_1\,\overline{\varphi}_1) .$$
The exterior differentiation of the third equation of (\ref{eq:6.9}) implies
$$ 0 = d^2\varphi_3 = a_2 \,d (\varphi_2 \overline{\varphi}_1) - a_1 \,d(\varphi_1 \overline{\varphi}_2) = (a_2\overline{\psi}_3 - a_1\psi_3) \wedge (\varphi_2 \,\overline{\varphi}_2 - \varphi_1\,\overline{\varphi}_1) .$$
Note that
$$ a_2\overline{\psi}_3 - a_1\psi_3 = a_2(a_1 \overline{\varphi}_3 + a_2 \varphi_3) - a_1 (a_1 \varphi_3 + a_2 \overline{\varphi}_3 ) = (a_2^2-a_1^2) \varphi_3, $$
which yields a contradiction. This shows that the case of distinct $a_1, a_2$ and $a_3$ cannot occur.
\vspace{0.15cm}

\noindent {\bf Case 2:} $a_1=a_2=a_3$.

\vspace{0.1cm}
Let us denote by $a>0$ the common value of those $a_i$ in this case. Equalities in (\ref{eq:aithetab})  now imply that $\theta^b$ is skew-symmetric. Since $\psi_i=a(\varphi_i+\overline{\varphi}_i) = \overline{\psi}_i$ for each $i$, by (\ref{eq:6.7}) we know that $\gamma$ is skew-symmetric. So $\theta= \theta^b-\gamma$ is also skew-symmetric, and we have
\begin{equation} \label{eq:6.10}
 \theta = \begin{bmatrix} 0 & x  & y \\ -x & 0 & z \\ - y & -z &  0 \end{bmatrix}\!\!, \ \ \ \ \tau = a \begin{bmatrix} \varphi_2\varphi_3 \\ \varphi_3\varphi_1 \\ \varphi_1\varphi_2 \end{bmatrix}\!\!,
 \end{equation}
where $x$, $y$, $z$ are real $1$-forms. As a result, the Chern curvature matrix $\Theta =d\theta - \theta \wedge \theta $ is also skew-symmetric. The structure equation for the Chern connection gives us
\begin{equation*}
 d \varphi = - \,^t\!\theta \wedge \varphi + \tau =
\begin{bmatrix} x \varphi_2 + y \varphi_3 +a\varphi_2\varphi_3 \\ -x \varphi_1 + z \varphi_3 +a\varphi_3\varphi_1 \\  -y \varphi_1 -z \varphi_2 +a\varphi_1\varphi_2  \end{bmatrix}\!\!.
\end{equation*}
It follows that
\begin{eqnarray*}
d (\varphi_2\varphi_3) & = & (x\varphi_3-y\varphi_2)\varphi_1 \\
d (\varphi_3\varphi_1) & = & (x\varphi_3+ z\varphi_1)\varphi_2 \\
d (\varphi_2\varphi_3) & = & (-y\varphi_2+z\varphi_1)\varphi_3.
\end{eqnarray*}
On one hand, if we let $\xi = x\varphi_3 -y\varphi_2 + z\varphi_1$, then the above equations simply say that $d\tau = a \,\xi \wedge \varphi$. On the other hand, by (\ref{eq:6.10}) we get $\theta \tau = a\,\xi \wedge \varphi$. So by the first Bianchi identity
$$ d\tau = \,^t\!\Theta \varphi - \,^t\!\theta \tau, $$
we conclude that $ \,^t\!\Theta \varphi = 0$. This means that the Hermitian threefold $M^3$ is Chern K\"ahler-like, that is, the Chern curvature tensor $R^c$ obeys the K\"ahler symmetry $R^c_{i\overline{j}k\overline{\ell}} = R^c_{k\overline{j}i\overline{\ell}} $ for any $i,j,k,\ell$.

We claim that in this case the metric $g$ must be Chern flat, namely, $R^c=0$. If $\{ i,k\} \cap \{ j,\ell\} \neq \emptyset$, say for instance $1$ is contained in the intersection, then by the K\"ahler symmetry, $R^c_{i\overline{j}k\overline{\ell}}$ can be written as $R^c_{a\overline{b}1\overline{1}}$, which has to vanish since $\Theta_{11}=0$ as $\Theta$ is skew-symmetric. When $\{ i,k\} \cap \{ j,\ell\} = \emptyset$, what we need to show are the equalities $R^{c}_{i\overline{j}k\overline{j}}=0$ and $R^{c}_{i\overline{j}i\overline{j}}=0$ where $i,j,k$ are distinct, as the dimension is $3$. From the skew-symmetry of $\Theta$, it yields that
$$ R^c_{i\overline{j}k\overline{j}} = - R^c_{i\overline{j}j\overline{k}} = -R^c_{i\overline{k}j\overline{j}}=0,$$
$$ R^c_{i\overline{j}i\overline{j}} = - R^c_{i\overline{j}j\overline{i}} = -R^c_{i\overline{i}j\overline{j}}=0.$$
So $(M^3,g)$ is indeed Chern flat in this case. Under the special frame $e$, the tensor $B$ equals to
\[B =  \begin{bmatrix} c & 0 & 0 \\ 0 & c & 0 \\ 0 & 0 & c   \end{bmatrix}\!\!,\]
where $c=2a^2>0$, so $B$ has rank $3$.

\vspace{0.15cm}

\noindent {\bf Case 3:} $a_1=a_2>a_3\,$ or $\,a_1>a_2=a_3$.

\vspace{0.1cm}

In this case $B$ has two distinct eigenvalues. First we will rule out the possibility of $a_3>0$. Suppose that $a_3>0$,
namely, either $a_1=a_2>a_3>0$ or $a_1>a_2=a_3>0$, we will derive a contradiction.

Since the argument for these two situations are analogous, let us focus on the case $a_1=a_2>a_3>0$. Write $a_1=a_2=a$ for simplicity.  Since the eigenspaces of $B$ are $\nabla^b$-parallel, it follows that $\theta^b_{13}=\theta^b_{23}=0$. By (\ref{eq:aithetab}), we get
$$ \theta^b =  \begin{bmatrix} 0  & \alpha & 0 \\ -\alpha & 0 & 0  \\ 0 & 0 & 0 \end{bmatrix}\!\!, $$
where $\overline{\alpha }= \alpha$. Write $\alpha' =\alpha + \psi_3$, where $\psi_3=a(\varphi_3+\overline{\varphi}_3)$ is real, and the structure equation of Chern connection amounts to
\begin{equation} \label{eq:6.11}
 d \varphi = - \,^t\!(\theta^b -\gamma ) \varphi + \tau =
\begin{bmatrix} \alpha' \varphi_2 + a_3 \varphi_3 \overline{\varphi}_2 \\ -\alpha' \varphi_1 - a_3 \varphi_3 \overline{\varphi}_1 \\  -a_3\varphi_1\varphi_2 +a (\varphi_2 \overline{\varphi}_1 - \varphi_1 \overline{\varphi}_2)  \end{bmatrix}\!\!.
\end{equation}
From the first two lines we obtain
\begin{eqnarray*}
d (\varphi_1\varphi_2) & = & d\varphi_1 \wedge \varphi_2 - \varphi_1 \wedge d\varphi_2 \ = \ -a_3 \varphi_3 (\varphi_1 \overline{\varphi}_1 + \varphi_2 \overline{\varphi}_2),\\
d (\varphi_1\overline{\varphi}_2) & = & d\varphi_1\wedge \overline{\varphi}_2 - \varphi_1\wedge d\overline{\varphi}_2 \ = \ \alpha' ( \varphi_2 \overline{\varphi}_2 - \varphi_1 \overline{\varphi}_1 ).
\end{eqnarray*}
Here we used the fact $\overline{\alpha'}=\alpha'$. In particular,
$$ d(\varphi_2 \overline{\varphi}_1 ) = - \overline{d(\varphi_1 \overline{\varphi}_2) } = \alpha ' ( \varphi_2 \overline{\varphi}_2 - \varphi_1 \overline{\varphi}_1 ) = d(\varphi_1 \overline{\varphi}_2) . $$
Take the exterior differentiation of the third equation of (\ref{eq:6.11}), we get
$$ 0 = d^2\varphi_3 = -a_3 \,d(\varphi_1\varphi_2) + a\, d(\varphi_2 \overline{\varphi}_1 - \varphi_1 \overline{\varphi}_2) = a_3^2  \varphi_3 (\varphi_1 \overline{\varphi}_1 + \varphi_2 \overline{\varphi}_2),$$
which is a contradiction. This shows that the condition $a_3>0$ cannot occur,  therefore we are  either in the situation $a_1>0=a_2=a_3$ or in the situation $a_1=a_2>0=a_3$, which implies
\[B=\begin{bmatrix} c & 0 & 0 \\ 0 & 0 & 0 \\ 0 & 0 & 0   \end{bmatrix}\ \ \text{or}\ \
\begin{bmatrix} c & 0 & 0 \\ 0 & c & 0 \\ 0 & 0 & 0   \end{bmatrix}\!\!,\]
where $c=2a_1^2>0$, and the rank of $B$ is either $1$ or $2$. This completes the proof of Proposition \ref{Btype}.
\end{proof}

\begin{example}[The simple complex Lie group $SL(2,{\mathbb C})$] \label{exampleSL2C}
 Let us consider the only simple complex Lie algebra ${\mathfrak g}={\mathfrak s}{\mathfrak l}(2,{\mathbb C})$ in complex dimension $3$. It consists of all complex $2\times 2$ matrices with zero trace. Take
\begin{equation*}
X = \left[ \begin{array}{cc} -i & 0 \\ 0 & i \end{array} \right] , \ \ \ \ \ Y = \left[ \begin{array}{cc} 0 & -i \\ -i & 0 \end{array} \right] , \ \ \ \ \ Z= \left[ \begin{array}{cc} 0 & 1 \\ -1 & 0 \end{array} \right] .
\end{equation*}
Then $\{ X, Y, Z\}$ forms a  basis of ${\mathfrak g}$ satisfying
\begin{equation} \label{sl2C}
[X, Y] =-2Z, \ \ \ \ \ [Y,Z]=-2X, \ \ \ \  [Z,X]=-2Y.
\end{equation}
Let $g_0$ be the inner product on the complex vector space ${\mathfrak g}$ so that $\{ X, Y, Z\}$ becomes a unitary basis. Then it corresponds to a left-invariant metric on the Lie group $G=SL(2,{\mathbb C})$, which we still denote by $g_0$, compatible with the complex structure of $G$. This is the {\em standard metric} of $G$. Clearly $g_0$ is Chern flat and non-K\"ahler, and it can be verified that it is also BTP and balanced. It is the metric induced by the Killing form.
\end{example}

Next we show that when the rank of the  $B$ tensor is $3$, the balanced BTP threefold must be a quotient of $SL(2,{\mathbb C})$ equipped with (a constant multiple of) the standard Killing metric $g_0$.

\begin{proposition}
Let $(M^3,g)$ be a compact balanced BTP threefold. Assume that the rank of the $B$ tensor is $3$. Then $(M^3,g)$ is holomorphically isometric to a compact quotient of the simple complex Lie group $SL(2,{\mathbb C})$ equipped with (a scaling of) the standard metric $g_0$. Furthermore, the canonical line bundle of $M^3$ is holomorphically trivial, and the restricted holonomy group of the Bismut connection $\mbox{Hol\,}_0^{b}(M)$ is contained in $ SO(3)\subseteq U(3)$.
\end{proposition}

\begin{proof}
In the proof of Case 2 in Proposition \ref{Btype}, we already know that $(M^3,g)$ is a compact Chern flat threefold. By \cite{Boothby}, the universal cover of $M^3$ is a connected, simply-connected complex Lie group $G$, and the metric $\tilde{g}$ which is the lift of $g$ is  left-invariant and is  compatible with the complex structure of $G$. Let $\{ \varepsilon_1, \varepsilon_2, \varepsilon_3\}$ be a left-invariant $\nabla^c$-parallel unitary frame of $G$. Denote by $T^j_{ik}$ the components of the Chern torsion under the frame $\varepsilon$, which are all constants. As $\tilde{g}$ is also non-K\"ahler balanced  BTP, after a constant unitary transformation, we may assume that $\varepsilon$ is a special frame on $G$ and $a_1=a_2=a_3=a>0$ under this frame. Hence the Lie algebra of $G$ is ${\mathfrak g}={\mathbb C}\{ \varepsilon_1,\varepsilon_2,\varepsilon_3\}$, satisfying
$$ [\varepsilon_1,\varepsilon_2]=-a\varepsilon_3, \ \ \ [\varepsilon_2,\varepsilon_3]=-a\varepsilon_1, \ \ \ [\varepsilon_3,\varepsilon_1]=-a\varepsilon_2. $$
Therefore, after scaling the metric by a constant multiple, $G$ is holomorphically isomorphic to $SL(2,{\mathbb C})$. Again from the proof of Case 2 in Proposition \ref{Btype},  we know that the matrix $\theta^b$ is skew-symmetric, so the curvature matrix $\Theta^b$ is also skew-symmetric, hence the restricted holonomy group of the Bismut connection $\nabla^b$ is contained in $SO(3)\subseteq U(3)$.

Next let us prove that $M^3$ must have trivial canonical line bundle, or equivalently, $M^3$ admits a global holomorphic $3$-form which is nowhere zero. To see this, let $e$, $\tilde{e}$ be two special frames in a neighborhood $U\subseteq M$, with dual coframes $\varphi$, $\tilde{\varphi}$ respectively, satisfying $\tilde{e}=Pe$ for some $U(3)$-valued smooth function $P$ on $U$. Denote by $T^i_{jk}$ and $\tilde{T}^i_{jk}$ the components of the Chern torsion under the frame $e$, $\tilde{e}$ respectively. We have $T^1_{23}=T^2_{31}=T^3_{12}=a>0$, $T^i_{ij}=0$ for any $i$, $j$, and the same holds for $\tilde{T}^i_{jk}$. Let $S$ be the $3\times 3$ matrix with $S_{ij}$ being the determinant of the $2\times 2$ matrix obtained by deleting the $i$-th row and $j$-th column in $P$. That is,  $S=(\det P) \,^t\!P^{-1}$. From $\tilde{T}^i_{jk} = \sum_{\alpha, \beta ,\gamma } \overline{P_{i\alpha }} P_{j\beta } P_{k \gamma } T^{\alpha}_{\beta \gamma }$, we derive
$$ \sum_{\alpha} \overline{P_{i\alpha }}S_{j \alpha } = \delta_{ij}, \ \ \ \forall \ 1\leq i,j\leq 3. $$
That is, $P^{\ast}S=I$, where $P^{\ast}$ is the conjugate transpose of $P$. Taking determinant, we get $\det P=1$, thus $^t\!P P=I$ so $P$ is actually an $SO(3)$-valued function. Therefore we have proved that in the $\mbox{rank}\,B=3$ case the special frames are related by $SO(3)$ changes. Since $\tilde{\varphi}=\,^t\!P^{-1}\varphi = P\varphi$, we see that the $(3,0)$-form
$$\psi := \tilde{\varphi}_1 \wedge \tilde{\varphi}_2\wedge \tilde{\varphi}_3 = \det P \,\varphi_1\wedge \varphi_2 \wedge \varphi_3= \varphi_1\wedge \varphi_2 \wedge \varphi_3$$
is independent of the choice of special frames thus is globally defined on $M$. By (\ref{eq:6.10}), $d\psi=\mbox{tr}(\theta) \psi =0$, which implies that $\psi$ is a global nowhere vanishing holomorphic $3$-form on $M$, hence the canonical line bundle $K_{\!M}$ is trivial. This completes the proof of the proposition.
\end{proof}

\begin{remark}
The proof above indicates that any compact smooth quotient $M=SL(2,{\mathbb C})/\Gamma $, where $\Gamma$ is a discrete subgroup of the group of holomorphic isometries for $g_0$, will have trivial canonical line bundle, so its unrestricted holonomy group for the  Chern connection is contained in $SU(3)$.
\end{remark}

This completes the proof of the  Chern flat case of Theorem \ref{thm1.1}. In the next two sections, we  will  deal with the case $\mathrm{rank}\,B=1$, which leads to the Wallach threefold $(X,g)$ where the underlying space is the flag threefold $\mathbb{P}(T_{\mathbb{P}^2})$, the projectivization of the holomorphic tangent bundle of ${\mathbb P}^2$, and the metric is the one induced from the Killing form of
${\mathfrak s}{\mathfrak u}(3)$. The case $\mathrm{rank}\,B=2$ will be called the {\em middle type}, which constitutes the main body of balanced BTP threefolds. They will be discussed in \S \ref{mddtype3D} in detail. Putting these three cases together, we will complete the proof of Theorem \ref{thm1.1}.

\vspace{0.4cm}

\section{The Fano case}\label{FANO}
In this section, we deal with the case $\mathrm{rank}\,B=1$.
%Eventually, we will show that this leads us to a unique example, the Wallach threefold, up to a scaling of the metric by a constant.
Let $(M^3,g)$ be a compact, non-K\"ahler, balanced BTP threefold with $B$ tensor
\[ B= \begin{bmatrix} 2a_1^2 & 0 & 0 \\ 0 & 0 & 0 \\ 0 & 0 & 0 \end{bmatrix}\!\!,\]
under a special frame $e$, where the only non-zero component of the Chern torsion tensor is $a_1=T^1_{23} >0$. We may assume for simplicity that $a_1=1$ after  scaling  the metric $g$ by a suitable constant multiple.

From (\ref{eq:aithetab}) we get $\theta^b_{12}=\theta^b_{13}=0$ and $\theta^b_{11}=\theta^b_{22} +\theta^b_{33}$. Then it follows that
\begin{equation*}
\theta^b = \begin{bmatrix} x+y & 0  & 0 \\ 0 & x & \alpha \\ 0 & -\overline{\alpha} &  y \end{bmatrix}\!\!, \ \ \ \
\gamma =  \begin{bmatrix} 0 & -\overline{\varphi}_3  & \overline{\varphi}_2 \\ \varphi_3 & 0 & 0 \\ -  \varphi_2 & 0 &  0 \end{bmatrix}\!\!, \ \ \
\tau = \begin{bmatrix} \varphi_2\varphi_3 \\ 0 \\ 0  \end{bmatrix}\!\!,
\end{equation*}
where $\overline{x}=-x$ and $\overline{y}=-y$. From $\theta = \theta^b-\gamma$ and $d\varphi = -\,^t\!\theta \wedge \varphi +\tau $, it yields that
\begin{equation*}
\theta = \begin{bmatrix}  x+y & \overline{\varphi}_3  & -\overline{\varphi}_2 \\ -\varphi_3 & x & \alpha \\ \varphi_2 & -\overline{\alpha} &  y \end{bmatrix}\!\!, \ \ \ \ d\varphi  = \begin{bmatrix}  -(x+y)\varphi_1 - \varphi_2\varphi_3 \\ \varphi_1 \overline{\varphi}_3 - x \varphi_2 +\overline{\alpha} \varphi_3  \\ - \varphi_1 \overline{\varphi}_2 - \alpha \varphi_2 -y \varphi_3  \end{bmatrix}\!\!.
\end{equation*}
The matrix of the curvature of $\nabla^b$ is $\Theta^b=d \theta^b - \theta^b \wedge \theta^b$, whose entries are
\begin{equation} \label{eq:7.1}
\left\{ \begin{array}{lll}  \Theta^b_{22} \ = \ dx + \alpha \overline{\alpha}, \ \ \ \ \ \Theta^b_{33} \ = \ dy - \alpha \overline{\alpha}, \\
 \Theta^b_{23} \ = \ d\alpha  - x \alpha - \alpha y, \ \ \  \Theta^b_{12} \ = \ \Theta^b_{13} \ = \ 0, \\
 \Theta^b_{11} \ = \ dx+dy \ = \  \Theta^b_{22}+ \Theta^b_{33}.
 \end{array} \right.
\end{equation}
For convenience, we will use $\varphi_{ij}$ and $\varphi_{i\bar{j}}$ as the abbreviation of $\varphi_i \wedge \varphi_j$ and $\varphi_i \wedge \overline{\varphi}_j$, respectively. From the exterior differentiation $d\varphi$ of $\varphi$, the first Bianchi identity of $\nabla^b$ amounts to
\begin{equation} \label{eq:7.2}
\left\{ \begin{array}{lll} 0 =  d^2\varphi_1 \ = \ \{ \varphi_{2\overline{2}} + \varphi_{3\overline{3}} - \Theta^b_{11} \} \wedge \varphi_1, \\
0 = d^2\varphi_2 \ = \ \{  \varphi_{1\overline{1}} -  \varphi_{3\overline{3}} - \Theta^b_{22} \} \wedge \varphi_2 - \Theta^b_{32}\wedge \varphi_3, \\
0 = d^2\varphi_3 \ = \ \{  \varphi_{1\overline{1}} -  \varphi_{2\overline{2}} - \Theta^b_{33} \} \wedge \varphi_3 - \Theta^b_{23}\wedge \varphi_2.
\end{array} \right.
\end{equation}
It follows from \cite[Theorem 1.1]{ZhaoZ24} that the BTP condition $\nabla^bT=0$ implies
$$ R^b_{ijk\overline{\ell}} =0 \ \ \ \mbox{and} \ \ \  R^b_{i\overline{j}k\overline{\ell}} = R^b_{k\overline{\ell}i\overline{j}}, \ \ \ \ \ \forall \ 1 \leq i,j,k,\ell \leq n. $$
So by $\Theta^b_{12}=\Theta^b_{13}=0$ and $\Theta^b_{11}= \Theta^b_{22} + \Theta^b_{33}$, we get
$$ R^b_{1\overline{b}i \overline{j}} = R^b_{i \overline{j}1\overline{b}}=0, \ \ \ \ \ R^b_{1 \overline{1}i\overline{j}} = R^b_{i\overline{j}1 \overline{1}} = R^b_{i\overline{j}2 \overline{2}} + R^b_{i\overline{j}3 \overline{3}}
= R^b_{2 \overline{2}i\overline{j}} + R^b_{3 \overline{3}i\overline{j}} $$
for any $i$, $j$ and any $b\in \{ 2,3\}$. Write
\begin{eqnarray*}
\Theta^b_{22} & = & A \varphi_{2\overline{2}} + B \varphi_{3\overline{3}} + E \varphi_{2\overline{3}} + \overline{E} \varphi_{3\overline{2}} + (A+B) \varphi_{1\overline{1}} \\
\Theta^b_{33} & = & B \varphi_{2\overline{2}} + C \varphi_{3\overline{3}} + F \varphi_{2\overline{3}} + \overline{F} \varphi_{3\overline{2}} + (B+C) \varphi_{1\overline{1}} \\
\Theta^b_{23} & = & E \varphi_{2\overline{2}} + F \varphi_{3\overline{3}} + D \varphi_{2\overline{3}} + G \varphi_{3\overline{2}} + (E+F) \varphi_{1\overline{1}} \\
\Theta^b_{11} & = & (A+B) \varphi_{2\overline{2}} + (B+C) \varphi_{3\overline{3}} + (E+F) \varphi_{2\overline{3}} + (\overline{E}+\overline{F}) \varphi_{3\overline{2}} + (A+2B+C) \varphi_{1\overline{1}}
\end{eqnarray*}
where $A,B,C,G$ are local real smooth functions. Then the first Bianchi identity (\ref{eq:7.2}) implies that
$$ E+F=0, \ \ \ \ A+B=B+C=G-B=1. $$
In particular,
\begin{equation}  \label{eq:7.3}
\Theta^b_{11} = 2\varphi_{1\overline{1}} + \varphi_{2\overline{2}} + \varphi_{3\overline{3}}
\end{equation}

\begin{remark} \label{remark7.1}
The pattern of the Bismut curvature above implies that the restricted holonomy group of $\nabla^b$ in this case is contained in $G_1 \subseteq U(3)$, where
$$  G_1 = \left\{ \begin{bmatrix} \det Y & 0 \\ 0 & Y \end{bmatrix} \, \Big| \ Y \in U(2) \right\}, $$
which is isomorphic to $U(2)$.
\end{remark}

Denote by $\Theta$ the curvature matrix of the Chern connection $\nabla^c$ under $e$. The balanced condition $\eta =0$ implies that
$$ \mathrm{tr} \,\Theta = \mathrm{tr}\, \Theta^b = 2\Theta^b_{11} .$$
The K\"ahler form of the metric $g$ is $\omega = \sqrt{-1} \big(  \varphi_1\overline{\varphi}_1 + \varphi_2\overline{\varphi}_2 +\varphi_3\overline{\varphi}_3 \big)$, and its first Chern Ricci form is
\begin{equation} \label{omegatilde}
 Ric(\omega) = \sqrt{-1} \mbox{tr}\,\Theta = 2 \sqrt{-1} \big( 2 \varphi_1\overline{\varphi}_1 + \varphi_2\overline{\varphi}_2 +\varphi_3\overline{\varphi}_3 \big) := 2\tilde{\omega},
 \end{equation}
which is positive definite. Since the Ricci form is always  closed and globally defined, $\tilde{\omega}$ is the K\"ahler form of a K\"ahler metric $\tilde{g}$ on $M^3$. Since $\tilde{\omega}^3=2\omega^3$, we have
\begin{equation} \label{omegatilde2}
 Ric(\tilde{\omega}) = Ric(\omega) = 2\tilde{\omega}.
 \end{equation}
That is,  $\tilde{\omega}$ is a K\"ahler-Einstein metric with positive Ricci curvature $2$, hence $M^3$ is a Fano threefold. Denote by $E$ and $F$ the $C^{\infty}$ complex vector bundles on $M$ whose fibers are  $E_x = {\mathbb C} \{ e_2(x), e_3(x)\}$ and $F_x={\mathbb C}  e_1(x)$, respectively. They are both globally defined since $E$ is the eigenspace of $B$ corresponding to the eigenvalue $0$, and $F$ is the orthogonal complement of $E$ in the holomorphic tangent bundle $T_{\!M}=T^{1,0}M$. For any $i\in \{ 2, 3\}$ and any $j$, we have
$$ \langle \nabla^c_{\overline{j}}e_i , \overline{e}_1 \rangle = \langle \nabla^b_{\overline{j}}e_i - \gamma_{\overline{j}}e_i , \overline{e}_1  \rangle = \theta^b_{i1}(\overline{e}_{j}) -  \gamma_{i1}(\overline{e}_{j}) = 0.$$
This means that $\nabla^c_{\overline{X}}E \subseteq E$ for any type $(1,0)$ vector field $X$, so $E$ is holomorphic.
Note that the distribution $E$ is not a foliation, while $F$, on the other hand, is a foliation but is not holomorphic.

Equip $E$ with the restriction metric from $(M^3,g)$. By the formula of Chern connection matrix $\theta$ on $M$, the matrices of connection and curvature of the Hermitian bundle $E$ under the local frame $\{ e_2, e_3\}$ are respectively
$$ \theta^E = \begin{bmatrix} x & \alpha \\ -\overline{\alpha} & y \end{bmatrix}\!\!, \ \ \ \  \ \Theta^E = d\theta^E -\theta^E \wedge \theta^E = \begin{bmatrix} \Theta^b_{22} & \Theta^b_{23} \\ \Theta^b_{32} & \Theta^b_{33} \end{bmatrix}\!\!.$$
In particular, $\sqrt{-1}\,\mathrm{tr}\,\Theta^E = \sqrt{-1}(\Theta^b_{22} + \Theta^b_{33}) = \frac{\sqrt{-1}}{2}\mathrm{tr}\,\Theta^b = \frac{\sqrt{-1}}{2}\,\mathrm{tr}\,\Theta = \tilde{\omega}$. This means that
\begin{equation} \label{eq:7.6}
 c_1(E) = \frac{1}{2} c_1(M) = [\tilde{\omega}].
\end{equation}
Denote by $L$ the holomorphic line bundle $T_{\!M}/E$ on $M$. We have  the short exact sequence
\begin{equation} \label{eq:7.7}
0 \rightarrow E \rightarrow T_{\!M} \rightarrow L \rightarrow 0.
\end{equation}
Let $h=c_1(L)$ be the first Chern class of $L$. The above short exact sequence implies that
$$ c_1(E)+h =c_1(M), \ \ \ c_2(E) + hc_1(E) = c_2(M), \ \ \ c_2(E)h =c_3(M).$$
So by (\ref{eq:7.6}) we obtain
\[c_1(E)=h, \quad c_1(M)=2h, \quad c_2(E)=c_2(M)-h^2, \quad c_2(M)h - h^3=c_3(M).\]
In particular, $L$ is an ample line bundle on $M$. The anti-canonical line bundle $K_{\!M}^{\!-\!1}=L^{\otimes 2}$ as holomorphic line bundles are uniquely determined by their Chern classes on Fano manifolds, and the Chern numbers of $M^3$ satisfy
\begin{equation}  \label{eq:7.8}
c_1(M)c_2(M) = 2h (c_2(E)+h^2) = 2c_3(M) + \frac{1}{4}c_1^3(M).
\end{equation}

Recall that the {\em index} of a Fano manifold $X^n$ is the largest positive integer $r$ so that $K^{\!-\!1}_{\!X} = A^{\otimes r}$ for an ample line bundle $A$. It is necessarily less than or equal to $n+1$, where $r=n+1$ if and only if $X={\mathbb P}^n$ and $r=n$ if and only if $X={\mathbb Q}^n$, the smooth quadratic hypersurface in ${\mathbb P}^{n+1}$. Fano manifolds satisfying $r=n-1$ are called {\em del Pezzo manifolds,} which are classified by Fujita \cite{Fujita} as one of the following seven types according to their {\em degree} $d$, which is the self intersection number $A^n$:
\begin{enumerate}
\item $d=1$: $X^n_6\subset {\mathbb P}(1^{n-1}, 2, 3)$, a degree $6$ hypersurface in the weighted projective space.
\item $d=2$: $X^n_4\subset {\mathbb P}(1^{n}, 2)$, a degree $4$ hypersurface in the weighted projective space.
\item $d=3$: $X^n_3\subset {\mathbb P}^{n+1}$, a cubic hypersurface.
\item $d=4$: $X^n_{2,2}\subset {\mathbb P}^{n+2}$, a complete intersection of two quadrics.
\item $d=5$: $Y^n$, a linear section of ${\mathbb Gr}(2,5) \subset {\mathbb P}^9$.
\item $d=6$: ${\mathbb P}^1\!\times \! {\mathbb P}^1\!\times \!{\mathbb P}^1$, or ${\mathbb P}^2\!\times \! {\mathbb P}^2$, or the flag threefold ${\mathbb P}(T_{ \!{\mathbb P}^{2}} )$.
\item $d=7$: ${\mathbb P}^3\# \overline{{\mathbb P}^3} $, the blow-up of ${\mathbb P}^3$ at a point.
\end{enumerate}

For $n=3$, del Pezzo threefolds were classified by Iskovskikh \cite{Iskov} earlier, and in Table 12.2 of \cite{IP} we can find the third betti number $b_3$, hence the Euler number $c_3=4-b_3$ of del Pezzo threefolds of degree $1\leq d\leq 5$:
\begin{equation} \label{eq:7.9}
c_3(X^3_6)=-38, \ \ \ c_3(X^3_4)=-16, \ \ \ c_3(X^3_3) = -6, \ \ \ c_3(X^3_{2,2})= 0, \ \ \ c_3(Y^3)= 4.
\end{equation}

For the balanced {\em BTP} threefold $(M^3,g)$ with $\mbox{rank}\,B=1$, it holds that $K_{\!M}^{\!-\!1}=L^{\otimes 2}$ for an ample line bundle $L$, so the index of $M^3$ is either $4$ or $2$, which means $M^3$ is biholomorphic to either ${\mathbb P}^3$ or a del Pezzo threefold. It is well-known that $c_1c_2=24$ holds for any Fano threefold, so the equality (\ref{eq:7.8}) implies that
\begin{equation} \label{eq:7.10}
c_3(M) = 12 - \frac{1}{8}c_1^3 .
\end{equation}
If $M^3$ is a del Pezzo threefold of degree $d$, then $c_1^3=8d$, hence the equality (\ref{eq:7.10}) yields $c_3=12-d$. This rules out the possibility of $1\leq d\leq 5$ by (\ref{eq:7.9}). The case ${\mathbb P}^3\# \overline{{\mathbb P}^3}$ of degree $d=7$ has Euler number $c_3=6$, which is not equal to $12-7$. Similarly, for the case ${\mathbb P}^1 \times {\mathbb P}^1 \times {\mathbb P}^1$ of degree $d=6$, its Euler number $c_3=8$ is not equal to $12-6$, which indicates that neither can be $M^3$. Therefore only two possibilities are left, namely, $M^3$ {\em is either the flag threefold  ${\mathbb P}(T_{{\mathbb P}^2} )$ or ${\mathbb P}^3$}.

\vspace{0.1cm}

Consider the short exact sequence (\ref{eq:7.7}) of holomorphic vector bundles on $M^3$, where $E$ has the fiber ${\mathbb C}\{ e_2,e_3\}$ under a special frame $e$. Denote by $\Omega$ the holomorphic cotangent bundle of $M^3$ and write $\Omega(L)=\Omega \otimes L$. Let $\xi \in H^0(M, \Omega (L))$ be the nowhere zero holomorphic section which gives the map $T_{\!M} \rightarrow L$ in (\ref{eq:7.7}). The K\"ahler-Einstein metric $\tilde{\omega}$ on $M^3$ naturally induces Hermitian metrics on $\Omega$ and $L=-\frac{1}{2}K_{\!M}$, hence on $\Omega (L)$. We claim the following:
\begin{claim}\label{cst_norm}
The norm $\parallel \!\xi \!\parallel^2$ is a constant under the K\"ahler-Einstein metric $\tilde{\omega}$.
\end{claim}

\begin{proof}
To see this, let $e$ be a local special frame of $(M^3,g)$, with dual coframe $\varphi$. Let $s$ be a local holomorphic frame of $L$. Since $\xi$ is a $L$-valued $1$-form, we have $\xi = \psi \otimes s$ where $\psi$ is a nowhere-zero local holomorphic $1$-form. The kernel of the map given by $\xi$ is $E$, so $\psi = f\varphi_1$ for some local smooth function $f$, which is nowhere zero. By the structure equation, it follows that $d\varphi_1=-(x+y)\varphi_1 - \varphi_2\varphi_3$, so the holomorphicity of $\psi$ gives us
$$ 0 = \overline{\partial} \psi =  \overline{\partial}f \wedge \varphi_1 - f \,(x+y)^{0,1}\wedge \varphi_1. $$
Hence $(x+y)^{0,1}=\overline{\partial}\log f$. By the fact that $\overline{x}=-x$ and $\overline{y}=-y$, we get
$$ x+y = - \partial \log \overline{f} + \overline{\partial}\log f, \ \ \ \  \frac{1}{\sqrt{-1}}\tilde{\omega} = \Theta^b_{11} = d(x+y) = \partial \overline{\partial}\log |f|^2. $$
On the other hand,  $L=-\frac{1}{2}K_{\!M}$ is equipped with the induced metric from $\tilde{\omega}$, so we have
$$ \frac{1}{\sqrt{-1}} \tilde{\omega} = \Theta^L = - \partial \overline{\partial}\log \parallel \!\!s\!\!\parallel ^2. $$
Combine the above two equations, we obtain
\begin{equation} \label{eq:pluriclosedness}
 \partial \overline{\partial}\log (|f|^2 \!\parallel\!\!s\!\!\parallel^2) = 0.
 \end{equation}
On the other hand, since $\{ \sqrt{2} \varphi_1, \varphi_2, \varphi_3\}$ is a local unitary coframe for $\tilde{\omega}$, the norm square of $\varphi_1$ under $\tilde{\omega}$ is $\frac{1}{2}$, therefore we have
$ \parallel\!\!\xi\!\!\parallel^2=\frac{1}{2}\,|f|^2\parallel\!\!s\!\!\parallel^2$.
It is a global positive function on $M^3$, and its log is pluriclosed by (\ref{eq:pluriclosedness}), hence it must be a constant. This establishes the claim.
\end{proof}

Then Claim \ref{cst_norm} rules out the possibility of ${\mathbb P}^3$:
\begin{claim}
The expression $\|\xi\|^2$ above cannot be a constant, if $M^3={\mathbb P}^3$ endowed with the standard K\"ahler-Einstein metric.
\end{claim}

\begin{proof}
Assume that $M^3={\mathbb P}^3$. In this case $L={\mathcal O}(2)$. It follows from \cite{Schneider} for instance that $\Omega (2)$ admits a trivial line subbundle, which corresponds to a nowhere zero holomorphic section $\xi \in V:=H^0({\mathbb P}^3,\Omega (2))$. $\xi$ determines a surjective bundle map $T_{\mathbb{P}^3} \rightarrow L$ so that the exact sequence \eqref{eq:7.7} holds. Let $\tilde{\omega}$  be the (scaled) Fubini-Study metric of ${\mathbb P}^3$ with Ricci curvature $2$. It has constant holomorphic sectional curvature $1$. Let $[Z_0\!:\!Z_1\!:\!Z_2\!:\!Z_3]$ be the standard unitary homogeneous coordinate of ${\mathbb P}^3$. In the coordinate neighborhood $U_0=\{ Z_0\neq 0\}$, let $z_i=\frac{Z_i}{Z_0}$, $1\leq i\leq 3$ and it follows that
$$ \sqrt{-1}\, \Theta^L=\tilde{\omega} = \frac{1}{2}\mbox{Ric}(\tilde{\omega}) = 2\sqrt{-1}\,\partial \overline{\partial} \log (1+|z|^2), \ \ \ \ \ \ \parallel \!Z_0^2\! \parallel^2 = \frac{1}{(1+|z|^2)^2} $$
where $Z_0^2$ is a local frame of $L$ in $U_0$ and $|z|^2= |z_1|^2 + |z_2|^2 + |z_3|^2$. Under the coordinate $z$,
$$ \tilde{g}_{i\overline{j}} = \frac{2}{1+|z|^2}\delta_{ij} - \frac{2}{(1+|z|^2)^2}\overline{z}_iz_j, \ \ \ \ \ \ \tilde{g}^{\overline{i}j} = \frac{1}{2}(1+|z|^2) \big( \delta_{ij} + \overline{z}_iz_j\big)  .$$

It is well-known that $V \cong {\mathbb C}^6 $ has a basis $\{ \lambda_{ij} \}_{0\leq i<j\leq 3}$, where
$ \lambda_{ij} = Z_i dZ_j - Z_j dZ_i$. As $\xi \in V$ is a nowhere zero section,
it follows that $\xi = \sum a_{ij}\lambda_{ij}$ for some constants $a_{ij}$. In $U_0$,
$$  \lambda_{0i}=Z_0^2dz_i, \ \ \ \ \lambda_{ij}=Z_0^2(z_idz_j-z_jdz_i), \ \ \ \ \ 1\leq i,j\leq 3. $$
Thus $ \xi = Z_0^2 (\ell_{1}dz_1 + \ell_2dz_2 + \ell_3dz_3)$, where
\begin{eqnarray*}
\ell_1 & = & a_1 - a_{12}z_2-a_{13}z_3 \\
\ell_2 & = & a_2 + a_{12}z_1-a_{23}z_3 \\
\ell_3 & = & a_3 + a_{13}z_1 + a_{23}z_2
\end{eqnarray*}
It yields that
\begin{eqnarray*}
\parallel\! \xi\! \parallel^2 & = & \parallel \!Z_0^2\! \parallel^2 \sum_{i,j=1}^3 \ell_i \overline{\ell_j} \, \tilde{g}^{\overline{i}j} \\
& = & \parallel \!Z_0^2\! \parallel^2  \frac{1}{2}(1+|z|^2) \big( \sum_i |\ell_i|^2 + |\sum_k \ell_k\overline{z}_k |^2 \big) \\
& = & \frac{1}{2(1+|z|^2)}  \big(  |\ell_1|^2 + |\ell_2|^2 + |\ell_3|^2 + |\ell_1 \overline{z}_1 + \ell_2 \overline{z}_2 + \ell_3 \overline{z}_3 |^2 \big)
\end{eqnarray*}
Now assume that $\|\xi\|^2$ above is a constant, which means
\begin{equation} \label{eq:xisquareconstant}
  |\ell_1|^2 + |\ell_2|^2 + |\ell_3|^2 + |\ell_1 \overline{z}_1 + \ell_2 \overline{z}_2 + \ell_3 \overline{z}_3 |^2 = c\,(1 + |z_1|^2 + |z_2|^2 + |z_3|^2),
  \end{equation}
where $c>0$ is a constant. We want to derive at a contradiction. For each $1\leq i\leq 3$, write $\ell_i=\ell_i^{(0)} + \ell_i^{(1)}$ where $\ell_i^{(0)},\ \ell_i^{(1)}$ are respectively  the degree $0$ and degree $1$ part of $\ell_i$. By looking at the degree $4$ part of the left hand side of (\ref{eq:xisquareconstant}), we get
\[ |\ell_1^{(1)} \overline{z}_1 + \ell_2^{(1)} \overline{z}_2 + \ell_3^{(1)} \overline{z}_3 |^2 =0,\]
which means $ a_{12}=a_{13}=a_{23}=0$, and (\ref{eq:xisquareconstant}) now becomes
\[ |a_{1}|^2 + |a_{2}|^2 + |a_{3}|^2 + |a_{1} \overline{z}_1 + a_{2} \overline{z}_2 + a_{3} \overline{z}_3 |^2 = c\,(1 + |z_1|^2 + |z_2|^2 + |z_3|^2). \]
From this we obtain
\[\begin{bmatrix} \overline{a}_1a_{1} & \overline{a}_1a_{2} & \overline{a}_1a_{3} \\
\overline{a}_2a_{1} & \overline{a}_2a_{2} & \overline{a}_2a_{3} \\
\overline{a}_3a_{1} & \overline{a}_{3}a_{2} & \overline{a}_{3}a_{3}\end{bmatrix}
= c\begin{bmatrix} 1 & & \\ & 1 & \\ & & 1\end{bmatrix}\!\!,\]
which is a contradiction as the two sides have different ranks. This shows that $\|\xi\|^2$ cannot be a constant, and we have completed the proof of the claim.
\end{proof}

Now we are left with the only possibility for a compact balanced BTP threefold $(M^3,g)$ with $B$ rank equal to $1$: the flag threefold $X^3 = {\mathbb P}(T_{{\mathbb P}^2} )$. Note that $X^3$ is the hypersurface in $N^4={\mathbb P}^2\times {\mathbb P}^2$ defined by $Z_0W_0+Z_1W_1+Z_2W_2=0$, where $Z$, $W$ are the standard unitary homogeneous coordinate of the two factors of $N$. For $i=1,2$, denote by $\pi_i: X^3 \rightarrow {\mathbb P}^2$ the restriction on $X$ of the projection map from $N$ onto its $i$-th factor. The Picard group $\mbox{Pic}(X)\cong {\mathbb Z}^{\oplus 2}$ is generated by $L_1$ and $L_2$, where $L_i=\pi_i^{\ast}{\mathcal O}_{{\mathbb P}^2}(1)$, and the anti-canonical line bundle of $X$ is $-K_{\!X} = 2L=2(L_1+L_2)$. Here for convenience we used the additive notation for line bundles. The K\"ahler-Einstein metric $\tilde{\omega}$ on $X$ is the restriction of the product of Fubini-Study metric, and we have
\begin{equation} \label{eq:7.13}
 \sqrt{-1}\, \Theta^L = \tilde{\omega} = \frac{1}{2}\mbox{Ric}(\tilde{\omega}) = \omega_0|_X, \ \ \ \omega_0 = \sqrt{-1} \, \partial \overline{\partial} \log |Z|^2 + \sqrt{-1} \, \partial \overline{\partial} \log |W|^2  .
 \end{equation}

\begin{claim} \label{eq:7.14}
Let $X^3$ be the flag threefold ${\mathbb P}(T_{{\mathbb P}^2} )$, $\Omega$ its holomorphic cotangent bundle, and $L$ be the line bundle  $\pi_1^{\ast}{\mathcal O}_{{\mathbb P}^2}(1) \otimes \pi_2^{\ast}{\mathcal O}_{{\mathbb P}^2}(1)$. Then we have
\[H^0(X, \Omega \otimes L) = {\mathbb C} \xi, \ \ \ \ \ \ \ \ \xi = \sum_{i=0}^2 W_idZ_i = - \sum_{i=0}^2 Z_idW_i.\]
\end{claim}

\begin{proof}
It is clear that $\xi$ defined in the claim is a global holomorphic section of $\Omega \otimes L$, and is nowhere zero, hence it gives a surjective bundle map $T_X \rightarrow L$ for the short exact sequence \eqref{eq:7.7}. Here we want to show that the vector space $H^0(X, \Omega \otimes L)$ is one-dimensional, hence any section is a constant multiple of $\xi$. To see this, let us denote by $T_{X|{\mathbb P}^2}$ the relative tangent bundle of the map $\pi_1: X \rightarrow {\mathbb P}^2$, given by
$$ 0 \rightarrow T_{X|{\mathbb P}^2} \rightarrow T_X \rightarrow \pi_1^{\ast}T_{{\mathbb P}^2} \rightarrow  0 .$$
Then we have $T_{X|{\mathbb P}^2}=2L-3L_1=2L_2-L_1$. Taking the dual of the above short exact sequence and tensoring it with $L$, we get
\begin{equation} \label{eq:7.15}
0 \rightarrow  \pi_1^{\ast}\Omega_{{\mathbb P}^2} \!\otimes \!L \rightarrow \Omega_X\! \otimes\! L \rightarrow L' \rightarrow  0 , \ \ \ \ \ \  \ \
L' = -T_{X|{\mathbb P}^2} + L = 2L_1-L_2.
\end{equation}
On the other hand, $X={\mathbb P}(T_{{\mathbb P}^2})$, so the relative Euler sequence is
\begin{equation} \label{eq:7.16}
0 \rightarrow  {\mathcal O}_X \rightarrow \pi_1^{\ast}\Omega_{{\mathbb P}^2} \!\otimes \!L \rightarrow  T_{X|{\mathbb P}^2} \rightarrow  0.
\end{equation}
Since $L^3_1=0$, $L_1^2L_2=1$, we have $L_1^2 L'=-1$, so $H^0(X,L')=0$ as $L^2_1$ is represented by the fibers of $\pi_1$ which will have non-negative intersection with any effective divisor in $X$. Similarly, $H^0(X, T_{X|{\mathbb P}^2})=0$. So by (\ref{eq:7.15}) and (\ref{eq:7.16}), we get
$$  H^0(X, \Omega_X\! \otimes \!L ) \cong H^0(X, \pi_1^{\ast}\Omega_{{\mathbb P}^2} \!\otimes \!L) \cong H^0(X,  {\mathcal O}_X) \cong {\mathbb C}. $$
This establishes Claim \ref{eq:7.14}.
\end{proof}

Consider the global $(1,1)$-form on $X$ defined by
\begin{equation} \label{eq:7.17}
\sigma = \sqrt{-1}\,\frac{\,^tWdZ \wedge \overline{^tWdZ} }{|Z|^2 \,|W|^2}
\end{equation}
where $Z$, $W$ are unitary homogeneous coordinate on the two factors of $N={\mathbb P}^2\times {\mathbb P}^2$, viewed as  column vectors. It is not hard to see that the norm $\parallel\! \!\sigma \!\!\parallel =\frac{1}{2}$ with respect to the K\"ahler-Einstein metric $\tilde{\omega}$ of $X$. Consider the Hermitian metric $g$ on $X$ with K\"ahler form $\omega = \tilde{\omega }- \sigma$, which is clearly a homogeneous Hermitian metric on $X$. We will call the Hermitian manifold $(X,g)$ the {\em Wallach threefold} from now on, to honor the influential work \cite{Wallach} in geometry.

We will verify in the next section that $(X,g)$ is indeed balanced and BTP.  We will also show that its Chern connection has non-negative bisectional curvature and positive holomorphic sectional curvature, and all three Ricci tensors of the Chern connection are positive. The sectional curvature of the Levi-Civita connection of $g$ is non-negative, and the Levi-Civita connection has constant Ricci curvature $3$, thus $g$ lies in the boundary of the set of metrics with positive sectional curvature discovered by Wallach in \cite{Wallach}.

Note that homogeneous metrics on $X$ with positive sectional curvature, which are all Hermitian as observed by Wallach in \cite{Wallach}, form a moduli which depends on three real parameters. After scaling, these metrics form a peculiar planer region (see for example Figure 1 in \cite{BM}). It is not clear which metric in the set is the `best' amongst its peers.
%Our metric $g$ corresponds to the one where all three parameters are equal (that is, the metric given by the Killing form).
%It is Einstein and has non-negative sectional curvature but not strictly positive sectional curvature, but it is the unique (up to scaling) balanced and {\em BTP} metric on $X$. Its Chern %connection also has non-negative bisectional curvature, while the (unique) K\"ahler-Einstein metric $\tilde{g}$ of $X$ does not have nonnegative bisectional curvature.

\vspace{0.4cm}

\section{The Wallach threefold}\label{WCH3D}
Let $\omega_0$ be the product of Fubini-Study metric on $N^4={\mathbb P}^2\times {\mathbb P}^2$, given by (\ref{eq:7.13}), where $Z$ and $W$ are unitary homogeneous coordinates, and the flag threefold $X$ is defined by the smooth ample divisor $\{ \,^t\!ZW=0\}$ in $N$. Here and below we will consider $Z$ and $W$ as column vectors. The restriction $\tilde{\omega }= \omega_0|_X$ is the K\"ahler-Einstein metric on $X$ with $\mbox{Ric}(\tilde{\omega})=2\tilde{\omega}$, and our Hermitian metric $g$, which will be called the {\em Wallach metric} from now on, is defined by $\omega = \tilde{\omega} - \sigma $, where the global $(1,1)$-form $\sigma$ on $X$ is given in (\ref{eq:7.17}). We will verify that $g$ is balanced and BTP, and compute its Chern and Riemannian curvature.

Fix any point $p\in X=SU(3)/T^2$. Note that for any $A\in SU(3)$, the map $([Z], [W]) \mapsto ([AZ],[\overline{A}W])$ is an isometry on $(X,g)$. So without loss of generality, we may assume that $p =([1\!\!:\!0\!:\!0], [0\!:\!0\!:\!\!1])$. Consider the neighborhood $U_{02}=\{ Z_0\neq 0\} \times \{ W_2\neq 0\}$ in $N$, with local holomorphic coordinate $(z_1, z_2, w_0, w_1)$ where $z_i=\frac{Z_i}{Z_0}$, $i=1,2$, and $w_j = \frac{W_j}{W_2}$, $j=0,1$. Within $U_{02}$, the hypersurface $X$ is defined by
\begin{equation} \label{eq:8.1}
w_0=-z_2-z_1w_1,
\end{equation}
so $(z_1, z_2, w_1)$ becomes a local holomorphic coordinate in $U=X\cap U_{02}$. Let us write $|z|^2=|z_1|^2+|z_2|^2$ and $|w|^2 = |w_0|^2 + |w_1|^2$ as usual, then in $U_{02}$ we have
\begin{equation} \label{eq:8.2}
\frac{1}{\sqrt{-1}}\omega_0 = \sum_{i,j=1}^2 \frac{(1+|z|^2)\delta_{ij} - \overline{z}_iz_j } {(1+|z|^2)^2} dz_i\wedge d\overline{z}_j + \sum_{i,j=0}^1 \frac{(1+|w|^2)\delta_{ij} - \overline{w}_iw_j } {(1+|w|^2)^2} dw_i\wedge d\overline{w}_j,
\end{equation}
and in $U$, $\tilde{\omega}$ is just the restriction of $\omega_0$ on $X$ via the equation (\ref{eq:8.1}). For convenience, let us write $w_1=z_3$, and define
\begin{equation} \label{eq:8.3}
 \alpha = 1+ |z_1|^2 + |z_2|^2, \ \ \  \beta = 1+|z_3|^2 + |f|^2, \ \ \ f=z_2+z_1z_3.
\end{equation}
In $U$, $(z_1,z_2,z_3)$ gives a local holomorphic coordinate for $X$, and $p$ corresponds to the origin $(0,0,0)$. By (\ref{eq:8.2}) the metric $\tilde{g}$ has components
\begin{equation} \label{eq:8.4}
 \tilde{g}_{i\overline{j}} = \frac{\alpha_{i\overline{j}}}{\alpha} - \frac{\alpha_i \alpha_{\overline{j}} }{\alpha^2} + \frac{\beta_{i\overline{j}}}{\beta} - \frac{\beta_i \beta_{\overline{j}} }{\beta^2}, \ \
 \ \ \ \ 1\leq i,j\leq 3,
\end{equation}
where subscripts stand for partial derivatives in $z_i$ or $\overline{z}_j$.  Taking partial derivative in $z_k$, we obtain
\begin{equation} \label{eq:8.5}
 \tilde{g}_{i\overline{j},k} = - \frac{1}{\alpha^2} \big( \alpha_k \alpha_{i\overline{j}}  + \alpha_i \alpha_{k\overline{j}} \big) + \frac{2\alpha_i \alpha_k \alpha_{\overline{j}} }{\alpha^3} + \frac{\beta_{ik\overline{j}}}{\beta} - \frac{1}{\beta^2} \big( \beta_k \beta_{i\overline{j}} + \beta_i \beta_{k\overline{j}} + \beta_{\overline{j}} \beta_{ik} \big)  + \frac{2\beta_i\beta_k \beta_{\overline{j}} }{\beta^3}. \ \ \
\end{equation}
Here we used the fact that $\alpha_{ik}=0$ and $\alpha_{ik\overline{j}} =0$. At the origin, $\alpha (0)=\beta (0)=1$, $\alpha_i(0)=\beta_i(0)=0$, $\beta_{ik}(0)=0$, and $\alpha_{i\overline{j}}(0)= \delta_{i1}\delta_{j1}+ \delta_{i2}\delta_{j2}$, $\beta_{i\overline{j}}(0)=  \delta_{i2}\delta_{j2}+ \delta_{i3}\delta_{j3}$, so we have
\begin{eqnarray}
&&  \tilde{g}_{i\overline{j},k}(0) \ = \beta_{ik\overline{j}}(0) =f_{ik}\overline{f_j}(0)= (\delta_{i1}\delta_{k3}+\delta_{i3}\delta_{k1})\delta_{j2} , \label{eq:8.6} \\
 && \tilde{g}_{i\overline{j},kp}(0) = 0, \label{eq:8.7} \\
 && \tilde{g}_{i\overline{j},k\overline{\ell}}(0) = - \big(\alpha_{i\overline{j}} \alpha_{k\overline{\ell}} + \alpha_{i\overline{\ell}} \alpha_{k\overline{j}} \big)  + f_{ik}\overline{ f_{j\ell} } - \big(\beta_{i\overline{j}} \beta_{k\overline{\ell}} + \beta_{i\overline{\ell}} \beta_{k\overline{j}} \big). \label{eq:8.8}
\end{eqnarray}
From this, we see that
\begin{equation} \label{eq:8.9}
\left\{ \begin{array}{llll}  \tilde{g}_{i\overline{j},k\overline{\ell}}(0)  & = & 0 , \ \ \ \ \mbox{if} \ \{ i,k\} \neq \{ j,\ell\},  \\ \tilde{g}_{i\overline{i},i\overline{i}}(0) &  =  &  -2 (\alpha_{i\overline{i}} + \beta_{i\overline{i}} )  \ = \ - 2\,\big( 1+ \delta_{i2}\big), \\
 \tilde{g}_{i\overline{i},k\overline{k}}(0)  & = &  \tilde{g}_{i\overline{k},k\overline{i}}(0)  \ = \ -(\alpha_{i\overline{i}}\alpha_{k\overline{k}} + \beta_{i\overline{i}}\beta_{k\overline{k}}) + |f_{ik}|^2 \ = \  \left\{ \begin{array}{ll} -1, \ \mbox{if} \ \{ i,k\} = \{ 1,2\} \ \mbox{or} \ \{ 2,3\} , \\ \ \,1, \ \ \mbox{if} \ \{ i,k\} = \{ 1,3\} .\end{array} \right.
 \end{array}  \right.
\end{equation}
Similarly, since in $U$ the $(1,1)$-form $\sigma$ is given by
$$ \sigma = \sqrt{-1}\sum_{i,j=1}^3 \sigma_{i\overline{j}} dz_i \wedge d\overline{z}_j = \frac{\sqrt{-1} }{\alpha \beta } (z_3dz_1+dz_2) \wedge (\overline{z}_3 d\overline{z}_1 + d\overline{z}_2),
$$
we therefore  have
 \begin{equation*}
 \sigma_{i\overline{j}} = \frac{1}{\alpha \beta} ( \delta_{i1}\delta_{j1} |z_3|^2 + \delta_{i1}\delta_{j2} z_3 + \delta_{i2}\delta_{j1}\overline{z}_3 + \delta_{i2}\delta_{j2}).
\end{equation*}
From this we compute
\begin{equation} \label{eq:8.10}
\left\{ \begin{array}{llll} \sigma_{i\overline{j}}(0) & = & \delta_{i2}\delta_{j2}, \\
\sigma_{i\overline{j},k}(0) & = & \delta_{i1}\delta_{j2}\delta_{k3}, \\
 \sigma_{i\overline{j},kp}(0) & = & 0,\\
 \sigma_{i\overline{j},k\overline{\ell}}(0)  & = & \delta_{ij}\delta_{i1} \delta_{k\ell}\delta_{k3} - \delta_{k\ell} (1+\delta_{k2})\delta_{ij} \delta_{i2}.
\end{array}  \right.
\end{equation}
By (\ref{eq:8.4}), we have $\tilde{g}_{i\overline{j}}(0)=\delta_{ij}(1+\delta_{i2})$, so at the origin $g_{i\overline{j}}=\tilde{g}_{i\overline{j}} - \sigma_{i\overline{j}}$ satisfies
\begin{equation} \label{eq:8.11}
 g_{i\overline{j}}(0) = \delta_{ij}, \ \ \
g_{i\overline{j},k}(0) = \delta_{i3}\delta_{j2}\delta_{k1}, \ \ \
 g_{i\overline{j},kp}(0) =0,
\end{equation}
\begin{equation} \label{eq:8.12}
g_{i\overline{j},k\overline{\ell}}(0)  \,= \,0  \  \ \mbox{if} \ \{ i,k\} \neq \{ j,\ell\}, \ \ \ \ \ \ \  g_{i\overline{i},i\overline{i}}(0) \,  = \, -2,
\end{equation}
\begin{eqnarray}
\label{eq:8.13} && g_{i\overline{k},k\overline{i}}(0)  \ = \  \left\{ \begin{array}{ll} -1, \ \mbox{if} \ \{ i,k\} = \{ 1,2\} \ \mbox{or} \ \{ 2,3\} , \\ \ \,1, \ \ \mbox{if} \ \{ i,k\} = \{ 1,3\} .\end{array} \right. \\
\label{eq:8.14} && g_{i\overline{i},k\overline{k}}(0)  \ = \ \left\{ \begin{array}{lll} -1 + \delta_{i2},\ \ \ \ \ \mbox{if} \ \{ i,k\} = \{ 1,2\} \ \mbox{or} \ \{ 2,3\} , \\ \ \,1-\delta_{i1}\delta_{k3}, \ \ \mbox{if} \  \{ i,k\} = \{ 1,3\} . \\ \end{array} \right.
\end{eqnarray}
The curvature components of the Chern connection $\nabla^c$, defined as
\[ R^c_{k\overline{\ell}i\overline{j}} = \sum_r \Theta_{ir}(\frac{\partial}{\partial z_k},\frac{\partial}{\partial \overline{z}_\ell})g_{r\bar{j}},
\quad \Theta = \overline{\partial} \theta =\overline{\partial} (\partial g g^{-1}),\]
where $g=(g_{i\overline{j}})$, is given by
$$ R^c_{k\overline{\ell}i\overline{j}} = - g_{i\overline{j},k\overline{\ell}} + \sum_{p,q} g_{i\overline{p},k} \overline{ g_{j\overline{q},\ell}  } g^{\overline{p}q} .$$
At the origin, $g_{i\overline{j}} (0)=\delta_{ij}$, and $g_{i\overline{j},k} (0)=0$ except $g_{3\overline{2},1} (0)=1$, so the second term on the right hand side of the above equality is $\delta_{ij}\delta_{i3}\delta_{k\ell}\delta_{k1}$, and by (\ref{eq:8.12}), (\ref{eq:8.13}) and (\ref{eq:8.14}) we get at the origin that
\begin{eqnarray}
\label{eq:8.15} && R^c_{i\overline{j}k\overline{\ell}} = 0 , \ \mbox{if} \ \{i,k\} \neq \{ j,\ell\}, \\
\label{eq:8.16} && R^c_{i\overline{i}i\overline{i}} = 2, \\
\label{eq:8.17} && R^c_{1\overline{2}2\overline{1}} = R^c_{2\overline{1}1\overline{2}}
= R^c_{3\overline{2}2\overline{3}} = R^c_{2\overline{3}3\overline{2}} =1,
\ \ R^c_{1\overline{3}3\overline{1}} = R^c_{3\overline{1}1\overline{3}} =-1,\\
\label{eq:4.18} && R^c_{1\overline{1}2\overline{2}}  = R^c_{1\overline{1}3\overline{3}} =R^c_{3\overline{3}1\overline{1}}= R^c_{3\overline{3}2\overline{2}} =0, \ \ R^c_{2\overline{2}1\overline{1}} = R^c_{2\overline{2}3\overline{3}} =1.
\end{eqnarray}
In other words, at the origin, the Chern curvature matrix is
\begin{equation}
\Theta = \left[ \begin{array}{ccc}   2dz_{1\overline{1}}+dz_{2\overline{2}} & dz_{2\overline{1}} &  -dz_{3\overline{1}}  \\ dz_{1\overline{2}}  & 2 dz_{2\overline{2}}  & dz_{3\overline{2}}  \\
- dz_{1\overline{3}}  & dz_{2\overline{3}}  & dz_{2\overline{2}}+ 2dz_{3\overline{3}}  \end{array} \right]
\end{equation}
Here we write $dz_{i\overline{j}}$ for $dz_i\wedge d\overline{z}_j$. In particular, $\sqrt{-1} \,\mathrm{tr} \, \Theta = 2\sqrt{-1}(dz_{1\overline{1}}+2dz_{2\overline{2}}+dz_{3\overline{3}}) = 2\tilde{\omega}$ as expected. The bisectional (Griffiths) curvature of $g$ is $R^c_{X\overline{X}Y\overline{Y}} = \,^tY \Theta (X, \overline{X}) \overline{Y}$, which is equal  to
\begin{eqnarray*}
&&  2\sum_{1 \leq i \leq 3} |X_iY_i|^2 + |X_2Y_1|^2+|X_2Y_3|^2 + 2\,\mbox{Re} (X_1\overline{X}_2\overline{Y}_1Y_2)  + \ 2\,\mbox{Re} (X_2\overline{X}_3\overline{Y}_2Y_3) - 2\,\mbox{Re} (X_1\overline{X}_3\overline{Y}_1Y_3) \\
& = & |X_2Y_1|^2+|X_2Y_3|^2 + |X_1\overline{Y}_1 + X_2\overline{Y}_2|^2 + |X_1\overline{Y}_1 - X_3\overline{Y}_3|^2 + |X_2\overline{Y}_2 + X_3\overline{Y}_3|^2 \ \geq \ 0.
\end{eqnarray*}
This indicates that the metric $g$ is non-K\"ahler since $X^3$ is not a Hermitian symmetric space. Note that when $X=Y$, the (Chern) holomorphic sectional curvature of $g$ is given by
\begin{equation*}
 R^c_{X\overline{X}X\overline{X}} =  |X_1X_2|^2 + |X_2X_3|^2 + (|X_1|^2  + |X_2|^2)^2 + (|X_1|^2 - |X_3|^2)^2 + (|X_2|^2 + |X_3|^2)^2,
\end{equation*}
which is positive for any $X\neq 0$, so $(X^3,g)$ has positive Chern holomorphic sectional curvature. Also, the first, second and third Chern Ricci form of $g$ are respectively
\begin{equation}\label{C_Ricci_Wallach}
Ric(\omega) = 2\tilde{\omega}, \ \ \  Ric^{(2)}(\omega) = 4\omega - \tilde{\omega}, \ \ \ Ric^{(3)} (\omega) = 2\tilde{\omega }.
\end{equation}
They  are all positive definite, with the first and third Ricci equal to each other.

\vspace{0.1cm}

Next let us verify that $(X^3,g)$ is balanced and BTP. First let us recall the formula for Bismut connection and curvature under natural frames. Let $(z_1, \ldots , z_n)$ be a local holomorphic coordinate on a Hermitian manifold $(M^n,g)$, and write $\varepsilon_i=\frac{\partial}{\partial z_i}$. Under the frame $\varepsilon$, which we view as a column vector, the Levi-Civita connection $\nabla$, Chern connection $\nabla^c$, and Bismut connection $\nabla^b$ are given by
$$ \nabla^c \varepsilon = \theta \varepsilon, \ \ \ \nabla^b \varepsilon = \theta^b \varepsilon, \ \ \ \nabla \varepsilon = \theta^{(1)} \varepsilon + \overline{\theta^{(2)}} \,\overline{ \varepsilon}. $$
Then it is easy to see that $\theta =\partial g g^{-1}$, where $g=(g_{i\overline{j}})$. Denote by $T$ the torsion tensor of $\nabla^c$, and write
$T(\varepsilon_i , \varepsilon_k) = \sum_j T^j_{ik} \varepsilon_j$, then we have
\begin{equation} \label{eq:8.20}
 T^j_{ik} = \sum_{\ell} (g_{k\overline{\ell},i } - g_{i\overline{\ell},k }) g^{\overline{\ell}j} .
 \end{equation}
Since $\nabla$ is torsion free, it holds that
$$ 2\langle \nabla_xy,z\rangle = x\langle y,z\rangle + y\langle x,z\rangle -z \langle x,y\rangle +\langle [x,y],z\rangle - \langle [y,z],x\rangle - \langle [x,z],y\rangle $$
for any vector fields $x$, $y$, $z$ on $M^n$, so under the frame $\varepsilon$ we have
\begin{equation} \label{eq:8.21}
 \theta^{(1)}_{ij} = \frac{1}{2}\sum_{k,\ell} \big( g_{k\overline{\ell},i } + g_{i\overline{\ell},k }\big)g^{\overline{\ell}j}  dz_k +  \frac{1}{2}\sum_{k,\ell} \big( g_{i\overline{\ell},\overline{k} } - g_{i\overline{k},\overline{\ell} }\big) g^{\overline{\ell}j} d\overline{z}_k.
 \end{equation}
By the relation $\theta^b=2\theta^{(1)} -\theta$, we get
\begin{equation} \label{eq:8.22}
 \theta^b_{ij} = \sum_{k,\ell} g_{k\overline{\ell},i } g^{\overline{\ell}j}  dz_k +  \sum_{r,k,\ell} g_{i\overline{r}}\overline{T^r_{k\ell} }  g^{\overline{\ell}j} d\overline{z}_k.
 \end{equation}
The BTP condition is given by
\begin{equation} \label{eq:8.23}
\nabla^bT=0 \ \Longleftrightarrow \  dT^j_{ik} =\sum_r \big( \theta^b_{ir} T^j_{rk} + \theta^b_{kr} T^j_{ir} - \theta^b_{rj} T^r_{ik} \big), \ %\forall \ i,j,k.
 \end{equation}
When $g_{i\overline{j}}=\delta_{ij}$ at the origin $0$, then by (\ref{eq:8.22}) the BTP condition at $0$  is given by \begin{eqnarray}
\frac{\partial}{\partial z_{\ell}} T^j_{ik} & = & \sum_r \big(   g_{\ell \overline{r},i} T^j_{rk} +   g_{\ell \overline{r},k} T^j_{ir} -  g_{\ell \overline{j},r} T^r_{ik}   \big) ,\label{eq:8.24} \\
 \frac{\partial}{\partial \overline{z}_{\ell}} T^j_{ik} & = &  \sum_r \big(   T^j_{ir} \overline{ T^k_{\ell r}} -   T^j_{kr} \overline{ T^i_{\ell r}} +  T^r_{ik} \overline{ T^r_{j\ell }}   \big)  .\label{eq:8.25}
\end{eqnarray}
Now let us check the BTP condition for our Wallach threefold $(X^3,g)$. At the origin, we have $g_{i\overline{j}}=\delta_{ij}$, and $g_{i\overline{j},k} =0$ except $g_{3\overline{2},1} =1$, so by (\ref{eq:8.20}) we know that all components of $T$ vanish except $T^2_{13}=1$. In particular, Gauduchon's torsion $1$-form $\eta=0$ as $\eta_k = \sum_i T^i_{ik}$, so $g$ is balanced.

For (\ref{eq:8.24}), the right hand side is zero because, for each of these three terms, one of the two factors is zero when $r$ is $2$ or not $2$. Its left hand side at $0$ is equal to $(g_{k\overline{j},i\ell} - g_{i\overline{j},k\ell})$, which is zero by the last equality in (\ref{eq:8.11}). For (\ref{eq:8.25}), the left hand side  at the origin is given by
$$  (g_{k\overline{j},i\overline{\ell}} - g_{i\overline{j},k\overline{\ell}}) - (g_{k\overline{2},i} - g_{i\overline{2},k} ) \, \overline{g_{j\overline{2},\ell} } .$$
When $\{i,k\}\neq \{j,\ell\}$, both sides of (\ref{eq:8.25}) are zero. The same is true when $i=j=k=\ell$, so we just need to check the $i\neq k$ and $\{ i,k\} =\{ j,\ell\}$ case. Assume that $i=j\neq k=\ell$. Then the left hand side of (\ref{eq:8.25}) at the origin is equal to
\[(g_{k\overline{j},i\overline{\ell}} - g_{i\overline{j},k\overline{\ell}}) - (g_{k\overline{2},i} - g_{i\overline{2},k} ) \, \overline{g_{j\overline{2},\ell} } =
\begin{cases}
-\delta_{i2},\ \text{if}\ \{i,k\}=\{1,2\}\ \text{or}\ \{2,3\},\\
1,\quad \ \ \text{if}\ \{i,k\}=\{1,3\},
\end{cases}\]
by the equalities \eqref{eq:8.13} and \eqref{eq:8.14}. In the mean time, the right hand side of (\ref{eq:8.25}) is equal to
$$  \sum_r ( T^i_{ir} \overline{T^k_{kr}} - |T^i_{kr}|^2 + |T^r_{ik}|^2) =
\begin{cases} -\delta_{i2},\ \text{if}\ \{i,k\}=\{1,2\}\ \text{or}\ \{2,3\},\\
1,\quad \ \ \text{if}\ \{i,k\}=\{1,3\}.\end{cases}$$
So (\ref{eq:8.25}) holds in this case. The $i=\ell \neq k=j$ case can be verified similarly. So the Wallach threefold $(X^3,g)$ is indeed non-K\"ahler, balanced and BTP.

In the Appendix, we will also verify that the Wallach threefold $(X^3,g)$ as a Riemannian manifold is Einstein (with constant Ricci curvature $3$) and has non-negative sectional curvature. The metric $g$ is a boundary point of the region of positively curved metrics on $X^3$ discovered by Wallach \cite{Wallach}.

\vspace{0.4cm}

\section{Balanced BTP threefolds of middle type}\label{mddtype3D}

In this section, we will discuss balanced BTP threefolds of middle type, which constitutes all balanced BTP threefolds except the Chern flat and Fano ones. Let $(M^3,g)$ be a compact, balanced BTP threefold of middle type, namely, its $B$ tensor has rank $2$. We want to analyze its geometric and topological structure.

\subsection{The local structure}\label{str}

Let $L$ be the kernel of the $B$ tensor. It is a complex line subbundle of the holomorphic tangent bundle $T_{\!M}$. First we show that $L$ is holomorphic.

\begin{lemma} \label{claim9.1}
$L$ is a holomorphic line bundle on $M^3$ satisfying $L^{\otimes 2}\cong {\mathcal O}_{\!M}$. It is actually a foliation.
\end{lemma}

\begin{proof}
In the statement of the lemma,  ${\mathcal O}_{\!M}$ denotes the trivial line bundle of $M$. Let $e$ be a special frame. Since the rank of the $B$ tensor is $2$, we have $a_1=a_2=a>0$ is a global constant. It follows from (\ref{eq:aithetab}) that $\theta^b_{13}=\theta_{23}^b=\theta^b_{33}=0$,  $\theta^b_{11}=\theta^b_{22}$ and $\theta^b_{12}+\theta^b_{21}=0$, which implies that
\[\gamma = a\!\begin{bmatrix} 0 & -(\varphi_3+\bar{\varphi}_3) & \overline{\varphi}_2 \\
                                      \varphi_3 + \bar{\varphi}_3 & 0 & - \bar{\varphi}_1 \\
                                       -\varphi_2 & \varphi_1 & 0 \end{bmatrix}\!\!, \ \,
\theta^b =  \begin{bmatrix} \alpha & \beta_0  & 0 \\ -\beta_0 & \alpha &  0  \\ 0 & 0 & 0 \end{bmatrix}\!\!, \ \,
\theta =  \begin{bmatrix} \alpha & \beta  &  - a\overline{\varphi}_2 \\ -\beta & \alpha &  a\overline{\varphi}_1  \\ a\varphi_2 & -a\varphi_1 & 0 \end{bmatrix}\!\! ,  \ \,
\tau = a \!\begin{bmatrix} \varphi_2\varphi_3 \\ \varphi_3\varphi_1  \\ 0  \end{bmatrix}\!\!,\]
where $\overline{\alpha}=-\alpha$, $\overline{\beta}_0=\beta_0$, and $\beta = \beta_0 +a\varphi_3 + a\overline{\varphi}_3$. This means that $\nabla^be_3=0$ for any special frame $e$. By the expression of the Chern connection matrix $\theta$, we have $\nabla^c_{\overline{i}}e_3=\sum_j \theta_{3j}(\bar{e}_i)e_j =0$, so $e_3$ is a local holomorphic vector field on $M^3$. Since $e_3$ is  a local section of $L$,  we know $L$ is a holomorphic line bundle.

To show that $L^{\otimes 2}={\mathcal O}_{\!M}$, let us assume that $\tilde{e}$ is another special frame. Since $e_3$ and $\tilde{e}_3$ are both sections of $L$ of unit length, $\tilde{e}_3=\rho e_3$ for some local smooth function $\rho$ satisfying $|\rho |=1$, while $\{ \tilde{e}_1, \tilde{e}_2\}$ is a change of $\{ e_1, e_2\}$ by a $U(2)$-valued local function $U$. From the definition of special frames, we have $\tilde{T}^1_{23}=\tilde{T}^2_{31}=a$, $T^1_{23}=T^2_{31}=a$, and all other torsion components are zero. Thus
\begin{eqnarray*}
&&  a=\tilde{T}^1_{23} = \rho \overline{U}_{11} U_{22} T^1_{23}+ \rho \overline{U}_{12} U_{21}T^2_{13} =a \rho (\overline{U}_{11}U_{22}-\overline{U}_{12}U_{21}), \\
&& a=\tilde{T}^2_{31} = \rho \overline{U}_{21}U_{12}T^1_{32}+ \rho \overline{U}_{22}U_{11}T^2_{31} =a \rho (U_{11} \overline{U}_{22}-U_{12}\overline{U}_{21}).
\end{eqnarray*}
This shows that $\rho =\overline{\rho}$, hence $\rho =\pm 1$. In other words, $\tilde{e}_3\otimes \tilde{e}_3 = e_3\otimes e_3 $ is independent of the choice of local special frames, hence can be defined  globally on $M^3$. This means that $L^{\otimes 2}={\mathcal O}_{\!M}$ is trivial.

To see that $L$ is actually a foliation, it suffices to show that $[e_3,\overline{e}_3]=0$. Since $T(e_i, \overline{e}_j)=0$, by the expression for $\theta$ we have
$$ [e_3,\overline{e}_3]=\nabla^c_{e_3}\overline{e}_3 - \nabla^c_{\overline{e}_3}e_3= \sum_k \big( \overline{\theta_{3k}(\overline{e}_3)} \,\overline{e}_k - \theta_{3k}(\overline{e}_3) \, e_3 \big) = 0.$$
This completes the proof of the lemma.
\end{proof}

In the proof of Lemma \ref{claim9.1}, we see that under any special frame $e$, the Bismut connection matrix $\theta^b$ takes a particularly simple form. It can be made diagonal after a constant unitary change of $\{ e_1, e_2\}$. Let
$ \tilde{e}_i=\sum_{j=1}^2 U_{ij}e_j$ for $ 1\leq i\leq 2$ and $\tilde{e}_3=-\sqrt{-1}e_3$, where
$$ U= \frac{1}{\sqrt{2}} \left[ \begin{array}{cc} 1 & \sqrt{-1}\\ \sqrt{-1} & 1 \end{array} \right], \ \ \  \ \ \mbox{then} \ \ \
U \left[ \begin{array}{cc} \alpha & \beta_0 \\ -\beta_0 & \alpha  \end{array} \right] U^{-1}
= \left[ \begin{array}{cc} \alpha -\sqrt{-1}\beta_0 & 0 \\ 0 & \alpha +\sqrt{-1}\beta_0   \end{array} \right]  .$$
That is, under the new frame $\tilde{e}$ the Bismut connection matrix is diagonal and the components of Chern torsion
become $T^1_{13}=-T^2_{23}=a$, where $\tilde{e}_3$ is also uniquely determined up to a sign.

\begin{definition} \label{def9.2}
Let $(M^3,g)$ be a compact balanced BTP threefold of middle type. A local unitary frame $e$ on $M^3$ is called an {\bf admissible frame},
if under $e$ the non-zero Chern torsion components are $T^1_{13}=-T^2_{23}=a>0$.
\end{definition}

From the discussion above, it is clear that the following hold:

\begin{lemma} \label{claim9.1b}
 Given a compact balanced BTP threefold of middle type, locally it always admits admissible frames. If both $e$ and $\tilde{e}$ are admissible frames, then $$ \mbox{either} \ \ \ \ \ \tilde{e}_1=\rho_1 e_1, \ \tilde{e}_2=\rho_2 e_2, \ \tilde{e}_3= e_3, \ \ \ \ \ \mbox{or} \ \  \ \ \ \tilde{e}_1=\rho_1 e_2, \ \tilde{e}_2=\rho_2 e_1,  \ \tilde{e}_3= -e_3, $$
 where  $\rho_1$ and $\rho_2$ are smooth local functions satisfying $|\rho_1|=|\rho_2|=1$.
 \end{lemma}

Also, under an admissible frame $e$, the Bismut connection matrix is diagonal and
\begin{equation} \label{eq:9.1b}
 \theta^b = \left[ \begin{array}{ccc} \theta^b_{11} & & \\ & \theta^b_{22} & \\ & & 0 \end{array} \right]\!\!,  \ \
 \gamma = a \!\left[ \begin{array}{ccc} \varphi_3\!-\!\overline{\varphi}_3 & 0 & \overline{\varphi}_1 \\
 & \overline{\varphi}_3\!-\!\varphi_3 & -\overline{\varphi}_2 \\
 -\varphi_1 & \varphi_2 & 0 \end{array} \right]\!\!, \ \
 \theta = \left[ \begin{array}{ccc} \theta_{11} & 0 & -a\overline{\varphi}_1 \\
0 & \theta_{22} & a\overline{\varphi}_2\\
a\varphi_1& -a\varphi_2 & 0 \end{array} \right]\!\!,
\end{equation}
where $\gamma = \theta^b-\theta$ is given by $\gamma_{ij} = \sum_k \big( T^j_{ik} \varphi_k - \overline{T^i_{jk}} \overline{\varphi}_k \big)$ under any unitary frame. We have
 \begin{equation} \label{eq:9.2c}
   \begin{cases}
   \theta_{11} =\theta^b_{11} - a(\varphi_3 - \overline{\varphi}_3), \\
   \theta_{22}  =\theta^b_{22} + a(\varphi_3 - \overline{\varphi}_3), \\
     d\varphi = -\,^t\!\theta \wedge \varphi + \tau \ = \ \left[ \begin{array}{c} -\theta_{11}\wedge \varphi_1 \\  -\theta_{22}\wedge \varphi_2 \\
     a (\varphi_{2\bar{2}}-\varphi_{1\bar{1}}) \end{array} \right] .
     \end{cases}
     \end{equation}
By $\Theta^b=d\theta^b-\theta^b \wedge \theta^b$, we see that the Bismut curvature matrix $\Theta^b$ under $e$ is also diagonal, satisfying $\Theta^b_{3\ast}=\Theta^b_{12}=0$. From $\Theta^b_{ij}=\sum_{k,\ell} R^b_{k\bar{\ell}i\bar{j}}\varphi_k\wedge \overline{\varphi}_{\ell}$, we know that under $e$ we have
$ R^b_{\ast \bar{\ast} 3\bar{\ast}} = R^b_{\ast \bar{\ast} 1\bar{2}} =0$. Recall that for any BTP metric and under any unitary frame, it holds that
\begin{equation} \label{eq:9.2b}
 R^b_{i\bar{j}k\bar{\ell}} = R^b_{k\bar{\ell}i\bar{j}}, \ \ \ \ \ R^b_{i\bar{j}k\bar{\ell}} - R^b_{k\bar{j}i\bar{\ell}} = \sum_s \big( T^j_{ks}\overline{ T^i_{\ell s}} +  T^{\ell}_{is}\overline{ T^k_{j s}} - T^j_{is}\overline{ T^k_{\ell s}} - T^{\ell}_{ks}\overline{ T^i_{j s}} - T^s_{ik}\overline{ T^s_{j\ell }} \big) \ \ \
 \end{equation}
for any indices $i$, $j$, $k$, $\ell$. The first equality is part (2) of Proposition 2.6 in \cite{ZhaoZ24}, and the second one is by Definition 2.3 and formula (2.4) in Proposition 2.5 of \cite{ZhaoZ24}. In particular, for our $(M^3,g)$ under an admissible frame $e$, by the first equation of (\ref{eq:9.2b}) and  the fact that $\Theta^b_{3\ast}=\Theta^b_{12}=0$, we know that the only possibly non-zero Bismut curvature components are $R^b_{1\bar{1}1\bar{1}}$, $R^b_{2\bar{2}2\bar{2}}$,  and $R^b_{1\bar{1}2\bar{2}}=R^b_{2\bar{2}1\bar{1}}$. By the second equality of (\ref{eq:9.2b}), for $1\leq i\neq k\leq 3$, we get
\begin{eqnarray*}
 R^b_{i\bar{i}k\bar{k}} - R^b_{k\bar{i}i\bar{k}} & = & \sum_s \big( |T^i_{ks}|^2 +  |T^{k}_{is}|^2 - 2Re(T^i_{is}\overline{ T^k_{k s}}) - |T^s_{ik}|^2 \big) \\
 & = & |T^i_{ki}|^2 + |T^k_{ik}|^2 - 2Re (T^i_{i3} \overline{ T^k_{k3}} )- ( |T^i_{ik}|^2 + |T^k_{ik}|^2) \ = \  -2Re (T^i_{i3} \overline{ T^k_{k3}} ).
 \end{eqnarray*}
From this we deduce
\begin{equation*}
R^b_{1\bar{1}2\bar{2}} = R^b_{2\bar{1}1\bar{2}} - 2Re (T^1_{13} \overline{ T^2_{23}} ) = 2a^2. \hspace{1cm}
\end{equation*}
Therefore
\begin{equation} \label{eq:9.3b}
 \Theta^b = \left[ \begin{array}{ccc} \Theta^b_{11} & & \\ & \Theta^b_{22} & \\ & & 0 \end{array} \right],  \ \ \ \ \Theta^b_{11} = R^b_{1\bar{1}1\bar{1}} \varphi_{1\bar{1}} + 2a^2 \varphi_{2\bar{2}}, \ \ \ \Theta^b_{22} = 2a^2\varphi_{1\bar{1}} + R^b_{2\bar{2}2\bar{2}} \varphi_{2\bar{2}}.
 \end{equation}
Taking trace, we get the formula for the (first) Bismut (or Chern) Ricci form:
\begin{equation} \label{eq:9.3c}
 Ric (\omega) = \sqrt{-1} \mbox{tr} \Theta^b = \sqrt{-1}( \lambda_1 \varphi_{1\bar{1}} + \lambda_2 \varphi_{2\bar{2}}), \ \ \ \ \ \lambda_1 = R^b_{1\bar{1}1\bar{1}} + 2a^2, \ \ \ \lambda_2 = R^b_{2\bar{2}2\bar{2}} + 2a^2.
 \end{equation}

\begin{lemma}  \label{remark9.2}
Let $(M^3,g)$ be a compact balanced BTP threefold of middle type. Then the restricted holonomy group $\mbox{Hol}_0^{\,b}(M)$ for the Bismut connection is equal to $ U(1)\!\times \!U(1) \!\times \!1$, which is an abelian group.
\end{lemma}

\begin{proof}
Let $e$ be an admissible frame. For any real tangent vectors $x$, $y$, the matrix $\Theta^b(x,y)$ is in the diagonal form
$\mathrm{diag} \{ \sqrt{-1}p, \sqrt{-1}q, 0\}$ where $ p, q \in {\mathbb R}$. By the holonomy theorem we know that the restricted holonomy group $\mbox{Hol}_0^{\,b}(M)$ is contained in $U(1)\!\times \!U(1) \!\times \!1$. Since $R^b_{1\bar{1}2\bar{2}}=R^b_{2\bar{2}1\bar{1}}=2a^2>0$, we know that both circle factors must be in  presence, so  $\mbox{Hol}_0^{\,b}(M)= U(1)\!\times \!U(1) \!\times \!1$.
\end{proof}

Let us denote by $\mbox{Hol}^{\,b}(M)\subset U(3)$ the global holonomy group for the Bismut connection. Note that here we did not claim that the global Bismut holonomy group is abelian. Also, in the formula (\ref{eq:9.3c}) the two eigenvalue functions $\lambda_1$ and $\lambda_2$ for the Ricci form may only be defined locally, as an admissible frame change may swap $e_1$ and $e_2$. However, if we denote by $s$ and $\sigma_2$ the first and second elementary symmetric functions of the eigenvalues of the Ricci form $Ric(\omega)$, or equivalently,
$$ 3 Ric(\omega )\wedge \omega^2 = s \,\omega^3, \ \ \ \ \ \ 3 (Ric(\omega ))^2\wedge \omega = \sigma_2 \,\omega^3, $$
then the functions $s$ (which is the Bismut (or Chern) scalar curvature) and $\sigma_2$ are globally defined on $M^3$. As a result, the set $\{ \lambda_1, \lambda_2\}$ is well-defined on $M^3$, since it is the set of two roots of the quadratic polynomial $\lambda^2-s\lambda+\sigma_2=0$.  Equivalently speaking, we have

\begin{remark}
Although $R^b_{1\bar{1}1\bar{1}}$ and $R^b_{2\bar{2}2\bar{2}}$ may only be local functions on $M^3$, the set $\{ R^b_{1\bar{1}1\bar{1}}, R^b_{2\bar{2}2\bar{2}}\}$ is globally defined on $M^3$, as they (after adding $2a^2$) are roots of the quadratic polynomial whose coefficients are global functions $s$ and $\sigma_2$. The entire Bismut (hence Chern and Riemannian) curvature are determined by $s$ and $\sigma_2$.
\end{remark}

Let $(M^3,g)$ be a compact balanced BTP threefold of middle type. Then the Bismut curvature is determined by two real-valued smooth functions $s$ and $\sigma_2$ on $M^3$. If it is locally homogeneous, then both $s$ and $\sigma_2$ are constants. Conversely, if both $s$ and $\sigma_2$ are constants, then as observed in \cite[Proposition 1.7]{PodestaZ}, $(M^3,g)$ is actually {\em Bismut Ambrose-Singer} (abbreviated as BAS, which means that the Bismut connection has parallel torsion and curvature), in particular it is locally homogeneous. For the convenience of readers we include the sketch of proof below.

\begin{lemma} [\cite{PodestaZ}] \label{lemmaPodestaZ}
Let $(M^3,g)$ be a compact balanced BTP threefold of middle type. Denote by $s$ and $\sigma_2$ the first and second elementary symmetric functions of the eigenvalues of the Chern Ricci form. If both  $s$ and $\sigma_2$ are constants, then the metric $g$ is BAS. In particular, in this case the universal cover of $(M^3,g)$ is a homogeneous Hermitian manifold.
\end{lemma}

\begin{proof}
Let $e$ be an admissible frame. The as shown above the only possibly non-zero components of the Bismut curvature are $R^b_{1\bar{1}1\bar{1}}$, $R^b_{2\bar{2}2\bar{2}}$, and $R^b_{1\bar{1}2\bar{2}} =R^b_{2\bar{2}1\bar{1}} =2a^2$, which are all constants. Since $\theta^b$ is diagonal, the condition  $\nabla^bR^b=0$ is equivalent to the following
$$ dR^b_{i\bar{j}k\bar{\ell}} = R^b_{i\bar{j}k\bar{\ell}} \big( \theta^b_{ii} - \theta^b_{jj} + \theta^b_{kk} - \theta^b_{\ell  \ell } \big) , \ \ \ \ \ \ \forall \ 1\leq i,j,k,\ell\leq 3. $$
If $i\neq j$ or $k\neq \ell$, then $R^b_{i\bar{j}k\bar{\ell}} =  0$ hence both sides are zero. If $i=j$ and $k=\ell$, then the parenthesis on the right is zero, while the left side is also zero as $R^b_{i\bar{i}k\bar{k}}$  is a constant. So $g$ is BAS.
\end{proof}

Next let us prove an interesting property enjoyed by compact balanced BTP threefolds of middle type.

\begin{proposition} \label{prop6.7}
Let $(M^3,g)$ be a compact balanced BTP threefold of middle type. Then $M^3$ does not admit any pluriclosed metric.
\end{proposition}

\begin{proof}
Let $e$ be an admissible frame in $M^3$, with dual coframe $\varphi$. Since $e_3$ is uniquely determined up to a sign, the $(1,1)$-form $\Phi = \varphi_3\wedge \overline{\varphi}_3=\varphi_{3\bar{3}}$ exists globally on $M^3$. By (\ref{eq:9.2c}), we know that locally we have $d\varphi_{1\bar{1}}=d\varphi_{2\bar{2}}=0$ and $d\varphi_3=a(\varphi_{2\bar{2}}-\varphi_{1\bar{1}})=-d\overline{\varphi}_3$, hence
\begin{eqnarray*}
&&  d\Phi \ = \  d\varphi_3\wedge (\varphi_3+\overline{\varphi}_3)  \ = \ a(\varphi_{2\bar{2}}-\varphi_{1\bar{1}}) \wedge (\varphi_3+\overline{\varphi}_3),\\
&& \partial \overline{\partial}\Phi \ = \  - d\partial\Phi \ = \ -d\big( a(\varphi_{2\bar{2}}-\varphi_{1\bar{1}}) \wedge \varphi_3\big)  \ = \ -a(\varphi_{2\bar{2}}-\varphi_{1\bar{1}}) \wedge d\varphi_3 \ = \ 2a^2 \varphi_{1\bar{1}}\varphi_{2\bar{2}}.
\end{eqnarray*}
Given any Hermitian metric $g_0$ on $M^3$ with K\"ahler form $\omega_0$, locally we have $\omega_0=\sqrt{-1} \sum_{i,j} h_{i\bar{j}}\varphi_{i\bar{j}}$, where $(h_{i\bar{j}})$ is positive definite. Let $f$ be the function on $M^3$ defined by $ \partial \overline{\partial}\Phi \wedge \omega_0=f\omega^3$, then by
$$ \partial \overline{\partial}\Phi \wedge \omega_0 =   2a^2\varphi_{1\bar{1}} \varphi_{2\bar{2}} \wedge \omega_0 = -\frac{1}{3}a^2h_{3\bar{3}} \,\omega^3, $$
we see that locally $f=-\frac{1}{3}a^2h_{3\bar{3}}$, hence $f<0$ everywhere on $M^3$. Therefore, $\int_M \partial \overline{\partial}\Phi \wedge \omega_0  <0$. In particular, $g_0$ cannot be pluriclosed as $M^3$ is compact. This completes the proof of the proposition.
\end{proof}

Recall that the Fino-Vezzoni Conjecture (\cite{FV2, FV16}) states that any compact complex manifold admitting both a balanced and a pluriclosed metric must be K\"ahlerian, while Streets-Tian Conjecture \cite{ST} states that any Hermitian-symplectic manifold must be K\"ahlerian. Supplementing Corollary 1.19 of \cite{ZhaoZ24}, we now have the following:

\begin{corollary} \label{cor6.8}
Both Fino-Vezzoni Conjecture and Streets-Tian Conjecture hold for all compact  BTP threefolds.
\end{corollary}

\begin{proof}
For balanced (but non-K\"ahler) BTP threefolds, the Fano case is already K\"ahlerian, while the middle type ones do not admit any pluriclosed metric by the above proposition, hence Fino-Vezzoni Conjecture holds, and the Streets-Tian Conjecture also holds as any Hermitian-symplectic metric is necessarily pluriclosed. We are only left with the Chern flat case. Streets-Tian Conjecture is known to be true for compact Chern flat manifolds in all dimensions by the work of Di Scala-Lauret-Vezzoni  \cite[Proposition 3.3]{DLV}, while in dimension $3$, both conjectures hold since any compact, non-K\"ahler Chern flat threefold $(M^3,g)$ cannot admit any pluriclosed metric, because $\sqrt{-1} \partial \overline{\partial} \omega = \,^t\!\tau \wedge \overline{\tau}$,
where $\tau$ is column vector of Chern torison $(2,0)$-forms. If $g_0$ is a pluriclosed metric on $M^3$, then one would have
$$ 0 =\int_M \sqrt{-1} \partial \overline{\partial} \omega \wedge \omega_0 = \int_M \,^t\!\tau \wedge \overline{\tau} \wedge \omega_0 >0, $$
which is a contradiction. Note that for the non-balanced case, both conjectures are valid in complex dimension $3$ by \cite[Corollary 1.19]{ZhaoZ24}.
\end{proof}

Therefore, combining Proposition \ref{prop6.7} and Corollary \ref{cor6.8}, we get the proof of Theorem \ref{thm1.2} stated in the introduction.

%\vspace{0.2cm}

\subsection{A Hermitian Lie algebra example}\label{example}

In \cite[Proposition 1.10]{ZhaoZ24}, we have seen the characterization of BTP Hermitian structures on nilpotent Lie algebras. In particular, in complex dimension $3$ there is only one (up to scaling the metric by a constant multiple) balanced BTP Hermitian nilpotent Lie algebra $N^3$ with nilpotent $J$ in the sense of \cite{CFGU}, whose structure equation is given by
\begin{equation} \label{nil}
d\varphi_1=d\varphi_2=0,\ \ \ \ \ \ \ d\varphi_3 = -a\varphi_{1\bar{1}} + a \varphi_{2\bar{2}},
\end{equation}
where $a>0$ is a constant, $\varphi$ is a unitary coframe, and as before we have abbreviated $\varphi_i\wedge \overline{\varphi}_j$ as  $\varphi_{i\bar{j}}$ for convenience. In this subsection, we will drop the nilpotency requirement and figure out all Hermitian Lie algebras in complex dimension $3$ that are balanced BTP of middle type. It turns that there are exactly two families of such Lie algebras, $A_{s,t}$ parameterized by $(s,t)\in {\mathbb R}^2$, and $B_{z,t}$ parameterized by $(z,t)\in {\mathbb C}\times {\mathbb R}$, such that at the origin they are just $N^3$, namely, $A_{0,0}=B_{0,0}=N^3$. Note that the two families have an overlap: $A_{0,t}=B_{0,t}$ for any $t$, and here we have assumed that the Lie algebra is unimodular, which is a necessary condition for the corresponding Lie group to admit a compact quotient. Other than $N^3$, these $A_{s,t}$ or $B_{z,t}$ are not nilpotent, instead they are all $3$-step solvable, which means that
the underlying real Lie algebra $\mathfrak{g}$ satisfies
\[\mathfrak{g}'=[\mathfrak{g},\mathfrak{g}] \neq 0,\quad \mathfrak{g}''=[\mathfrak{g}',\mathfrak{g}'] \neq 0,\quad \mathfrak{g}'''=[\mathfrak{g}'',\mathfrak{g}'']=0.\]
With the exception of $A_{-t,t}$, all others are not of Calabi-Yau type, namely, there is no non-trivial invariant holomorphic $3$-form.

Let ${\mathfrak g}$ be a Lie algebra of real dimension $2n$. Let $J$ be a complex structure on ${\mathfrak g}$, namely, a linear isomorphism satisfying $J^2=-id$ and the integrability condition
$$ [x,y]-[Jx,Jy] +J[Jx,y] + J[x,y] =0, \ \ \ \ \ \forall \ x,y \in {\mathfrak g}. \ \ \ \ \ $$
Let $g=\langle \cdot , \cdot \rangle$ be a metric (inner product) on ${\mathfrak g}$ compatible with $J$, that is, $\langle Jx,Jy\rangle =\langle x,y\rangle$, $\forall$  $x,y\in {\mathfrak g}$. We will call $({\mathfrak g}, J,g)$ a {\em Hermitian Lie algebra} (or Lie algebra equipped with a {\em Hermitian structure}). It corresponds to Lie groups equipped with a left-invariant complex structure and a compatible left-invariant metric. Write ${\mathfrak g}^{1,0}=\{ x-\sqrt{-1}Jx \mid x\in {\mathfrak g}\}$. A {\em unitary frame} $e$ of ${\mathfrak g}$ is a unitary basis of ${\mathfrak g}^{1,0}$. Its dual coframe is a basis $\varphi$ of the dual vector space $({\mathfrak g}^{1,0})^{\ast}$ satisfying $\varphi_i(e_j)=\delta_{ij}$ for all $1\leq i,j\leq n$. Following \cite{VYZ} and \cite{ChenZ25}, let us denote the structure constants by
$$ C^j_{ik} = \varphi_j( [e_i, e_k]), \ \ \ \ \ D^j_{ik} = \overline{\varphi}_i ( [\overline{e}_j, e_k]), \ \ \ \ \ 1\leq i,j,k\leq n. $$
Then under a unitary frame $e$ and its dual coframe $\varphi$, the structure equation becomes
\begin{equation} \label{structureLie}
d\varphi_i = -\frac{1}{2}\sum_{j,k} C^i_{jk} \,\varphi_j \wedge \varphi_k - \sum_{j,k} \overline{D^j_{ik}} \,\varphi_j \wedge \overline{\varphi}_k, \ \ \ \ \ \ \  1\leq i\leq n,
\end{equation}
or equivalently,
\begin{equation*} \label{strucgtureLie2}
[e_i,e_j]=\sum_k C^k_{ij}\,e_k, \ \ \ \ [e_i,\overline{e}_j] = \sum_k \big( \overline{D^i_{kj}} \,e_k - D^j_{ki} \,\overline{e}_k \big) , \ \ \ \ \ \ \ 1\leq i,j\leq n.
\end{equation*}
The Chern torsion components are given by
\begin{equation*} \label{torsionLie}
T^j_{ik} = - C^j_{ik} - D^j_{ik} + D^j_{ki},
\end{equation*}
while the entries of the Chern connection matrix $\theta$ under $e$ are
\begin{equation*} \label{ChernconnectionLie}
\theta_{ij} = \sum_k \big( D^j_{ik}\varphi_k - \overline{ D^i_{jk} } \overline{\varphi}_k \big) .
\end{equation*}
The unimodular condition for the Lie algebra ${\mathfrak g}$, meaning $\mbox{tr}(\mbox{ad}_x)=0$\, $\forall$ $x\in {\mathfrak g}$,
is characterized by
\begin{equation} \label{unimodular}
 {\mathfrak g} \,\mbox{ is unimodular} \ \Longleftrightarrow \ \sum_k (C^k_{ki}+D^k_{ki}) =0, \ \ \ \forall \ 1\leq i\leq n.
\end{equation}

Now suppose that ${\mathfrak g}$ is a real $6$-dimensional, unimodular Lie algebra, equipped with a Hermitian structure $(J,g)$ such that it is balanced BTP of middle type. Then we can choose a unitary frame $e$ on ${\mathfrak g}$ such that it is admissible. Denote by $\varphi$ its dual coframe. Write $\theta_{11}=\alpha -\overline{\alpha}$, $\theta_{22}=\beta -\overline{\beta}$, where
$$ \alpha = \sum_i x_i\varphi_i, \ \ \ \ \ \beta = \sum_i y_i\varphi_i. $$
From the structure equations (\ref{eq:9.2c}) and (\ref{structureLie}), we know that the only possibly non-zero components of $C$ and $D$ are
\begin{align*} \label{CandD}
C^1_{12}&=-x_2, &  C^1_{13}&=-x_3, & C^2_{12}&=y_1, & C^2_{23}&=-y_3,\\
D^1_{31}&=a,& D^2_{32}&=-a, & D^1_{1i}&=x_i, & D^2_{2i}&=y_i,
\end{align*}
for $1\leq i\leq 3$. Since ${\mathfrak g}$ is unimodular, by (\ref{unimodular}) for $i=1$ and $2$ we get $x_1=0$ and $y_2=0$. On the other hand, by (\ref{eq:9.2c}) and (\ref{eq:9.3b}) we have
$$ R^b_{1\bar{1}1\bar{1}} \varphi_{1\bar{1}} + 2a^2 \varphi_{2\bar{2}}  = \Theta^b_{11}= d\theta^b_{11}=d(\theta_{11}+a(\varphi_3-\overline{\varphi}_3)) = d\theta_{11} + 2a^2 (\varphi_{2\bar{2}}-\varphi_{1\bar{1}}). $$
Therefore,
$$ \lambda_1 \varphi_{1\bar{1}} =  (R^b_{1\bar{1}1\bar{1}}  + 2a^2)\varphi_{1\bar{1}} = d\theta_{11} = d(\alpha -\overline{\alpha}) =\partial \alpha + (\overline{\partial}\alpha - \partial \overline{\alpha} ) - \overline{\partial} \overline{\alpha} . $$
Hence $\partial \alpha =0$. By (\ref{eq:9.2c}) we have
$$ 0 = \partial \alpha =  \partial (x_2 \varphi_2 +x_3\varphi_3) = x_2 \varphi_2 \beta = x_2\varphi_2 (y_1\varphi_1 + y_3\varphi_3). $$
Thus $x_2y_1=x_2y_3=0$, which gives us
$$ \overline{\partial}\alpha = x_2 \overline{\partial}\varphi_2 + x_3 d\varphi_3 = - x_2 \varphi_2 \overline{\beta} + x_3 d\varphi_3
= ax_3 (\varphi_{2\bar{2}} - \varphi_{1\bar{1}}). $$
By $\overline{\partial}\alpha-\partial \overline{\alpha} = \lambda_1 \varphi_{1\bar{1}}$, we get $\lambda_1=x_3+\overline{x}_3=0$. By the same way, we conclude that  $y_1x_2=y_1x_3=0$ and $\lambda_2 = y_3+\overline{y}_3=0$. In summary, we have proved that
\begin{equation} \label{eq:xy}
x_1=y_2=0, \ \ \ x_2y_1=x_2y_3=0, \ \ \ y_1x_2=y_1x_3=0, \ \ \ \lambda_1=\lambda_2= x_3+\overline{x}_3=y_3+\overline{y}_3=0.
\end{equation}
Also, the (first) Chern (or Bismut) Ricci form is identically zero, and the only non-zero Bismut curvature components are $R^b_{1\bar{1}1\bar{1}}=R^b_{2\bar{2}2\bar{2}}=-2a^2$ and $R^b_{1\bar{1}2\bar{2}}=R^b_{2\bar{2}1\bar{1}}=2a^2$. To solve (\ref{eq:xy}), let us divide our discussion into the following two cases.

%\vspace{0.1cm}

{\bf Case 1.} When $x_2=y_1=0$. Write $x_3=\sqrt{-1}s$, $y_3=\sqrt{-1}t$, where $s,t\in {\mathbb R}$. The Hermitian Lie algebra $({\mathfrak g}, J,g)$ has structure equation:
\begin{equation} \label{Ast}
A_{s,t}: \ \ \ \left\{ \begin{array}{lll} d\varphi_1 = \sqrt{-1}s \,\varphi_1\wedge (\varphi_3 + \overline{\varphi}_3), \\
d\varphi_2 = \sqrt{-1}t \,\varphi_2\wedge (\varphi_3 + \overline{\varphi}_3), \\
d\varphi_3 = a(\varphi_{2\bar{2}} - \varphi_{1\bar{1}} ).  \end{array} \right.
\end{equation}

{\bf Case 2.} When $y_1\neq 0$. Then $x_2=x_3=0$ by (\ref{eq:xy}). Still write $y_3=\sqrt{-1}t$ for $t\in {\mathbb R}$, and write $y_1=z\in {\mathbb C}$, then the Hermitian Lie algebra $({\mathfrak g}, g,J)$ becomes:
\begin{equation} \label{Bzt}
B_{z,t}: \ \ \ \left\{ \begin{array}{lll} d\varphi_1 = 0, \\
d\varphi_2 = \varphi_2 \wedge (z\varphi_1 - \overline{z}\,\overline{\varphi}_1) +\sqrt{-1}t \,\varphi_2\wedge (\varphi_3 + \overline{\varphi}_3), \\
d\varphi_3 = a (\varphi_{2\bar{2}} - \varphi_{1\bar{1}} ).  \end{array} \right.
\end{equation}

If $x_2\neq 0$, then $y_1=y_3=0$ and we are in the situation which is isomorphic to the above one,  via the transformation $\varphi_1 \mapsto \varphi_2$, $\varphi_2\mapsto \varphi_1$, $\varphi_3 \mapsto -\varphi_3$. In summary, we have the following:

\begin{proposition} \label{LieHermitian}
Let ${\mathfrak g}$ be a real $6$-dimensional unimodular Lie algebra equipped with a Hermitian structure $(J,g)$. If $g$ is balanced BTP of middle type, then it is either $A_{s,t}$ given by (\ref{Ast}), or $B_{z,t}$ given by (\ref{Bzt}), where $t,s\in {\mathbb R}$ and $z\in {\mathbb C}$, and $a>0$ is a positive number. In both cases the Bismut (or Chern) Ricci form vanishes, and the non-zero Bismut curvature components are $R^b_{1\bar{1}1\bar{1}}=R^b_{2\bar{2}2\bar{2}}=-2a^2$ and $R^b_{1\bar{1}2\bar{2}}=R^b_{2\bar{2}1\bar{1}}=2a^2$.

Furthermore, $A_{0,0}=B_{0,0}=N^3$ given by (\ref{nil}), and $A_{0,t}=B_{0,t}$ for any $t$. For any $(s,t)\neq (0,0)$ and $(z,t)\neq (0,0)$, both $A_{s,t}$ and $B_{z,t}$ are not nilpotent but $3$-step solvable. Finally, $A_{s,t}$ admits an invariant global holomorphic $3$-form (namely, of Calabi-Yau type) if and only if $s+t=0$, and $B_{z,t}$ is of Calabi-Yau type if and only if $(z,t)=(0,0)$.
\end{proposition}

\begin{proof}
The only thing we need to clarify here is when $A_{s,t}$ or $B_{z,t}$ will be of Calabi-Yau type. In this case any compact quotient of the corresponding Lie group by a discrete subgroup will have trivial canonical line bundle. Given a Hermitian Lie algebra $({\mathfrak g},J,g)$, let $e$ be a unitary frame and $\varphi$ its dual coframe. Write $\Phi = \varphi_1\wedge \varphi_2 \wedge \cdots \wedge \varphi_n$. Then we have
$$ d\Phi = -\mbox{tr} (\theta) \wedge \Phi , $$
where $\theta $ is the matrix of the Chern connection under $e$. So ${\mathfrak g}$ admits an invariant holomorphic $n$-form iff $\overline{\partial} \Phi =0$ iff $\mbox{tr}(\theta)=0$. For our $6$-dimensional Lie algebra which is balanced BTP of middle type, this means $\alpha +\beta =0$, or equivalently, $y_1=x_2=x_3+y_3=0$. So $A_{s,t}$ will be of Calabi-Yau type iff $s+t=0$, while $B_{z,t}$ is never of Calabi-Yau type when $(z,t)\neq (0,0)$.
\end{proof}

\begin{remark}
Note that the conclusion above is only build upon the Lie algebra level. Although in general it is a challenging problem to determine whether a given solvable Lie group contains a uniform lattice, we can still briefly discuss the underlying real Lie algebra for the two families $A_{s,t}$ and $B_{z,t}$ with the existence of uniform lattices in the corresponding simply-connected Lie groups below.
\end{remark}

Let us write $\varphi_1=\phi_1+\sqrt{-1}\phi_2$,  $\varphi_2=\phi_3+\sqrt{-1}\phi_4$, $\varphi_3=\phi_5+\sqrt{-1}\phi_6$ into real and imaginary parts. Then the structure equations (\ref{Ast}) and (\ref{Bzt}) become
\begin{eqnarray}
&& A_{s,t}: \ \ \ \left\{ \begin{array}{lll} d\phi_1 = -2s \,\phi_2\wedge \phi_5, \ \ \ \ d\phi_2 = 2s \,\phi_1\wedge \phi_5, \\
d\phi_3 = -2t \,\phi_4\wedge \phi_5, \ \ \ \ d\phi_4 = 2t \,\phi_3\wedge \phi_5, \\
d\phi_5 = 0, \ \ \ \ d\phi_6 = 2a \,\phi_1\wedge \phi_2 - 2a\, \phi_3\wedge \phi_4,  \end{array} \right.  \label{Ast2} \\
&&
B_{z,t}: \ \ \ \left\{ \begin{array}{llll} d\phi_1 = d\phi_2 =  d\phi_5 = 0, \\
 d\phi_3 = 2v\, \phi_1\wedge \phi_4 + 2u \,\phi_2\wedge \phi_4 -  2t \,\phi_4\wedge \phi_5, \\
 d\phi_4 = -2v\, \phi_1\wedge \phi_3 - 2u \,\phi_2\wedge \phi_3 + 2t \,\phi_3\wedge \phi_5, \\
d\phi_6 = 2a \,\phi_1\wedge \phi_2 - 2a\, \phi_3\wedge \phi_4.  \end{array} \right.  \label{Bzt2}
\end{eqnarray}
Here $z=u+\sqrt{-1}v$ for $u,v \in \mathbb{R}$.
Denote by $\{ \varepsilon_1, \ldots ,  \varepsilon_6\}$ the real basis of the Lie algebra dual to $\phi$, then the non-trivial Lie brackets for $A_{s,t}$ are given by
\begin{eqnarray}
&&  -[\varepsilon_1, \varepsilon_2] = [\varepsilon_3, \varepsilon_4] = 2a\, \varepsilon_6,  \nonumber \\
&& [\varepsilon_1, \varepsilon_5] = -2s\, \varepsilon_2 , \ \ \ \ \ [\varepsilon_2, \varepsilon_5] = 2s\, \varepsilon_1, \\
&& [\varepsilon_3, \varepsilon_5] = -2t\, \varepsilon_4 , \ \ \ \ \ [\varepsilon_4, \varepsilon_5] = 2t\, \varepsilon_3. \nonumber
\end{eqnarray}
Similarly, the the non-trivial Lie brackets for $B_{z,t}$ are given by %(where $z=u+iv$)
\begin{eqnarray}
&&  -[\varepsilon_1, \varepsilon_2] = [\varepsilon_3, \varepsilon_4] = 2a\, \varepsilon_6,  \nonumber \\
&& [\varepsilon_1, \varepsilon_3] = 2v\, \varepsilon_4 , \ \ \ \ \ [\varepsilon_1, \varepsilon_4] = -2v\, \varepsilon_3, \\
&& [\varepsilon_2, \varepsilon_3] = 2u\, \varepsilon_4 , \ \ \ \ \ [\varepsilon_2, \varepsilon_4] = -2u\, \varepsilon_3. \nonumber \\
&& [\varepsilon_5, \varepsilon_3] = 2t\, \varepsilon_4 , \ \ \ \ \ [\varepsilon_5, \varepsilon_4] = -2t\, \varepsilon_3. \nonumber
\end{eqnarray}
Now let us analyse the underlying Lie algebra of $A_{s,t}$ and $B_{z,t}$ for $(s,t) \neq (0,0)$ and $(z,t) \neq (0,0)$ up to isomorphism.
It is easy to verify that, for any $t\neq 0$ and $s \neq 0$, $B_{z,t} \cong B_{0,t} \cong A_{0,t} \cong A_{0,1}
\cong {\mathfrak a} \cong A_{1,0} \cong A_{s,0}$ as isomorphic Lie algebras, for any $z\neq 0$, $B_{z,0} \cong B_{1,0} \cong {\mathfrak b}$,
and for the remaining cases, that is, $st\neq 0$, $A_{s,t} \cong {\mathfrak c}_s$, where the non-trivial Lie brackets for Lie algebras ${\mathfrak a}$, ${\mathfrak b}$ and ${\mathfrak c}_s$ for $s\neq 0$ are given by $-[\varepsilon_1, \varepsilon_2] = [\varepsilon_3, \varepsilon_4] = \varepsilon_6$ and the following equations respectively
\begin{eqnarray}
{\mathfrak a}: &&  [\varepsilon_3, \varepsilon_5] = -\varepsilon_4, \ \ \ [\varepsilon_4, \varepsilon_5] =  \varepsilon_3;   \nonumber \\
{\mathfrak b}: &&  [\varepsilon_2, \varepsilon_3] = \varepsilon_4, \ \ \ [\varepsilon_2, \varepsilon_4] =  -\varepsilon_3;   \nonumber \\
{\mathfrak c}_s: &&  [\varepsilon_1, \varepsilon_5] = -s\, \varepsilon_2, \ \ \ [\varepsilon_2, \varepsilon_5] =  s\, \varepsilon_1, \ \ \  [\varepsilon_3, \varepsilon_5] = - \varepsilon_4 , \ \ \  [\varepsilon_4, \varepsilon_5] =  \varepsilon_3. \nonumber
\end{eqnarray}
The commutators of ${\mathfrak a}$ and ${\mathfrak b}$ are both $3$-dimensional, which are given by ${\mathfrak a}' = {\mathfrak b}'={\mathbb R}\{ \varepsilon_3, \varepsilon_4, \varepsilon_6\} $, but their centers are of different dimensions, that is, ${\mathfrak z}({\mathfrak a})={\mathbb R}\{ \varepsilon_6\}$ while ${\mathfrak z}({\mathfrak b})={\mathbb R}\{ \varepsilon_5, \varepsilon_6\}$. In the mean time, the commutator $\mathfrak{c}_s'$ of ${\mathfrak c}_s$ and the commutator of $\mathfrak{c}_s'$ are given by
$$ {\mathfrak c}_s' = {\mathbb R}\{ \varepsilon_1,  \varepsilon_2, \varepsilon_3,  \varepsilon_4,\varepsilon_6\},\quad
 {\mathfrak c}_s'' =[{\mathfrak c}_s',{\mathfrak c}_s']= {\mathbb R}\{\varepsilon_6\},$$
where the former is clearly $5$-dimensional. Therefore ${\mathfrak a}$, ${\mathfrak b}$, and ${\mathfrak c}_s$ are different Lie algebras. 
For the isomorphic classes of $\mathfrak{c}_s$, we have the following

\begin{lemma}
For any $s,s'\neq0$, the Lie algebras ${\mathfrak c}_s \cong {\mathfrak c}_{s'}$ if and only if $s=s'$ or $s'=\frac{1}{s}$.
\end{lemma}

\begin{proof}
Let $\{\varepsilon_i\}_{i=1}^6$ be the `standard' basis of the Lie algebra $\mathfrak{c}_s$ described in the above paragraphes. Consider a new basis $\{\tilde{\varepsilon}_i\}_{i=1}^6$ of $\mathfrak{c}_s$ given by
\[\tilde{\varepsilon}_1=\varepsilon_3,\quad \tilde{\varepsilon}_2=\varepsilon_4,\quad \tilde{\varepsilon}_3=\varepsilon_1,\quad \tilde{\varepsilon}_4=\varepsilon_2,\quad
\tilde{\varepsilon}_5=\frac{\varepsilon_5}{s},\quad \tilde{\varepsilon}_6=-\varepsilon_6.\]
It is not difficult to see that this new basis is a `standard' basis for the Lie algebra $\mathfrak{c}_{\frac{1}{s}}$. So ${\mathfrak c}_s$ and $\mathfrak{c}_{\frac{1}{s}}$ are isomorphic.
%So we have shown that
%when $s=s'$ or $s'=-s=\pm 1$ or $s'=\frac{1}{s}$ holds, the two Lie algebra ${\mathfrak c}_s$ and ${\mathfrak c}_{s'}$ are isomorphic.
Conversely, suppose that $f:{\mathfrak c}_s \rightarrow {\mathfrak c}_{s'}$ is an isomorphism with $s \neq s'$, which can be formulated as
\[f\begin{pmatrix} \varepsilon_1 \\ \varepsilon_2 \\ \varepsilon_3 \\ \varepsilon_4 \\ \varepsilon_6 \\ \varepsilon_5 \end{pmatrix}
=\begin{pmatrix} a_{11} & a_{12} & a_{13} & a_{14} & b_1 & 0\\
                 a_{21} & a_{22} & a_{23} & a_{24} & b_2 & 0\\
                 a_{31} & a_{32} & a_{33} & a_{34} & b_3 & 0\\
                 a_{41} & a_{42} & a_{43} & a_{44} & b_4 & 0\\
                    0    &     0   &     0   &     0   &  c  & 0 \\
                 d_{1}  & d_{2}  & d_{3}  & d_{4}  & d_6 & d\\
                 \end{pmatrix}
\begin{pmatrix} \varepsilon'_1 \\ \varepsilon'_2 \\ \varepsilon'_3 \\ \varepsilon'_4 \\ \varepsilon'_6 \\ \varepsilon'_5 \end{pmatrix},\]
where $\det(a_{ij})_{1\leq i,j \leq 4} \neq 0$, $c \neq 0$ and $d \neq 0$. The block pattern for the matrix above is due to the fact that  $f({\mathfrak c}_s') = {\mathfrak c}_{s'}'$ and $f({\mathfrak c}_s'') = {\mathfrak c}_{s'}''$. As $f$ preserves the Lie bracket, we have
$-[f(\varepsilon_1),f(\varepsilon_2)]=[f(\varepsilon_3),f(\varepsilon_4)]=f(\varepsilon_6)$, which implies
\begin{equation}\label{bck_12346}
c= \begin{vmatrix} a_{11} & a_{12} \\ a_{21} & a_{22} \\ \end{vmatrix} - \begin{vmatrix} a_{13} & a_{14} \\ a_{23} & a_{24} \end{vmatrix}
=\begin{vmatrix} a_{33} & a_{34} \\ a_{43} & a_{44} \\ \end{vmatrix} - \begin{vmatrix} a_{31} & a_{32} \\ a_{41} & a_{42} \end{vmatrix}.
\end{equation}
By $[f(\varepsilon_1),f(\varepsilon_5)]=-s f(\varepsilon_2)$ and $[f(\varepsilon_2),f(\varepsilon_5)]=sf(\varepsilon_1)$, we get the equalities (I):
\begin{align*}
a_{21}&=-\frac{a_{12}ds'}{s},& a_{22}&=\frac{a_{11}ds'}{s},& a_{23}&=-\frac{a_{14}d}{s}, & a_{24}&=\frac{a_{13}d}{s},
& b_2&=\frac{1}{s}\left(\begin{vmatrix}a_{11} & a_{12} \\ d_1 & d_2 \end{vmatrix}-\begin{vmatrix}a_{13} & a_{14} \\ d_3 & d_4 \end{vmatrix}\right),\\
a_{11}&=\frac{a_{22}ds'}{s},& a_{12}&=-\frac{a_{21}ds'}{s}, & a_{13}&=\frac{a_{24}d}{s},& a_{14}&=-\frac{a_{23}d}{s},
& b_1&=\frac{1}{s}\left(\begin{vmatrix}a_{23} & a_{24} \\ d_3 & d_4 \end{vmatrix}-\begin{vmatrix}a_{21} & a_{22} \\ d_1 & d_2 \end{vmatrix}\right).
\end{align*}
Similarly, by $[f(\varepsilon_3),f(\varepsilon_5)]=- f(\varepsilon_4)$ and $[f(\varepsilon_4),f(\varepsilon_5)]=f(\varepsilon_3)$ we get
the equalities (II):
\begin{align*}
a_{41}&=-a_{32}ds',& a_{42}&=a_{31}ds',& a_{43}&=-a_{34}d, & a_{44}&=a_{33}d,
& b_4&=\begin{vmatrix}a_{31} & a_{32} \\ d_1 & d_2 \end{vmatrix}-\begin{vmatrix}a_{33} & a_{34} \\ d_3 & d_4 \end{vmatrix},\\
a_{31}&=a_{42}ds',& a_{32}&=-a_{41}ds', & a_{33}&=a_{44}d,& a_{34}&=-a_{43}d,
& b_3&=\begin{vmatrix}a_{43} & a_{44} \\ d_3 & d_4 \end{vmatrix}-\begin{vmatrix}a_{41} & a_{42} \\ d_1 & d_2 \end{vmatrix}.
\end{align*}
Furthermore, by $[f(\varepsilon_1),f(\varepsilon_3)]=0$ and $[f(\varepsilon_2),f(\varepsilon_3)]=0$ we obtain
\begin{equation}\label{bck_1323}
\begin{vmatrix}a_{13} & a_{14} \\ a_{33} & a_{34} \end{vmatrix}=\begin{vmatrix}a_{11} & a_{12} \\ a_{31} & a_{32} \end{vmatrix},\quad
\begin{vmatrix}a_{23} & a_{24} \\ a_{33} & a_{34} \end{vmatrix}=\begin{vmatrix}a_{21} & a_{22} \\ a_{31} & a_{32} \end{vmatrix}.
\end{equation}
Similarly, $[f(\varepsilon_1),f(\varepsilon_4)]=0$ and $[f(\varepsilon_2),f(\varepsilon_4)]=0$ are equivalent to
\begin{equation}\label{bck_1424}
\begin{vmatrix}a_{13} & a_{14} \\ a_{43} & a_{44} \end{vmatrix}=\begin{vmatrix}a_{11} & a_{12} \\ a_{41} & a_{42} \end{vmatrix},\quad
\begin{vmatrix}a_{23} & a_{24} \\ a_{43} & a_{44} \end{vmatrix}=\begin{vmatrix}a_{21} & a_{22} \\ a_{41} & a_{42} \end{vmatrix}.
\end{equation}
Note that the first four columns of equalities (I) and (II) give us
\begin{eqnarray}
&& ( \frac{d^2s'^2}{s^2} -1)  \begin{pmatrix}a_{11} & a_{12} \\ a_{21} & a_{22} \end{pmatrix} \ \ =  \ \   (\frac{d^2}{s^2} -1 )  \begin{pmatrix}a_{13} & a_{14} \\ a_{23} & a_{24} \end{pmatrix} \ \ =  \ \  0,   \label{eq:5.19}\\
&& (d^2s'^2 -1)  \begin{pmatrix}a_{31} & a_{32} \\ a_{41} & a_{42} \end{pmatrix} \ \ =  \ \ ( d^2 -1)  \begin{pmatrix}a_{33} & a_{34} \\ a_{43} & a_{44} \end{pmatrix} \ \ =  \ \ 0 ,  \label{eq:5.20} \\
&& \begin{vmatrix} a_{11} & a_{12} \\ a_{21} & a_{22} \\ \end{vmatrix} \ \, = \ \, \frac{ds'}{s}(a_{11}^2 + a_{12}^2), \ \ \ \ \ \begin{vmatrix} a_{13} & a_{14} \\ a_{23} & a_{24} \\ \end{vmatrix} \ \, = \ \, \frac{d}{s}(a_{13}^2 + a_{14}^2), \label{eq:5.21} \\
&& \begin{vmatrix} a_{31} & a_{32} \\ a_{41} & a_{42} \\ \end{vmatrix} \ \, = \ \, ds' (a_{31}^2 + a_{32}^2), \ \ \ \ \ \begin{vmatrix} a_{33} & a_{34} \\ a_{43} & a_{44} \\ \end{vmatrix} \ \, = \ \, d(a_{33}^2 + a_{34}^2).  \label{eq:5.22}
\end{eqnarray}
We will divide the discussion into two cases depending on whether
the matrix $\begin{pmatrix}a_{33} & a_{34} \\ a_{43} & a_{44} \end{pmatrix} $ is zero or not.

\noindent {\bf Case 1.} If $\begin{pmatrix}a_{33} & a_{34} \\ a_{43} & a_{44} \end{pmatrix} \neq 0$.

In this case the second equality of (\ref{eq:5.20}) implies that $d^2=1$
%and $\begin{vmatrix}a_{33} & a_{34} \\ a_{43} & a_{44} \end{vmatrix}=d(a_{33}^2 + a_{34}^2) \neq 0$.
Let us assume $d=1$ as the case $d=-1$ can be argued similarly. Hence $\begin{vmatrix}a_{33} & a_{34} \\ a_{43} & a_{44} \end{vmatrix} = d(a_{33}^2+a_{34}^2)>0$. We make the following claim:
\vspace{0.2cm}

\noindent {\em Claim 1:}  $\begin{pmatrix}a_{13} & a_{14} \\ a_{23} & a_{24} \end{pmatrix} \neq 0$
and $\begin{pmatrix}a_{31} & a_{32} \\ a_{41} & a_{42} \end{pmatrix} \neq 0$.

If on the contrary that $\begin{pmatrix}a_{13} & a_{14} \\ a_{23} & a_{24} \end{pmatrix}=0$. Then we have $
\begin{vmatrix}a_{11} & a_{12} \\ a_{21} & a_{22} \end{vmatrix} \neq 0$ since $\det(a_{ij}) \neq 0$.
By the first equality of (\ref{eq:5.19}) we conclude that $(\frac{ds'}{s})^2=(\frac{s'}{s})^2=1$, hence $s'=-s$ as we have assumed that  $s \neq s'$ at the very beginning. The first equality of (\ref{eq:5.21}) gives us
$$
\begin{vmatrix}a_{11} & a_{12} \\ a_{21} & a_{22} \end{vmatrix}=  -(a_{11}^2+a_{12}^2) <0.
$$
On the other hand, because $a_{13}=a_{14}=a_{23}=a_{24}=0$,  the equalities \eqref{bck_1323} and \eqref{bck_1424} become
\[\begin{vmatrix}a_{11} & a_{12} \\ a_{31} & a_{32} \end{vmatrix}=\begin{vmatrix}a_{21} & a_{22} \\ a_{31} & a_{32} \end{vmatrix}=0,\quad \quad \quad
\begin{vmatrix}a_{11} & a_{12} \\ a_{41} & a_{42} \end{vmatrix}=\begin{vmatrix}a_{21} & a_{22} \\ a_{41} & a_{42} \end{vmatrix}=0.\]
Since vectors $(a_{11},a_{12})$ and $(a_{21},a_{22})$ are linearly independent, we conclude that
$\begin{pmatrix}a_{31} & a_{32} \\ a_{41} & a_{42} \end{pmatrix} =0$. But then
the equality \eqref{bck_12346} would give us
\[c=\begin{vmatrix}a_{11} & a_{12} \\ a_{21} & a_{22} \end{vmatrix}=-(a_{11}^2+a_{12}^2)<0,\quad \quad
c=\begin{vmatrix} a_{33} & a_{34} \\ a_{43} & a_{44} \end{vmatrix}=a_{33}^2+a_{34}^2>0,\]
which is a contradiction. So we must have $\begin{pmatrix}a_{13} & a_{14} \\ a_{23} & a_{24} \end{pmatrix} \neq 0$. Similarly, we must have
 $\begin{pmatrix}a_{31} & a_{32} \\ a_{41} & a_{42} \end{pmatrix} \neq 0$ as well, and Claim 1 is proved.
\vspace{0.2cm}

\noindent {\em Claim 2:}  $\begin{pmatrix}a_{11} & a_{12} \\ a_{21} & a_{22} \end{pmatrix} \neq 0$.

Assume on the contrary that the above matrix is zero. Then by \eqref{bck_1323} and \eqref{bck_1424} we get
$$ \begin{vmatrix} a_{13} & a_{14} \\ a_{33} & a_{34} \\ \end{vmatrix} = \begin{vmatrix} a_{13} & a_{14} \\ a_{43} & a_{44} \\ \end{vmatrix} = \begin{vmatrix} a_{23} & a_{24} \\ a_{33} & a_{34} \\ \end{vmatrix} = \begin{vmatrix} a_{23} & a_{24} \\ a_{43} & a_{44} \\ \end{vmatrix} =0.
$$
Since the vectors $(a_{33},a_{34})$ and $(a_{43},a_{44})$ are linearly independent, we conclude that $\begin{pmatrix}a_{13} & a_{14} \\ a_{23} & a_{24} \end{pmatrix} =0$, contradicting to Claim 1. So we must have  $\begin{pmatrix}a_{11} & a_{12} \\ a_{21} & a_{22} \end{pmatrix} \neq 0$, and Claim 2 is proved.

By the claims and the equalities (\ref{eq:5.19}), (\ref{eq:5.20}), we conclude that $s^2=s'^2=1$. Since $s\neq s'$, we know that $(s,s')$ is either $(1,-1)$ or $(-1,1)$. The equality \eqref{bck_12346} now takes the following form
\[\begin{split}
c &=\begin{vmatrix} a_{11} & a_{12} \\ a_{21} & a_{22} \\ \end{vmatrix} - \begin{vmatrix} a_{13} & a_{14} \\ a_{23} & a_{24} \end{vmatrix}
=-(a^2_{11}+a_{12}^2)-\frac{1}{s}(a^2_{13}+a^2_{14}), \\
c &= \begin{vmatrix} a_{33} & a_{34} \\ a_{43} & a_{44} \\ \end{vmatrix} - \begin{vmatrix} a_{31} & a_{32} \\ a_{41} & a_{42} \end{vmatrix}
=(a_{33}^2 + a_{34}^2) -s'(a_{31}^2+a_{32}^2).
\end{split}\]
If $(s,s')=(1,-1)$, then the above two formula would dictate that $c<0$ and $c>0$ at the same time, which is absurd. So we must have $s=-1$ and $s'=1$.
Next let us write
\[r_1 = \sqrt{a_{11}^2 + a_{12}^2},\quad r_2 = \sqrt{a_{13}^2+a_{14}^2},\quad r_3=\sqrt{a_{31}^2+a_{32}^2},\quad r_4=\sqrt{a_{33}^2+a_{34}^2},\]
where $r_1,r_2,r_3,r_4$ are positive numbers. Then the four blocks of the matrix $(a_{ij})_{1\leq i,j\leq 4}$ can be reformulated as
\begin{align*}
\begin{pmatrix} a_{11} & a_{12} \\ a_{21} & a_{22}\end{pmatrix} &= \begin{pmatrix} a_{11} & a_{12} \\ a_{12} & -a_{11} \end{pmatrix}
=r_1 \begin{pmatrix} p_1 & p_2 \\ p_2 & -p_1 \end{pmatrix}, &
\begin{pmatrix} a_{13} & a_{14} \\ a_{23} & a_{24} \end{pmatrix} &= \begin{pmatrix} a_{13} & a_{14} \\ a_{14} & -a_{13} \end{pmatrix}
=r_2 \begin{pmatrix} p_3 & p_4 \\ p_4 & -p_3 \end{pmatrix}, \\
\begin{pmatrix} a_{31} & a_{32} \\ a_{41} & a_{42} \end{pmatrix} &= \begin{pmatrix} a_{31} & a_{32} \\ -a_{32} & a_{31} \end{pmatrix}
=r_3 \begin{pmatrix} p_5 & p_6 \\ -p_6 & p_5 \end{pmatrix}, &
\begin{pmatrix} a_{33} & a_{34} \\ a_{43} & a_{44} \end{pmatrix} &= \begin{pmatrix} a_{33} & a_{34} \\ -a_{34} & a_{33} \end{pmatrix}
=r_4 \begin{pmatrix} p_7 & p_8 \\ -p_8 & p_7 \end{pmatrix},
\end{align*}
where $p_1^2+p_2^2=p_3^2+p_4^2=p_5^2+p_6^2=p_7^2+p_8^2 =1$. The equality \eqref{bck_1323} then implies that
\[r_1 r_3(p_1 p_6 - p_2 p_5)=r_2 r_4(p_3 p_8 - p_4 p_7 ),\quad r_1 r_3(p_2 p_6 + p_1 p_5)=r_2 r_4(p_4 p_8 + p_3 p_7 ). \]
Square both equalities above and add them up,  we get $r^2_1 r^2_3 = r^2_2 r^2_4$,
which leads to $r_1 r_3 = r_2 r_4$. On the other hand, the equality \eqref{bck_12346} now says that
\[ c=r_2^2 -r_1^2 = r_4^2 -r_3^2.\]
Note that the Lie algebra isomorphism $f$ maps the center to the center and $c$ is the coefficient there, so $c\neq 0$. If $c>0$, then $r_2>r_1>0$ and $r_4>a_3>0$, which would lead to $r_2r_4>r_1r_3$. Similarly when $c<0$ we get $r_2r_4<r_1r_3$. This contradiction shows that our Case 1 cannot occur.

\vspace{0.2cm}

\noindent {\bf Case 2.} If $\begin{pmatrix}a_{33} & a_{34} \\ a_{43} & a_{44} \end{pmatrix} = 0$.

In this case by $\det(a_{ij})\neq 0$ we get $\begin{vmatrix}a_{13} & a_{14} \\ a_{23} & a_{24} \end{vmatrix} \neq 0$ and $
\begin{vmatrix}a_{31} & a_{32} \\ a_{41} & a_{42} \end{vmatrix} \neq 0$.
The equalities (\ref{eq:5.19}) and (\ref{eq:5.20}) imply that  $(ds')^2=1$ and $(\frac{d}{s})^2=1$, that is, $d^2=s^2=\frac{1}{s'^2}$.
The equalities \eqref{bck_1323} and \eqref{bck_1424} now take the following form
\[\begin{vmatrix}a_{11} & a_{12} \\ a_{31} & a_{32} \end{vmatrix}=\begin{vmatrix}a_{21} & a_{22} \\ a_{31} & a_{32} \end{vmatrix}=0,\quad \quad
\begin{vmatrix}a_{11} & a_{12} \\ a_{41} & a_{42} \end{vmatrix}=\begin{vmatrix}a_{21} & a_{22} \\ a_{41} & a_{42} \end{vmatrix}=0.\]
Since the vectors $(a_{31}, a_{32})$ and $(a_{41},a_{42})$ are linearly independent to each other, the above equalities imply that
$\begin{pmatrix}a_{11} & a_{12} \\ a_{21} & a_{22} \end{pmatrix} =0$. Hence by the equality \eqref{bck_12346} we get
\[c=-\begin{vmatrix}a_{13} & a_{14} \\ a_{23} & a_{24} \end{vmatrix} = - \frac{d}{s}(a_{13}^2+a_{14}^2),
\quad c=-\begin{vmatrix}a_{31} & a_{32} \\ a_{41} & a_{42} \end{vmatrix} = -ds'(a_{31}^2+a_{32}^2).\]
This indicates that $\frac{d}{s}$ and $ds'$ have the same sign, which leads to $s'=\frac{1}{s}$. This completes the proof of the lemma.
\end{proof}

Based on the above lemma, we may assume that the parameter $s$ for $\mathfrak{c}_s$ takes values in $[-1,0)\cup (0,1]$.

\begin{remark}
The kernel bundle $L$ of the tensor $B$ is clearly holomorphically trivial over $A_{s,t}$ and $B_{z,t}$. By
the theory developed in the forthcoming subsection \ref{double} and Section \ref{highD},
especially Proposition \ref{propVaisman}, Proposition \ref{propinverse} and Theorem \ref{thm6.3},
it shows that the real Lie algebra $\mathbb{R} \times \mathfrak{h}_5$ of $A_{0,0}=B_{0,0}$ and
 $\mathfrak{a}$, $\mathfrak{b}$,  $\mathfrak{c}_s$ can also serve as the underlying Lie algebras
of Vaisman unimodular Lie algebras of splitting type in complex dimension $3$, named Vaisman companion of $A_{s,t}$ and $B_{z,t}$.
Here $\mathfrak{h}_5$ is the 5-dimensional Lie algebra of the Heisenberg group. It turns out that
\[\mathfrak{a}=\mathbb{R} \ltimes_{D_0} \mathfrak{h}_5\quad \text{and}
\quad \mathfrak{c}_s = \mathbb{R} \ltimes_{D_{-s}} \mathfrak{h}_5,\]
where $s \in [-1,0)\cup (0,1]$, fit into the case (i) of \cite[Proposition 6.6]{AO}, while $\,\mathfrak{b}=\mathbb{R} \times \mathfrak{s}_5\,$
belongs to the case (ii) of the same proposition, where A. Andrada and M. Origlia \cite{AO}
have classified Vaisman unimodular solvable Lie algebras, especially the $6$-dimensional case,
and the definition of the semidirect product above has been given in \cite[the proof of Lemma 2.4]{AO}.
The proposition also shows that the simply-connected Lie groups corresponding to the Lie algebras
$\mathbb{R} \times \mathfrak{h}_5$, $\mathfrak{a}$, $\mathfrak{b}$ and $\mathfrak{c}_s$ for $s \in \mathbb{Q}$
admit uniform lattices.
\end{remark}

\vspace{0.1cm}

\subsection{The double cover and the Vaisman companions }\label{double}
In this subsection, we will investigate the global behavior of compact balanced BTP threefolds of middle type. Let $(M^3,g)$ be such a threefold, and denote by $J$ its complex structure. We have seen that the kernel $L=\ker (B)$ is a holomorphic line bundle satisfying $L^{\otimes 2}\cong {\mathcal O}_{\!M}$. For convenience, let us introduce the following terminology:

\begin{definition}
A compact balanced BTP threefold  of middle type is called a  {\bf primary} one if $L\cong {\mathcal O}_{\!M}$, otherwise it is called a {\bf secondary} threefold.
\end{definition}

\begin{proposition} \label{propsecondary}
A compact balanced BTP threefold $(M^3,g)$ of middle type is primary if and only if its Bismut holonomy group is abelian. In this case $\mbox{Hol}^{\, b}(M)= U(1)\!\times\!U(1)\!\times\!1$, and there are complex (non-holomorphic) line bundles $L_1$, $L_2$ on $M$ such that the tangent bundle is the orthogonal direct sum $L_1\oplus L_2\oplus L$. Any secondary threefold admits a double cover that is a primary one, where its Bismut holonomy group is the following ${\mathbb Z}_2$-extension of $U(1)\!\times\!U(1)\!\times\!1$:
$$ G= \left\{ \begin{bmatrix} \rho_1 & & \\ & \rho_2 & \\ && 1\end{bmatrix},
\begin{bmatrix} 0 & \rho_3  & \\  \rho_4 &0 & \\ && -1\end{bmatrix} \ \Bigg| \ |\rho_1|=|\rho_2|=|\rho_3|=|\rho_4|=1\ \right\} \subset U(3),$$
which is not abelian.
\end{proposition}

\begin{remark}
The Bismut holonomy group $Hol^b(M)$ above is the global holonomy group of Bismut connection, which should not be confused with
the restrict Bismut holonomy group $Hol_0^b(M)$ in Lemma \ref{remark9.2}. %The condition that the Bismut holonomy group is abelian
%should neither be confused with the terminology Bismut (holonomy) abelian, introduced in \cite[Definition 1.17]{ZhaoZ24}, which means
%that the restricted Bismut holonomy group is abelian.
\end{remark}

\begin{proof} Let $(M^3,g)$ be a compact balanced BTP threefold of middle type. Fix an admissible frame $e^0$. If $e$ is another admissible frame, then either $e_3=e^0_3$ or $e_3=-e^0_3$. Let us call the former {\em positive type} and the latter {\em negative type}. If $M$ can be covered by a collection of neighborhoods $\{ U_{\alpha} \}_{\alpha \in A}$ such that on each $U_{\alpha}$ there is a positive type admissible frame $e^{\alpha}$, then we can take the union of these $e^{\alpha}_3$ to get a global holomorphic section of $L$, hence $L$ is trivial. Conversely, if $L$ is trivial, then $e^0_3$ extends to a global holomorphic section of $L$ and one can use only positive type admissible frames to cover the entire manifold. For such a frame $e$, note that even though $e_1$, $e_2$ are only defined locally and have the ambiguity of rotating by a function with norm $1$, but ${\mathbb C}e_1$ and ${\mathbb C}e_2$ are well-defined, and become global complex line bundles on $M^3$. Clearly, $T_{\!M}=L_1\oplus L_2\oplus L$ is the orthogonal direct sum and the Bismut connection $\nabla^b$ preserves this splitting, so the holonomy group $\mbox{Hol}^{\,b}(M)$ is contained in $H=U(1)\!\times\!U(1)\!\times\!1$. It actually equals to $H$ since the restricted holonomy group is already equal to $H$.

Next let us assume that $L\not\cong {\mathcal O}_{\!M}$. We want to show that $\mbox{Hol}^{\,b}(M)$ is not abelian in this case.  By  Lemma \ref{claim9.1} we have $L^{\otimes 2}\cong {\mathcal O}_{\!M}$, which defines an unbranched  double cover $\pi : \hat{M} \rightarrow M$ with $\pi^{\ast}L \cong {\mathcal O}_{\!\hat{M}}$. Lift the metric up onto $\hat{M}$ and still denote it by $g$. Then $(\hat{M},g)$ becomes a compact balanced BTP threefold of middle type, and the kernel line bundle $\hat{L}$ of $B$ is equal to $\pi^{\ast}L$ which is trivial. Hence $\hat{M}$ is primary, with $T_{\hat{M}}=\hat{L}_1\oplus \hat{L}_2\oplus \hat{L}$ splitting orthogonally and  $\nabla^b$ preserves the decomposition. So the Bismut holonomy group of $\hat{M}$ is $H$.

$\pi$ induces an injective homomorphism  $\pi_{\ast}: \mbox{Hol}^{\,b}(\hat{M}) \rightarrow \mbox{Hol}^{\,b}(M)$. To be precise, denote by $f$ the involution on $\hat{M}$ which is in the deck transformation group of $\pi$, and fix a base point $p\in \hat{M}$ with $q=f(p)$. If $\gamma$ is a loop from $p$ to $p$, then the image $\pi (\gamma)$ is a loop from $\pi (p)$ to $\pi (p)$, and the $\nabla^b$-parallel transport $P_{\gamma}$ along $\gamma$ is mapped to the $\nabla^b$-parallel transport $P_{\pi (\gamma )}$ along $\pi (\gamma)$. Fix a path $\gamma_0$ from $p$ to $q$, then $x=P_{\pi (\gamma_0)}$ is also an element in $\mbox{Hol}^{\,b}(M)$, and $\mbox{Hol}^{\,b}(M)$ is generated by the subgroup $\pi_{\ast}(\mbox{Hol}^{\,b}(\hat{M}))$ and $x$, so  $\mbox{Hol}^{\,b}(M)$ contains $H$ as a subgroup of index either $1$ or $2$, depending on whether $x$ belongs to $\pi_{\ast}(\mbox{Hol}^{\,b}(\hat{M}))\cong H$ or not. Clearly, $f$ sends positive type admissible frames to negative type admissible frames, otherwise $M$ could be covered by positive type frames so $L\cong {\mathcal O}_{\!M}$ would be trivial, contradicting to our assumption. This means that $x$ is in the form
$$ x = \left[ \begin{array}{ccc} 0 & \rho_3 & \\ \rho_4 & 0 & \\ & & -1 \end{array}\right] . $$
Thus we conclude that $\mbox{Hol}^{\,b}(M)$ is isomorphic to the group $G$, which is the ${\mathbb Z}_2$-extension of $H$ given in the proposition. Note that $G$ is not abelian, and we have completed the proof of the proposition.
\end{proof}

From now on, we will assume that $(M^3,g)$ is a compact balanced BTP threefold of middle type that is {\bf primary}. The holomorphic tangent bundle $T_{\!M}$ decomposes as the orthogonal direct sum $L_1 \oplus L_2 \oplus L$ of complex line bundles. The decomposition does not vary holomorphically on $M$, or equivalently, each $L_i$ is not a holomorphic line bundle. However, if we let $F_i=L_i\oplus L$ for $i=1$, $2$, then we have

\begin{lemma} \label{claim9.5b}
Let $(M^3,g)$ be a compact balanced BTP threefold of middle type that is primary. Then for each $1\leq i\leq 2$, the subbundle $F_i=L_i\oplus L$ is holomorphic and is a foliation.
\end{lemma}

\begin{proof}
Let $e$ be an admissible frame on $M$ so that $e_3$ is extended to a global holomorphic section of $L$, and $e_i$ gives a local section of $L_i$ for $i=1$ and $2$. By (\ref{eq:9.1b}), we compute
$$ \nabla^c_{\overline{e}_j} e_1 = \sum_{i=1}^3 \theta_{1i}(\overline{e}_j)e_i = \theta_{11}(\overline{e}_j)e_1 -a \overline{\varphi}_1(\overline{e}_j)e_3 \in F_1, \ \ \ \ \forall \ 1\leq j\leq 3.$$
This means that $F_1$ is holomorphic. Similarly, $F_2$ is also holomorphic. Next, for any $1\leq i,j\leq 3$ we have
\begin{eqnarray*}
 && [e_i, e_j] = \nabla^c_{e_i}e_j -  \nabla^c_{e_j}e_i -T^c(e_i,e_j) = \sum_{k} \big(\theta_{jk}(e_i) - \theta_{ik}(e_j) - T^k_{ij} \big) e_k ,\\
 && [e_i, \overline{e}_j] = \nabla^c_{e_i}\overline{e}_j -  \nabla^c_{\overline{e}_j}e_i -T^c(e_i,\overline{e}_j) = \sum_{k} \overline{\theta_{jk}(\overline{e}_i)} \, \overline{e}_k  - \sum_k \theta_{ik}(\overline{e}_j)  e_k.
 \end{eqnarray*}
From this we get $\,[e_1,e_3]=-\theta_{11}\!(e_3)\,e_1$, $\,[e_3, \overline{e}_3]=0$, $\, [e_1, \overline{e}_3]=-\theta_{11}\!(\overline{e}_3)\,e_1$, and
$$ [e_1, \overline{e}_1] = \overline{\theta_{11}(\overline{e}_1)}\,\overline{e}_1  - \theta_{11}(\overline{e}_1)\,e_1  - a \,\overline{e}_3 + a \,e_3. $$
This shows that $F_1\oplus \overline{F}_1$ is closed under the Lie bracket, so $F_1\subset T_{\!M}$ is a holomorphic foliation. Similarly, so is $F_2$, and the lemma is proved.
\end{proof}

Denote by $\nabla$ the Levi-Civita connection of $(M^3,g)$. Let $e$ be an admissible frame. We have
$$\nabla e = \theta^{(1)}e+\overline{\theta^{(2)}}\,\overline{e}, \ \ \ \mbox{where}  \ \ \theta^{(1)}=\theta^b -\frac{1}{2}\gamma, \ \ \ \theta^{(2)}_{ij} = \frac{1}{2}\sum_k\overline{T^k_{ij}} \varphi_k. $$
Under any admissible frame $e$, by (\ref{eq:9.1b}) we have
\begin{eqnarray} \label{eq:LeviCivitab}
\nabla e_1 & = &  \alpha_1 e_1 - \frac{a}{2}\overline{\varphi}_1 (e_3 - \overline{e}_3), \label{eq:9A}\\
\nabla e_2 & = &  \alpha_2 e_2 + \frac{a}{2}\overline{\varphi}_2 (e_3 - \overline{e}_3), \label{eq:9B} \\
\nabla e_3 & = &  \frac{a}{2} \big( \varphi_1e_1-\varphi_2e_2 \big)
- \frac{a}{2} \big( \overline{\varphi}_1 \overline{e}_1 - \overline{\varphi}_2 \overline{e}_2 \big) , \label{eq:9C}
\end{eqnarray}
where
$$ \alpha_1 \ = \  \theta^b_{11} - \frac{a}{2}(\varphi_3 - \overline{\varphi}_3), \ \ \ \ \ \ \ \ \
\alpha_2 \ = \ \theta^b_{22} + \frac{a}{2}(\varphi_3 - \overline{\varphi}_3).  \ \ \ \ \ \ \  $$
In particular, $\nabla (e_3+\overline{e}_3)=0$, so the universal cover of $(M^3,g)$ splits off a line.

\vspace{0.2cm}

Next we observe that the primary balanced BTP threefold $(M^3,g)$ of middle type admits multiple complex structures compatible with the Riemannian metric $g$, or multiple {\em orthogonal complex structures} (abbreviated as OCS) in the terminology of Salamon \cite{Salamon2}. Fix an admissible frame $e$ on $M$. Since $L$ is trivial,  $e_3$ can be extend  to a global holomorphic section $\sigma$ of $L$ so that $\nabla^b \sigma =0$. From now on, we will only use admissible frames on $M$ that are positive type, meaning that their $e_3$ coincide with $\sigma$, so $e_3$ is uniquely determined, while $e_1$, $e_2$ are only local and can be rotated by a function with norm $1$.  Denote by $J$ the complex structure of $M$. Let us define almost complex structures $J_1$, $J_2$, $J_3$ on $M$ by
\begin{equation} \label{eq:Ji}
J_ie_i=\sqrt{-1}e_i, \ \ \ J_ie_j=-\sqrt{-1}e_j \ \ \ \ \mbox{if} \ \ j\neq i, \ \ \ \ \ 1\leq i,j\leq 3,
\end{equation}
for any positive type admissible frame $e$. Clearly, each $J_i$ is compatible with the metric $g$, satisfying $J\circ J_i=J_i \circ J$ and  $J_i \circ J_j = J_j\circ J_i$ for any $1\leq i,j \leq 3$. There are also the almost complex structures $-J$, $-J_i$ compatible with $g$, which have opposite orientation with $J$. Those $J$ and  $J_i$ generates an abelian group of $16$ elements:
$$ S=\langle J, J_1, J_2, J_3\rangle  =  \{ \pm 1, \,\pm J, \,\pm J_i, \,\pm J\circ J_i \}_{1\leq i\leq 3} \,\cong \,{\mathbb Z}_4  \oplus {\mathbb Z}_2 \oplus {\mathbb Z}_2. $$
We will call $S$ the {\em Salamon group of OCS} for the primary balanced BTP threefold of middle type.

\begin{lemma} \label{claim:int}
Let $(M^3,g)$ be a compact balanced BTP threefold of middle type that is primary. For each $1\leq i\leq 3$, let $J_i$ be the almost complex structure on $M^3$ defined by (\ref{eq:Ji}). Then they are all integrable.
\end{lemma}

\begin{proof}
Fix an $i\in \{ 1,2,3\}$. To show that $J_i$ is integrable, we just need to verify that its Nijenhuis tensor vanishes, namely,  for any two vector fields $x$ and $y$ on $M$, it holds that
\begin{equation} \label{eq:9int}
 N_i(x,y):= [x,y] - [J_ix,J_iy] + J_i[J_ix,y] + J_i[x,J_iy] =0.
 \end{equation}
Since the equality (\ref{eq:9int}) obviously holds for $x=e_j$ and $y=\overline{e}_j$ for any $j$, it suffices to check it for the case $x=e_j$, $y=e_k$ and the case  $x=e_j$, $y=\overline{e}_k$ for any $1\leq j\neq k\leq 3$.  Recall that our $J_i$ is defined by
$$ J_ie_i=Je_i,\ \ \ J_ie_k =-Je_k \ \ \mbox{when} \ \ k\neq i, \ \ \ 1\leq i,k\leq 3. $$

{\em Case 1.} $x=e_j$, $y=e_k$, $1\leq j\neq  k\leq 3$. First we may ignore the case when $i\not\in \{ j,k\}$, since $J_i$ acts the same way as $-J$ so (\ref{eq:9int}) holds. If $i=j$, then $[J_ie_i,J_ie_k]=[e_i,e_k]$, so the first two terms in $N_i(e_i,e_k)$ cancel each other. Similarly, the last two terms also cancel each other, so (\ref{eq:9int}) holds.

\vspace{0.1cm}

{\em Case 2.} $x=e_j$, $y=\overline{e}_k$, $1\leq j\neq k\leq 3$. Similarly with Case 1, we may ignore the case when $i\not\in \{ j,k\}$,
since $J_i$ acts the same way as $-J$ and we know that $N_J=0$. So we may assume that $i=j\neq k$  or $i=k\neq j$. In the first case, $[J_ie_i, J_i\overline{e}_k] = [\sqrt{-1}e_i, \sqrt{-1}\overline{e}_k]=-[e_i,\overline{e}_k]$, while $[J_ie_i,\overline{e}_k]+[e_i, J_i\overline{e}_k]=2\sqrt{-1}[e_i, \overline{e}_k]$, hence we have
$$ N_i(e_i, \overline{e}_k) = 2[e_i,\overline{e}_k] + 2\sqrt{-1}\,J_i [e_i, \overline{e}_k] . $$
If $i=k\neq j$, then $[J_ie_j, J_i\overline{e}_i] = -[e_j,\overline{e}_i]$, while $[J_ie_j, \overline{e}_i]+[e_j, J_i\overline{e}_i]=-2\sqrt{-1}[e_j, \overline{e}_i]$, thus
$$ N_i(e_j, \overline{e}_i) = 2[e_j,\overline{e}_i] - 2\sqrt{-1}\,J_i [e_j, \overline{e}_i] . $$
So what we need to show is that
\begin{equation} \label{eq:9.11bb}
J_i[e_i, \overline{e}_k]= \sqrt{-1} \,[e_i,\overline{e}_k] \ \ \ \mbox{and} \ \ \  J_i[e_j, \overline{e}_i]= -\sqrt{-1} \,[e_j,\overline{e}_i] ,\ \ \ \ \mbox{whenever} \ \ j,k\neq i.
\end{equation}
We have  $[e_i, \overline{e}_j] =  \nabla^c_{e_i}\overline{e}_j - \nabla^c_{\overline{e}_j}e_i = \sum_{k} \big( \overline{\theta_{jk }(\overline{e}_i) } \,\overline{e}_{k} - \theta_{ik }(\overline{e}_j)\,e_{k} \big) $, hence by (\ref{eq:9.1b}) we obtain
\begin{eqnarray*}
 &&  [e_1, \overline{e}_2] \ =  \ \overline{\theta_{22 }(\overline{e}_1) } \,\overline{e}_{2} - \theta_{11 }(\overline{e}_2)\,e_{1} , \ \ \ \ \  [e_1, \overline{e}_3]\  = \ - \, \theta_{11}(\overline{e}_3)\,e_1, \\
  &&   [e_2, \overline{e}_1] \ =  \ \overline{\theta_{11}(\overline{e}_2) } \,\overline{e}_{1} - \theta_{22 }(\overline{e}_1)\,e_{2} , \ \ \ \ \  [e_2, \overline{e}_3]\  = \ - \, \theta_{22}(\overline{e}_3)\,e_2,  \\
  && [e_3, \overline{e}_1] \ =  \ \overline{\theta_{11 }(\overline{e}_3) } \,\overline{e}_{1}  , \ \ \ \ \hspace{2.0cm}   [e_3, \overline{e}_2]\  = \   \overline{ \theta_{22}(\overline{e}_3)} \,\overline{e}_2.
\end{eqnarray*}
From this and the definition of $J_i$ we see that (\ref{eq:9.11bb}) holds. This completes the proof that each $J_i$ is integrable.
\end{proof}

\begin{proposition} \label{propVaisman}
Let $(M^3,g, J)$ be a compact balanced BTP threefold of middle type that is primary, and $J_i$ be defined by (\ref{eq:Ji}).
For each $I \in \{ \pm J_1, \pm J_2\}$, $(M^3,g,I)$ is a Vaisman threefold and its Bismut connection coincides with that of $(M^3,g,J)$. In particular, the Bismut holonomy group of $(M^3,g,I)$ is abelian and equals to $U(1)\!\times\!U(1)\!\times\!1$.
\end{proposition}

\begin{proof}
Let $e$ be an admissible frame in $(M^3,g,J)$ with $\varphi$ its dual coframe. Let us take $I=-J_2$ here, as the other cases are analogous. By definition, we have $Ie_1=\sqrt{-1}e_1$, $Ie_2=-\sqrt{-1}e_2$, and $Ie_3=\sqrt{-1}e_3$.
Let
$$\varepsilon_1=e_1, \ \ \ \varepsilon_2 =\overline{e}_2, \ \ \  \varepsilon_3=e_3; \ \ \ \ \mbox{and} \ \ \ \ \psi_1=\varphi_1, \ \ \ \psi_2=\overline{\varphi}_2, \ \ \ \psi_3=\varphi_3. $$
Then $\varepsilon$ is a local unitary frame on $(M^3,g,I)$ with dual coframe $\psi$. By (\ref{eq:9.2b}) we have
$$ d\psi_1=-\theta_{11}\psi_1,\ \ \ d\psi_2=\theta_{22}\psi_2, \ \ \ d\psi_3 = -a (\psi_{1\bar{1}} + \psi_{2\bar{2}}). $$
If we let
$$ \hat{\theta} = \left[ \begin{array}{ccc} \theta_{11} & 0 & -a\overline{\psi}_1 \\ 0 & -\theta_{22} &  -a\overline{\psi}_2 \\ a\psi_1 & a\psi_2 & 0 \end{array} \right], \ \ \ \ \hat{\tau} = \left[ \begin{array}{c} a\psi_1\psi_3 \\ a\psi_2\psi_3  \\ 0  \end{array} \right] ,$$
then we have $d\psi = -\,^t\!\hat{\theta}\wedge \psi + \hat{\tau}$. Since $\hat{\theta}$ is skew-Hermitian and the entries of $\hat{\tau}$ are $(2,0)$-forms in $(M^3,I)$, we know that $\hat{\theta}$ is the matrix of Chern connection for $(M^3,g,I)$ under $\varepsilon$, while $\hat{\tau}$ is the column vector of Chern torsion under $\varepsilon$. In particular, the only non-zero torsion components are $\hat{T}^1_{13}=\hat{T}^2_{23}=a$. Thus the matrices of the $\gamma$ tensor  and Bismut connection of $(M^3,g,I)$ under $\varepsilon$ are
$$ \hat{\gamma} =   a \left[ \begin{array}{ccc} \psi_3 - \overline{\psi}_3  & 0 & \overline{\psi}_1 \\ 0 & \psi_3 - \overline{\psi}_3 &  \overline{\psi}_2 \\ -\psi_1 & -\psi_2 & 0 \end{array} \right] , \ \ \ \ \ \hat{\theta}^b = \hat{\theta}+\hat{\gamma} = \left[ \begin{array}{ccc} \hat{\theta}^b_{11}  & 0 & 0 \\ 0 & \hat{\theta}^b_{22} &  0 \\ 0 & 0 & 0 \end{array} \right],$$
respectively, where
$$\hat{\theta}^b_{11} = \theta_{11} + a\psi_3 -a\overline{\psi}_3 = \theta^b_{11}, \ \ \ \ \ \ \hat{\theta}^b_{22} = -\theta_{22} + a\psi_3 -a\overline{\psi}_3 = - \theta^b_{22}. $$
Denote by $\hat{\nabla}^b$ the Bismut connection of $(M,g,I)$. Then the above calculation shows that $\hat{\nabla}^be_i=\nabla^be_i$ for $i=1$ and $3$, while
$$ \hat{\nabla}^b \overline{e}_2 = \hat{\nabla}^b \varepsilon_2 = \hat{\theta}^b_{22}\varepsilon_2= - \theta^b_{22} \overline{e}_2. $$
Taking complex conjugation and using the fact that $\overline{\theta^b_{22}} =- \theta^b_{22}$, we get $\hat{\nabla}^be_2 =\nabla^be_2$. This means that $\hat{\nabla}^b=\nabla^b$. Clearly, $\hat{\nabla}^b\hat{T}=0$, so $(M^3,g,I)$ is BTP. It is not balanced, as its Gauduchon's torsion $1$-form $\hat{\eta}=2a\psi_3$ which is lined up with $\varepsilon_3$. Since $\hat{T}^1_{23}=\hat{T}^2_{13}=0$ and  $\hat{T}^1_{13}=\hat{T}^2_{23}=a$, by \cite[Proposition 1.8]{ZhaoZ24} we know that $(M^3,g,I)$ is a Vaisman manifold. Since $(M^3,g,I)$ and $(M^3,g,J)$ shares the same metric and Bismut connection, their Bismut holonomy group is the same and equals to $H=U(1)\!\times\!U(1)\!\times\!1$. Alternatively, because the Bismut connection $\hat{\nabla}^b$ preserves the orthogonal decomposition $L_1\oplus L_2\oplus L$ on $M$, the Bismut holonomy group of $(M^3,g,I)$ is contained in $U(1)\!\times\!U(1)\!\times\!1$ thus is abelian, and it actually equals to $U(1)\!\times\!U(1)\!\times\!1$ since one computes that $\hat{R}^b_{1\bar{1}2\bar{2}}=\hat{R}^b_{2\bar{2}1\bar{1}}=-2a^2$ which is non-zero.
\end{proof}

\begin{remark}
It follows similarly from the proof of Lemma \ref{claim9.5b} that, for $i=1,2$, the subbundle $F_i \oplus \overline{F}_i$ is
always closed under Lie bracket on $(M^3,g,I)$ above for each $I \in \{ \pm J_1, \pm J_2\}$.
\end{remark}

Since the metric of the Vaismann manifold $(M^3,g,I)$ is not K\"ahler, we know that the first Betti number $b_1(M)>0$. By the beautiful  theorem of Ornea and Verbitsky \cite{OV03} on the structure of compact Vaisman manifolds, we also get: \emph{$M^3$ is a smooth fiber bundle over the circle $S^1$ with fiber being a compact Sasakian manifold $N^5$.}

Since $(M,g,I)$ and $(M,g,J)$ have the same Bismut connection, their Bismut curvature tensors, which we denote by $\hat{R}^b$ and $R^b$ respectively, are the same, in particular, they have the same sectional curvature. However, since the complex structures are different, the `bisectional curvature' or `holomorphic sectional curvature' are different. Under admissible frames, the non-zero Bismut curvature components are
\begin{equation}
\hat{R}^b_{1\bar{1}1\bar{1}} = R^b_{1\bar{1}1\bar{1}}, \ \ \  \hat{R}^b_{2\bar{2}2\bar{2}}=R^b_{2\bar{2}2\bar{2}}, \ \ \ \hat{R}^b_{1\bar{1}2\bar{2}}=-2a^2, \ \ \  R^b_{1\bar{1}2\bar{2}}=2a^2.
\end{equation}
In particular, $(M^3,g,J)$ is CYT if and only if $R^b_{1\bar{1}1\bar{1}}=R^b_{2\bar{2}2\bar{2}}=-2a^2$, while $(M^3,g,I)$ is CYT if and only if $\hat{R}^b_{1\bar{1}1\bar{1}}=\hat{R}^b_{2\bar{2}2\bar{2}}=2a^2$. This implies that a primary balanced BTP threefold of middle type $(M^3,g,J)$
and its Vaisman companion $(M^3,g,I)$ can never be CYT at the same time.

\begin{definition}
Let $(M^3,g,J)$ be a compact balanced BTP threefold of middle type that is primary. For each $I\in \{ \pm J_1, \pm J_2\}$, we will call the non-K\"ahler Vaisman threefold $(M^3,g,I)$ a {\bf Vaisman companion} of $(M^3,g,J)$.
\end{definition}

At this point, one is naturally curious about the Hermitian threefold $(M^3,g,J_3)$ or $(M^3,g,-J_3)$. It turns out that just like $(M^3,g,-J)$, they are balanced BTP threefold of middle type and share the same Bismut connection with $(M^3,g,J)$. The proof is analogous to that of Proposition \ref{propVaisman}, so we omit it here.

\begin{lemma}
Given a primary compact balanced BTP threefold of middle type $(M^3,g,J)$, the Hermitian threefolds
 $(M^3,g,\pm J_3)$ are also primary balanced BTP threefolds of middle type, whose Bismut connection coincide with that of $(M^3,g,J)$.
\end{lemma}

The converse to Proposition \ref{propVaisman} also holds, namely, any compact, non-K\"ahler Vaisman threefold with abelian Bismut holonomy group must be a Vaisman companion of some balanced BTP threefold of middle type that is primary.

\begin{proposition} \label{propinverse}
Let $(M^3,g,I)$ be a compact, non-K\"ahler Vaisman threefold such that its Bismut holonomy group is abelian. Then there exists another complex structure $J$ on $M^3$ compatible with $g$, such that $(M^3,g,J)$ is primary balanced BTP of middle type, and $(M^3,g,I)$ is a Vaisman companion of  $(M^3,g,J)$.
\end{proposition}

\begin{proof}
Let $(M^3,g,I)$ be a compact, non-K\"ahler Vaisman threefold. It is a non-balanced BTP manifold. Denote by $\hat{\nabla}^b$, $\hat{T}$, and $\hat{H}$ the Bismut connection, Chern torsion, and Bismut holonomy group of $(M^3,g,I)$, respectively. Then by \cite[Proposition 1.7]{ZhaoZ24}, we know that there always exist admissible frames, namely a local unitary frame $\epsilon $ in $M^3$ such that
$\nabla^b\epsilon_3=0$ with $\epsilon_3$ a global holomorphic vector field of constant norm, which indicates $\hat{H} \subset U(2)\times 1$, and the only non-zero torsion components are $\hat{T}^1_{13}=\hat{T}^2_{23}=a>0$. Write $\hat{H}=G\!\times\!1$ with $G\subset U(2)$. Since $\hat{H}$ is abelian, $G$ is an abelian subgroup of $U(2)$ hence is conjugate to a diagonal subgroup. This means that there exists a unitary change $\{ \varepsilon_1, \varepsilon_2\}$ of $\{ \epsilon_1, \epsilon_2\}$ so that $\hat{H}$ preserves the splitting ${\mathbb C}\varepsilon_1 \oplus {\mathbb C}\varepsilon_2$. Write $\varepsilon_i=\sum_{j=1}^2 U_{ij}\epsilon_j$, $1\leq i\leq 2$, where $U$ is a $U(2)$-valued local function, and let $\varepsilon_3=\epsilon_3$. Then under the new frame $\varepsilon$ the non-zero torsion components are still $\hat{T}^1_{13}=\hat{T}^2_{23}=a$, in other words, $\varepsilon$ is still an admissible frame of $(M^3,g,I)$ in the sense of Definition 1.6 of \cite{ZhaoZ24}, and the holonomy group $\hat{H}$ is contained in (hence equals to) $U(1)\!\times\!U(1)\!\times\!1$, which preserves an orthogonal decomposition of the tangent bundle into the direct sum of complex line subbundles $L_1\oplus L_2\oplus L_3$ where $L_i={\mathbb C}\varepsilon_i$.

Define an almost complex structure $J$ on $M^3$ by letting $J=I$ on $L_1\oplus L_3$ while letting $J=-I$ on $L_2$. In other words, $e$ will be  a local unitary frame for $(M^3,g,J)$ if $e_1=\varepsilon_1$, $e_2=\overline{\varepsilon}_2$ and $e_3=\varepsilon_3$. By analogous deduction as in the proofs of Lemma \ref{claim:int} and Proposition \ref{propVaisman}, we see that $J$ is integrable, and $(M^3,g,J)$ is balanced BTP whose Bismut connection $\nabla^b$ coincides with $\hat{\nabla}^b$. Its Bismut holonomy group is equal to $U(1)\!\times\!U(1)\!\times\!1$, hence it is primary, and $(M^3,g,I)$ is one of its Vaisman companions. This completes the proof of Proposition \ref{propinverse}.
\end{proof}

%\vspace{0.1cm}

\subsection{Examples of Vaisman threefolds with abelian Bismut holonomy}\label{Hopf}
In this subsection, let us consider some concrete examples of Vaisman threefold with abelian Bismut holonomy.
The first one is a Vaisman companion of the complex nilmanifold determined by $N^3$, since
the corresponding simply-connected Lie group $G$ admits uniform lattice as shown in Subsection \ref{example}:

\begin{example}
Let $M^3=G/\Gamma $ be a compact quotient of the nilpotent Lie group $G$ by a discrete subgroup $\Gamma \subset G$, where the Lie algebra ${\mathfrak g}$ of $G$ admits a unitary coframe $\varphi$ satisfying the structure equation
$$ d\varphi_1=d\varphi_2=0,\ \ \ \ d\varphi_3= -a (\varphi_{1\bar{1}}+ \varphi_{2\bar{2}}). \hspace{2cm} $$
Here $a>0$ is a constant. It is a Vaisman companion to the balanced BTP nilmanifold determined by $N^3$
with the structure equation $d\varphi_1=d\varphi_2=0$,  $d\varphi_3= -a \varphi_{1\bar{1}}+ a\varphi_{2\bar{2}}$.
\end{example}

For this $M^3$, under the unitary frame $e$ dual to $\varphi$, the structure constants $C^j_{ik}=0$, and the only non-zero $D$ components are $D^1_{31}=D^2_{32}=a$, so the only non-trivial torsion components are $T^1_{13}=T^2_{23}=a$. From this we get the matrices for Chern connection and curvature and the tensor $\gamma = \theta^b-\theta$:
$$ \theta = a\! \left[ \begin{array}{ccc} 0& 0 & -\overline{\varphi}_1 \\ 0& 0 & - \overline{\varphi}_2 \\ \varphi_1 & \varphi_2 & 0 \end{array} \right], \ \ \  \Theta = a^2\! \left[ \begin{array}{ccc} -\varphi_{1\bar{1}} & -\varphi_{2\bar{1}} & 0 \\  -\varphi_{1\bar{2}}  &  -\varphi_{2\bar{2}}   & 0 \\ 0 & 0 & \varphi_{1\bar{1}} \!+\! \varphi_{1\bar{1}} \end{array} \right] , \ \ \ \gamma  = a\! \left[ \begin{array}{ccc} \varphi_3 \!-\!  \overline{\varphi}_3  & 0 & \overline{\varphi}_1 \\ 0 & \varphi_3 \!-\!\overline{\varphi}_3 & \overline{\varphi}_2 \\  -\varphi_1 & -\varphi_2 & 0 \end{array} \right]. $$
Hence the matrices for the Bismut connection and curvature are:
$$ \theta^b = a\! \left[ \begin{array}{ccc} \varphi_3 \!-\! \overline{\varphi}_3 &  &  \\ & \varphi_3 \!-\!  \overline{\varphi}_3  &  \\  &  & 0 \end{array} \right], \ \ \ \ \Theta^b =-2 a^2\! \left[ \begin{array}{ccc} \varphi_{1\bar{1}} \!+\!  \varphi_{2\bar{2}}  &  &  \\ & \varphi_{1\bar{1}} \!+\!  \varphi_{2\bar{2}} &  \\  &  & 0 \end{array} \right],  $$
In particular, $\mbox{tr}(\Theta)=0$ so $M^3$ is Chern Ricci flat, while its (first) Bismut Ricci form is
$$ \sqrt{-1} \mbox{tr}(\Theta^b) = -4a^2 \sqrt{-1} \big( \varphi_{1\bar{1}} +  \varphi_{2\bar{2}}  \big) . $$
In particular, the Vaisman threefold $M^3$ is not CYT.

%\vspace{0.1cm}

Similarly, by taking Vaisman companion of the solvmanifold determined by $A_{s,t}$ or $B_{z,t}$, when the corresponding simply-connected Lie groups have uniform lattices (when the parameter is rational), we get examples of compact Vaisman threefolds with abelian Bismut holonomy. See \cite{AO} for a more detailed discussion on such Vaisman manifolds.

\vspace{0.4cm}

\section{Generalization to higher dimensions}\label{highD}

At present time, we do not know how to approach the classification problem for compact balanced BTP manifolds in dimension $4$ or higher, despite our belief that such manifolds should form a highly restrictive special class.  However, one could presumably  at least try to generalize the three types of such threefolds, namely the Chern flat case, the Fano case, and the middle type ones. In other words, one could ask smaller questions such as:

\begin{question}\begin{enumerate}
\item What kind of compact Chern flat manifolds  are  BTP?
\item What kind of Fano manifolds can admit balanced but non-K\"ahler BTP metrics?
\item What are the high dimensional generalizations of balanced BTP threefolds of middle type?
\end{enumerate}
\end{question}

For part (1), note that compact Chern flat manifolds are always balanced. The recent work \cite{PodestaZ} gives a satisfactory  answer to (1). By the classic result of Boothby \cite{Boothby}, any compact Chern flat manifold $(M^n,g)$ is a quotient of a complex Lie group $G$ equipped with a left-invariant metric $\tilde{g}$ which is compatible with the complex structure of $G$, and $\tilde{g}$ is the lift of  $g$.  Theorem 1.2 in \cite{PodestaZ} states that, if $g$ is BTP, then $G$ must be reductive, and in fact it is the orthogonal direct product
$ G={\mathbb C}^k \times G_1 \times \cdots \times G_r$
where each $G_i$ is a simple complex Lie group. Conversely,  any reductive complex Lie group $G$ admits a left-invariant metric which is BTP (and is balanced and Chern flat).

For part (2), not much is known except a couple of partial results from \cite{PodestaZ}.  Note that the flag threefold is $X=SU(3)/T^2$. Theorem 1.4 of \cite{PodestaZ} generalizes the Wallach threefold case to higher dimensions: for any $k\geq 2$, the metric $g$ naturally induced by the Killing form on full flag $X^n = SU(k+1)/T^k$ is (balanced, non-K\"ahler) BTP, and any non-K\"ahler BTP metric on $X^n$ is a constant multiple of $g$. Here $n=\frac{1}{2}k(k+1)$ is the complex dimension. Theorem 1.3 of \cite{PodestaZ} says the same thing holds if one replaces the type A full flag by any K\"ahler C-space of the form $Y=K/H$ where $K$ is a compact simple group and the isotropy representation of $H$ has at most two direct summands. Table 1 of \cite{PodestaZ} listed all such K\"ahler C-spaces. The Killing metric $g$ is K\"ahler when and only when $Y$ is a compact Hermitian symmetric space. For all others in Table 1, we get examples of non-K\"ahler, balanced BTP manifolds.

Now let us focus on part (3). From our previous discussion, we know that primary balanced BTP threefolds of middle type have Vaisman companions, and these Vaisman threefolds have abelian Bismut holonomy groups, which forces a splitting on the tangent bundle and the splitting is preserved by the Bismut connection.

To be more precise, let  $(M^3,g,J)$ be a primary balanced BTP threefold of middle type, then as we have seen before, it has a  Vaisman companion $(M^3,g,I)$, which is a particular type of Vaisman threefold. The specialty of $(M^3,g,I)$ lies in the fact that its Bismut holonomy group $\mbox{Hol}^b(M)$ is $U(1)\!\times\!U(1)\!\times\!1$, or equivalently, the holomorphic tangent bundle $T_M$ of $(M^3,g,I)$ is the orthogonal direct sum $L_1\oplus L_2\oplus L$ of complex line bundles and the decomposition is preserved by the Bismut connection $\nabla^b$. For a given (non-K\"ahler) Vaisman $n$-manifold, denote by $\eta$ the Gauduchon torsion $1$-form, which is the global $(1,0)$-form on $M^n$ defined by $\partial (\omega^{n-1}) = -\eta \wedge \omega^{n-1}$. Denote by $\chi$ the vector field dual to $\eta$, namely, $\langle Y, \overline{\chi}\rangle =\eta(Y)$ for any type $(1,0)$ vector $Y$. It is easy to verify that $\chi$ is a holomorphic vector field of constant length, which also holds for general non-balanced BTP manifold as shown in \cite[Proposition 1.7]{ZhaoZ24}. The above discussion motivates us to propose the following:

\begin{definition}
An odd dimensional Vaisman manifold $(M^{2m+1},g,I)$ is said to be {\em of splitting type,} if there exist rank $m$ complex subbundles $L_1$, $L_2$ of the holomorphic tangent bundle $T_M$, such that $T_M=L_1\oplus L_2\oplus L$ is the orthogonal direct sum, and the Bismut connection preserves the decomposition. Here $L={\mathbb C}\chi $ is the holomorphic line bundle generated by the holomorphic vector field $\chi$ dual to Gauduchon's torsion $1$-form $\eta$.
\end{definition}

Equivalently speaking, a Vaisman manifold $M^{2m+1}$ is of splitting type if the Bismut holonomy group is contained in $U(m)\!\times\!U(m)\!\times\!1$. When $m>1$, this does not mean that the Bismut holonomy group must be abelian. Also, a compact Vaisman manifold $M$ is always a metric fiber bundle over $S^1$ with fiber being a Sasakian manifold $N$. $M$ being of splitting type means that $N$ is $(4m+1)$-dimensional and satisfies a particular condition. As we have seen in the $m=1$ case and below, $N$ will admit two orthogonal complex structures $I$ and $J$ on the orthogonal complement of its Reeb vector field, satisfying $IJ=JI$, so it is a sort  of `bi-Hermitian' structure. It would be an interesting problem itself to analyze or classify this special type of Sasakian manifolds.

Mimic the $3$-dimensional case, and we have the following:

\begin{theorem} \label{thm6.3}
Let $(M^{2m+1},g,I)$ be a compact Vaisman manifold of splitting type. Define an almost complex structure $J$ on $M$ by letting $J=I$ on $L_1\oplus L$ and $J=-I$ on $L_2$. Then $J$ is integrable, and the Hermitian manifold $(M^{2m+1},g,J)$ is balanced BTP. Furthermore, the Bismut connection of these two Hermitian manifolds coincide.
\end{theorem}

\begin{proof} The proof is analogous to the $3$-dimensional case, and we give it here for the sake of completeness. Write $n=2m+1$, and let $\nabla^b$ be the Bismut connection and $\chi$ be the holomorphic vector field on $(M^n,g,I)$ dual to Gauduchon's torsion $1$-form $\eta$. By \cite{AndV}, Vaisman manifolds are BTP, so $\nabla^b \chi=0$. Let $\{ e_1, \ldots , e_n\}$ be a local unitary frame of $(M^n,g,I)$ so that $\lambda e_n=\chi$ where $\lambda =|\eta|>0$ is a global constant,  $\{ e_1, \ldots , e_m\}$ spans $L_1$, and $\{ e_{m\!+\!1}, \ldots , e_{2m}\}$ spans $L_2$. Also let $\varphi$ be the coframe dual to $e$. For convenience, let us denote by $e'$, $e''$ the column vector $\,^t\!(e_1, \ldots , e_m)$ and $\,^t\!(e_{m\!+\!1}, \ldots , e_{2m})$, respectively, and similarly write $\varphi'$, $\varphi''$ for the column vector $\,^t\!(\varphi_1, \ldots , \varphi_m)$ and $\,^t\!(\varphi_{m\!+\!1}, \ldots , \varphi_{2m})$. Since both $L_1$ and $L_2$ are preserved by $\nabla^b$, we have $\nabla^be'=\theta^b_1e'$, $\nabla^be''=\theta^b_1e''$, and $\nabla^be_n=0$ for some $m\times m$ matrices of $1$-forms $\theta^b_1$, $\theta^b_2$, so the Bismut connection matrix under $e$ is block-diagonal:
$$ \theta^b = \left[ \begin{array}{ccc} \theta^b_1 &&  \\  &\theta^b_2 & \\ &&0 \end{array} \right]. $$
On the other hand, since $(M^{n},g,I)$ is locally conformally K\"ahler, its Chern torsion components under any unitary frame $e$ would satisfy
$$ T^j_{ik} = \frac{1}{n-1} \big( \delta_{ij} \eta_k - \delta_{kj} \eta_i \big), \ \ \ \ \ \forall \ 1\leq i,j,k\leq n, $$
where $\eta = \sum_i \eta_i \varphi_i$. By our choice of $e$, we have $\eta_1=\cdots =\eta_{2m}=0$ and $\eta_n=\lambda$, so the only non-zero torsion components are $T^i_{in}=\frac{\lambda}{n-1}=a\,$ for $1\leq i\leq 2m$, and the Chern connection matrix and torsion vector are
\begin{equation} \label{eq:splittheta}
 \theta = \theta^b-\gamma = \left[ \begin{array}{ccc} \theta^b_1 -\alpha I & 0 &  -a\overline{\varphi}' \\  0 &\theta^b_2 -\alpha I & -a\overline{\varphi}''\\ a\,^t\!\varphi' & a\,^t\!\varphi'' &0 \end{array} \right], \ \ \ \ \ \ \tau = a \! \left[ \begin{array}{c} \varphi'   \\  \varphi'' \\ 0 \end{array} \right] \wedge \varphi_n,
 \end{equation}
where $\alpha = a\varphi_n - a\overline{\varphi}_n$. From this, we obtain
\begin{equation} \label{eq:splitstructure}
 d\varphi = -\,^t\!\theta \wedge \varphi + \tau = \left[ \begin{array}{c} -\,^t\!\theta_1 \wedge \varphi'   \\  -\,^t\!\theta_2\wedge \varphi'' \\ -a\,\big( \,^t\!\varphi' \wedge \overline{\varphi}' + \,^t\!\varphi'' \wedge \overline{\varphi}''  \big)  \end{array} \right],
\end{equation}
where $\theta_i = \theta^b_i-\alpha I$ for $i=1$, $2$. By the definition of $J$ on $M^n$, it holds that $Jx=Ix$ for $x\in L_1\oplus L\oplus \overline{L}_1\oplus \overline{L}$ and $Jx=-Ix$ for $x\in L_2\oplus \overline{L}_2$. In order the verify that $J$ is integrable, we need to show that
\begin{equation} \label{eq:splitintegrable}
 [x,y] -[Jx,Jy] + J [Jx,y] + J[x,Jy] =0
 \end{equation}
for any vector fields $x$, $y$ on $M^n$. Note that the same equality has been established for $I$, so when both $x$ and $y$ are in $L_1\oplus L\oplus \overline{L}_1\oplus \overline{L}$, the equality (\ref{eq:splitintegrable}) holds. Similarly, if both $x$ and $y$ are in $L_2\oplus \overline{L}_2$, then (\ref{eq:splitintegrable}) holds as either each term on the left is zero or the plus of the first two and also the plus of the last two are zeros. It remains to check the case for $x\in  L_1\oplus L\oplus \overline{L}_1\oplus \overline{L}$ and $y\in L_2\oplus \overline{L}_2$, in which case (\ref{eq:splitintegrable}) becomes
\begin{equation} \label{eq:splitintegrable2}
 ([x,y] + [Ix,Iy]) + J ([Ix,y] - [x,Iy]) =0 .
 \end{equation}
It suffices to check the case (a): $x=e_i$, $y=e_{\alpha}$ and the case (b):  $x=e_i$, $y=\overline{e}_{\alpha}$, for any $i\in \{ 1, \ldots , m,n\}$ and $\alpha \in \{ m\!+\!1, \ldots , 2m\}$. For case (a), each parenthesis on the left hand side of (\ref{eq:splitintegrable2}) is zero, so the equality holds. For case (b), (\ref{eq:splitintegrable2}) becomes
\begin{equation} \label{eq:splitintegrable3}
2[e_i, \overline{e}_{\alpha}] + 2\sqrt{-1}\,J [ e_i, \overline{e}_{\alpha}]=0.
 \end{equation}
Since $T(e_i, \overline{e}_{\alpha})=0$, so by (\ref{eq:splittheta}) we have
$$ [e_i, \overline{e}_{\alpha}] = \nabla^c_{e_i}\overline{e}_{\alpha} - \nabla^c_{\overline{e}_{\alpha}}e_i = \sum_{\beta =m+1}^{2m}  \overline{\theta_{\alpha \beta} (\overline{e}_i) } \, \overline{e}_{\beta} -  \sum_{k=1}^{m} \theta_{ik} (\overline{e}_{\alpha}) e_k $$
if $1\leq i\leq m$, while
$$ [e_n, \overline{e}_{\alpha}] =  \sum_{\beta =m+1}^{2m}  \overline{\theta_{\alpha \beta} (\overline{e}_n) } \, \overline{e}_{\beta} , $$
as $\theta_{n\ast}$ are all $(1,0)$-forms.
Therefore $[e_i, \overline{e}_{\alpha}] $ is always a linear combination of $\overline{e}_{\beta}$ and $e_k$, for $1\leq k\leq m$ and $m\!+\!1\leq \beta \leq 2m$. Since $J\overline{e}_{\beta}=-I\overline{e}_{\beta}= \sqrt{-1}\overline{e}_{\beta}$ and $Je_k=Ie_k=\sqrt{-1}e_k$, we have $J[e_i, \overline{e}_{\alpha}] = \sqrt{-1}\,[e_i, \overline{e}_{\alpha}]$, hence (\ref{eq:splitintegrable3}) holds. This completes the proof that $J$ is integrable.

Next let us show that the Bismut connection $\hat{\nabla}^b$ of the Hermitian manifold $(M^n,g,J)$ will coincide with the Bismut connection $\nabla^b$ of the original Vaisman manifold $(M^n,g,I)$. Let $\varepsilon_i=e_i$ and $\psi_i =\varphi_i$ for $1\leq i\leq m$ or $i=n$, while $\varepsilon_{\alpha}=\overline{e}_{\alpha}$ and $\psi_{\alpha}=\overline{\varphi}_{\alpha}$ for $m\!+\!1\leq \alpha \leq 2m$. Then $\varepsilon$ becomes a unitary frame for $(M^n,g,J)$ and $\psi$ is its dual coframe. Again write $\psi'=\,^t\!(\psi_1, \ldots , \psi_m)$ and  $\psi''=\,^t\!(\psi_{m\!+\!1}, \ldots , \psi_{2m})$  for the column vectors. By taking the complex conjugation in the middle portion of (\ref{eq:splitstructure}), we obtain
\begin{equation} \label{eq:splitstructure2}
 d\psi =  \left[ \begin{array}{c} -\,^t\!\theta_1 \wedge \psi'   \\  \theta_2\wedge \psi'' \\ a\,\big( \,^t\!\psi'' \wedge \overline{\psi}'' - \,^t\!\psi' \wedge \overline{\psi}'  \big)  \end{array} \right] =-\,^t\!\hat{\theta}\wedge \psi +\hat{\tau},
\end{equation}
where
$$ \hat{\theta} = \left[ \begin{array}{ccc}  \theta_1 & 0 & -a\overline{\psi}'   \\
0 & \overline{\theta}_2 &  a\overline{\psi}'' \\ a\,^t\!\psi'  & -a\,^t\!\psi''  & 0 \end{array} \right]  ,
\ \ \ \ \ \ \ \ \hat{\tau}=a\left[ \begin{array}{c}  \psi'   \\   -\psi'' \\ 0  \end{array} \right] \wedge \psi_n.$$
Since $\hat{\theta}$ is skew-Hermitian, and the entries of $\hat{\tau}$ are all $(2,0)$-forms in $(M^n,g,J)$, we know that $\hat{\theta}$, $\hat{\tau}$ are respectively the Chern connection matrix and torsion column vector under $\varepsilon$. In particular, the only non-zero torsion components of $(M^n,g,J)$ under $\varepsilon$ are $T^i_{in}=a=- T^{\alpha}_{\alpha n}$ for each $1\leq i\leq m$ and each $m\!+\!1\leq \alpha \leq 2m$.  From this, we deduce that the components for Gauduchon's torsion $1$-form are $\hat{\eta}_j = \sum_{k=1}^n T^k_{kj}=0$ for each $j$, hence $\hat{\eta}=0$ and $(M^n,g,J)$ is balanced. Also, the $\gamma$ tensor for $(M^n,g,J)$ has matrix representation
$$ \hat{\gamma} = \hat{\theta}^b -\hat{\theta} =  a\! \left[ \begin{array}{ccc}  (\psi_n-\overline{\psi}_n)I & 0 & \overline{\psi}'   \\  0 & - (\psi_n-\overline{\psi}_n)I  &  -\overline{\psi}'' \\ -\,^t\!\psi'  & \,^t\!\psi''  & 0 \end{array} \right]. $$
Recall that $\theta_1=\theta_1^b-\alpha I$, $\theta_2=\theta_2^b-\alpha I$, where $\alpha =a(\varphi_n-\overline{\varphi}_n)=a(\psi_n-\overline{\psi}_n)$, so from the above equality we get
$$  \hat{\theta}^b = \hat{\theta} + \hat{\gamma} =   \left[ \begin{array}{ccc}  \theta_1^b &  &   \\   &  \overline{\theta}_2^b  &   \\  &  & 0 \end{array} \right]. $$
Hence $\hat{\nabla}^be' = \hat{\nabla}^b \varepsilon' = \hat{\theta}^b_1 \varepsilon' = \theta^b_1\varepsilon ' = \theta^b_1 e' = \nabla^b e'$, $\hat{\nabla}^be_n = \hat{\nabla}^b\varepsilon_n =0 = \nabla^b e_n$, and
$$ \hat{\nabla}^be'' = \hat{\nabla}^b \overline{\varepsilon''} = \overline{ \hat{\theta}^b_2 \varepsilon''}  = \overline{ \overline{ \theta^b_2 } \varepsilon''} =  \theta^b_2 e'' =  \nabla^b e''. $$
Therefore, $\hat{\nabla}^b=\nabla^b$, namely, Hermitian manifolds $(M^n,g,I)$ and $(M^n,g,J)$ share the same Bismut connection, thus the latter is a balanced  BTP manifold. This completes the proof of the theorem.
\end{proof}

\vspace{0.4cm}

\section{Appendix}

In the appendix, we will calculate the Riemannian curvature of the Wallach threefold $(X^3,g)$ discussed in \S \ref{WCH3D}, which shows that the Riemannian sectional curvature is non-negative and the Ricci curvature is constantly equal to $3$.

First let us recall some general formulae from existing literature. Let $e$ be a local unitary frame on a Hermitian manifold $(M^n,g)$. We have
$$ \nabla^c_{\ell}T^j_{ik} - \nabla^b_{\ell}T^j_{ik} = \sum_r \big(T^r_{\ell i}T^j_{kr} + T^r_{k\ell}T^j_{ir} - T^r_{ik}T^j_{r\ell} \big). $$
By \cite[Proposition 2.5]{ZhaoZ24}, we know that $\nabla^b_{1,0}T=0$ implies that the right hand side of the equality above is zero. Therefore  $\nabla^b_{1,0}T=0 $ indicates $ \nabla^c_{1,0}T=0$. It implies if $g$ is BTP, then the $(1,0)$-part of Chern covariant differentiation of the torsion vanishes, that is, $T^j_{ik;\ell}=0$, where the index after the semicolon stands for covariant derivative with respect to the Chern connection. The $(0,1)$-part of Chern covariant differentiation of torsion, on the other hand, is given by
$$ T^j_{ik;\overline{\ell}}= R^c_{k\overline{\ell}i\overline{j}} - R^c_{i\overline{\ell}k\overline{j}} .$$
Denote by $R$ the Riemannian curvature tensor, namely, the curvature of the Levi-Civita connection of $g$.
Note that the symbol $T_{ik}^j$ defined at the beginning of \S \ref{BS3D} is two times of that in \cite{YZ18Cur}.
So by \cite[Lemma 7]{YZ18Cur} and the equality above, we obtain
\begin{eqnarray}
R_{ijk\overline{\ell}} & = & \frac{1}{2}T^{\ell}_{ij;k}
+ \frac{1}{4}\sum_r \big( T^{\ell}_{ri}T^r_{jk} - T^{\ell}_{rj}T^r_{ik} \big)    \label{eq:8.26} \\
R_{k\overline{\ell}i\overline{j}} & = & \frac{1}{2}\big( R^c_{i\overline{\ell}k\overline{j}}  +  R^c_{k\overline{j}i\overline{\ell}}  \big)
+ \frac{1}{4}\sum_r \big( T^r_{ik}\overline{T^r_{j\ell } } - T^j_{kr}\overline{T^i_{\ell r} } - T^{\ell}_{ir}\overline{T^k_{jr } } \big) \label{eq:8.27}
\end{eqnarray}
Then let us specialize to the Wallach threefold $(X,g)$ at the origin $0$. Recall from \S \ref{WCH3D} that $g_{i\overline{j}}=\delta_{ij}$ at the origin and all the components of $T$ vanish except $T^2_{13}=1$, so the right hand side of (\ref{eq:8.26}) is zero, hence $R_{ijk\overline{\ell}} =0$. By the formulae on $R^c_{i\overline{j}k\overline{\ell}}$ obtained in \S \ref{WCH3D}, we get
\begin{eqnarray}
&& R_{i\overline{j}k\overline{\ell}}=0,\ \text{if}\ \{ i,k\}\neq \{ j,\ell\}, \\
&& R_{i\overline{i}i\overline{i}}=2, \\
&& R_{1\overline{1}2\overline{2}} = R_{3\overline{3}2\overline{2}} = - R_{1\overline{1}3\overline{3}}  = \frac{3}{4}, \ \
R_{1\overline{2}2\overline{1}} = R_{3\overline{2}2\overline{3}} = \frac{1}{2}, \ \
R_{1\overline{3}3\overline{1}} = - \frac{1}{4},\label{eq:8.28}
\end{eqnarray}
where we recall that $R_{i\overline{j}k\overline{\ell}}=R_{k\overline{\ell}i\overline{j}}$ always holds for the Riemannian curvature $R$.
Now we compute the sectional curvature of $R$. Let $x$, $y$ be any two real tangent vector of $X$ at the origin $0$ satisfying $x\wedge y\neq 0$. Write $x=X+\overline{X}$ and $y=Y+\overline{Y}$ for type $(1,0)$ tangent vectors $X$ and $Y$.  We have
$$ x\wedge y = X\wedge Y + \overline{X} \wedge \overline{Y} + \big( X\wedge \overline{Y} - Y \wedge \overline{X} \big). $$
By Gray's theorem, $R_{XYZW}=0$ for any type $(1,0)$ tangent vectors $X$, $Y$, $Z$, $W$. Also, we have shown above $R_{XYZ\overline{W}}=0$
for the Wallach threefold $(X,g)$. By the first Bianchi identity, we have
$$ - R_{XY\overline{X}\overline{Y}} = - R_{X\overline{X}Y\overline{Y}} + R_{Y\overline{X}X\overline{Y}}. $$
It implies
\begin{eqnarray*}
 R_{xyyx} & = &  - R (x\wedge y, x\wedge y) \ \, = \ \,  -2 R(X\wedge Y, \overline{X} \wedge \overline{Y} ) - R( X\wedge \overline{Y} - Y \wedge \overline{X} , X\wedge \overline{Y} - Y \wedge \overline{X}) \\
 & = & -2 R_{X\overline{X}Y\overline{Y}} + 2 R_{X\overline{Y}Y\overline{X}} - R( X\wedge \overline{Y} - Y \wedge \overline{X} , X\wedge \overline{Y} - Y \wedge \overline{X}) \\
 & = & -2 R_{X\overline{X}Y\overline{Y}} + 4 R_{X\overline{Y}Y\overline{X}} - 2\mbox{Re} \{ R_{X\overline{Y}X\overline{Y}} \}.
\end{eqnarray*}
For $i\neq k$, let us write $R_{i\overline{i}k\overline{k}}=a_{ik}$ and $R_{i\overline{k}k\overline{i}}=b_{ik}$.
Then we have
\begin{equation} \label{eq:8.29}
2 b_{ik} - a_{ik}= \frac{1}{4}, \ \ \  \ 2a_{ik}- b_{ik} = \left\{  \begin{array}{ll} -\frac{5}{4}, \ \mbox{if} \  \{i,k\} =\{ 1,3\}  \\ \ 1, \ \ \ \mbox{otherwise} \end{array}\!\!,\right.\ \ \ \
a_{ik}+ b_{ik} = \left\{  \begin{array}{ll} -1, \ \mbox{if} \  \{i,k\} =\{ 1,3\}  \\ \ \frac{5}{4}, \ \ \ \mbox{otherwise} \end{array}\!\!. \right.
\end{equation}
So we get
\begin{eqnarray*}
 R_{xyyx} & = &  \sum_{i,j,k,\ell} R_{i\overline{j}k\overline{\ell}} \,\{ -2X_i\overline{X}_j Y_k \overline{Y}_{\ell} + 4 X_i\overline{Y}_j Y_k \overline{X}_{\ell} -  2\mbox{Re} (X_i\overline{Y}_j X_k \overline{Y}_{\ell}) \}  \\
 & = & \sum_i 2 \{ 2|X_iY_i|^2 - 2\mbox{Re} (X_i^2\overline{Y}_i^2)\} + \sum_{i\neq k} a_{ik} \{-2|X_iY_k|^2 + 4 X_i\overline{Y}_iY_k\overline{X}_k  - 2\mbox{Re} (X_i\overline{Y}_iX_k\overline{Y}_k ) \} + \\
 & & + \sum_{i\neq k} b_{ik} \{ -2  X_i\overline{X}_kY_k\overline{Y}_i + 4 |X_iY_k|^2 - 2\mbox{Re} (X_i\overline{Y}_kX_k\overline{Y}_i ) \} \\
 & = & 4\sum_i \{ |X_iY_i|^2 - \mbox{Re} (X_i^2\overline{Y}_i^2)\} +  2\sum_{i<k} \{ (2b_{ik}-a_{ik})(|X_iY_k|^2 + |X_kY_i|^2) \} + \\
&&  + \, 2\sum_{i<k} \{ (2a_{ik}-b_{ik}) 2\mbox{Re} (X_i\overline{X}_k \overline{Y}_i Y_k ) \} -2\sum_{i<k} \{ (a_{ik}+b_{ik}) 2\mbox{Re} (X_iX_k\overline{Y}_i \overline{Y}_k )  \} \\
& = & 2\sum_{i<k} F_{ik},
\end{eqnarray*}
where
\begin{eqnarray*}
F_{ik} & = & |X_iY_i|^2 - \mbox{Re} (X_i^2\overline{Y}_i^2) + |X_kY_k|^2 - \mbox{Re} (X_k^2\overline{Y}_k^2)
+ \frac{1}{4}(|X_iY_k|^2 + |X_kY_i|^2) + \\
&& + \, 2(2a_{ik}-b_{ik}) \mbox{Re} (X_i\overline{X}_k \overline{Y}_i Y_k ) -  2(a_{ik}+b_{ik}) \mbox{Re} (X_iX_k\overline{Y}_i \overline{Y}_k ).
\end{eqnarray*}
For $(ik)=(13)$, $2a_{ik}-b_{ik}=-\frac{5}{4}$ and $a_{ik}+b_{ik}=-1$, so we have
\begin{eqnarray*}
F_{ik} & = & |X_iY_i|^2 - \mbox{Re} (X_i^2\overline{Y}_i^2) + |X_kY_k|^2 - \mbox{Re} (X_k^2\overline{Y}_k^2)
+ \frac{1}{4}(|X_iY_k|^2 + |X_kY_i|^2) + \\
&& - \, \frac{5}{2} \mbox{Re} (X_i\overline{X}_k \overline{Y}_i Y_k ) + 2\mbox{Re} (X_iX_k\overline{Y}_i \overline{Y}_k )\\
& = & 2\{\mbox{Im}(X_i\overline{Y}_i)\}^2 + 2\{\mbox{Im}(X_k\overline{Y}_k)\}^2 + \frac{1}{4} |X_iY_k - X_kY_i |^2 -4 \mbox{Im}(X_i\overline{Y}_i)\mbox{Im}(X_k\overline{Y}_k) \\
& = & 2 \{ \mbox{Im}(X_i\overline{Y}_i) - \mbox{Im}(X_k\overline{Y}_k) \}^2 + \frac{1}{4} |X_iY_k - X_kY_i |^2 \ \geq \ 0.
\end{eqnarray*}
Similarly, for $(ik)=(12)$ or $(23)$, $2a_{ik}-b_{ik}=1$, $a_{ik}+b_{ik}=\frac{5}{4}$, so
\begin{eqnarray*}
F_{ik} & = & |X_iY_i|^2 - \mbox{Re} (X_i^2\overline{Y}_i^2) + |X_kY_k|^2 - \mbox{Re} (X_k^2\overline{Y}_k^2)
+ \frac{1}{4}(|X_iY_k|^2 + |X_kY_i|^2) + \\
&& + \,  2\mbox{Re} (X_i\overline{X}_k \overline{Y}_i Y_k ) - \frac{5}{2} \mbox{Re} (X_iX_k\overline{Y}_i \overline{Y}_k )\\
& = & 2\{\mbox{Im}(X_i\overline{Y}_i)\}^2 + 2\{\mbox{Im}(X_k\overline{Y}_k)\}^2 + \frac{1}{4} |X_i\overline{Y}_k - Y_i\overline{X}_k |^2 + 4 \mbox{Im}(X_i\overline{Y}_i)\mbox{Im}(X_k\overline{Y}_k) \\
& = & 2 \{ \mbox{Im}(X_i\overline{Y}_i) + \mbox{Im}(X_k\overline{Y}_k) \}^2 + \frac{1}{4} |X_i\overline{Y}_k - Y_i\overline{X}_k |^2 \ \geq \ 0.
\end{eqnarray*}
Hence, $R_{xyyx}$ is equal to
\begin{equation*}
4(I_1+I_2)^2+4(I_2+I_3)^2 + 4(I_1-I_3)^2 +\frac{1}{2}|X_1\overline{Y}_2-Y_1\overline{X}_2|^2 +\frac{1}{2}|X_2\overline{Y}_3-Y_2\overline{X}_3|^2 +\frac{1}{2}|X_1Y_3-Y_1X_3|^2,
\end{equation*}
where $I_i=\mbox{Im}(X_i\overline{Y}_i)$. Therefore the metric $g$ has  non-negative Riemannian sectional curvature. Note that the Riemannian sectional curvature is not strictly positive here. If we take $X_1=X_2=X_3\in {\mathbb R}\setminus \{ 0\}$, and $Y_1=\overline{Y}_2=Y_3=\rho \not\in {\mathbb R}$, then $I_1=I_3=-I_2$ and the expression above vanishes, so we get $R_{xyyx}=0$ with $x\wedge y\neq 0$.

To see the Ricci curvature of $g$, let $y_i=e_i+\overline{e}_i$, then the formula above becomes
$$ R_{x y_i y_i x} = 4|X_i|^2 - 4\mbox{Re}(X_i^2) +  |X_j|^2 +  |X_k|^2,$$
where $\{ i,j,k\} =\{ 1,2,3\}$. Similarly, if we let $y_{i^{\ast}}=\sqrt{-1}e_i-\sqrt{-1}\overline{e}_i$, then we get
$$ R_{xy_{i^{\ast}} y_{i^{\ast}} x} = 4|X_i|^2 + 4\mbox{Re}(X_i^2) +  |X_j|^2 +  |X_k|^2.$$
Add up the two equalities above for $i$ from $1$ to $3$ and we get $12|X|^2=6|x|^2$. Let $\varepsilon_i = \frac{1}{\sqrt{2}}(e_i+\overline{e}_i)$ and  $\varepsilon_{i^{\ast}} = \frac{\sqrt{-1}}{\sqrt{2}}(e_i-\overline{e}_i)$. Then $\{ \varepsilon_i, \varepsilon_{i^{\ast}}\}$ form an orthonormal tangent frame, so the Ricci curvature of the Riemannian metric $g$ is
\begin{equation}
\mathtt{Ric}(x) = \frac{1}{|x|^2} \sum_i \big( R_{x\varepsilon_i\varepsilon_ix} + R_{x\varepsilon_{i^{\ast}}\varepsilon_{i^{\ast}}x} \big)= \frac{1}{2|x|^2} 6|x|^2 = 3.
\end{equation}
That is, $g$ is an Einstein metric on $X$ with Ricci curvature $3$.

\vspace{0.2cm}

%\textbf{Acknowledgement}: We would like to thank mathematicians Gabriel Khan, Lei Ni, Bo Yang,

\end{document}